\documentclass[reqno, 11pt, letterpaper, oneside]{article}

\usepackage{amsmath,amsthm,amssymb}
\usepackage{bbm}
\usepackage{geometry}
\usepackage[final]{graphicx}
\usepackage{setspace}

\newtheorem{theorem}{Theorem}[section]
\newtheorem{lemma}[theorem]{Lemma}
\newtheorem{corollary}[theorem]{Corollary}

\newcommand{\cB}{{\mathcal B}}
\newcommand{\cC}{{\mathcal C}}
\newcommand{\cD}{{\mathcal D}}
\newcommand{\cE}{{\mathcal E}}
\newcommand{\cF}{{\mathcal F}}
\newcommand{\cG}{{\mathcal G}}
\newcommand{\cJ}{{\mathcal J}}
\newcommand{\cH}{{\mathcal H}}
\newcommand{\cI}{{\mathcal I}}
\newcommand{\cL}{{\mathcal L}}
\newcommand{\cN}{{\mathcal N}}
\newcommand{\cM}{{\mathcal M}}
\newcommand{\cQ}{{\mathcal Q}}
\newcommand{\cS}{{\mathcal S}}
\newcommand{\cT}{{\mathcal T}}
\newcommand{\cV}{{\mathcal V}}

\newcommand{\EE}{{\mathbb E}}
\newcommand{\PP}{{\mathbb P}}
\newcommand{\indic}[1]{\mathbbm{1}_{\{#1\}}}

\newcommand{\badE}[1][i]{\cB_{{#1}}(\Sigma)}
\newcommand{\badEC}[1][i]{\cB_{{#1}}(\Sigma^*)}
\newcommand{\badEL}[1][i]{\cB_{\leq {#1}}(\Sigma)}

\newcommand{\opens}[1][i]{X_{\Sigma}({#1})}
\newcommand{\interms}[1][i]{Y_{\Sigma}({#1})}
\newcommand{\parts}[1][i]{Z_{\Sigma}({#1})}
\newcommand{\partOs}[1][i]{T_{\Sigma}({#1})}
\newcommand{\partCs}[1][i]{Z_{\Sigma^*}({#1})}
\newcommand{\partOCs}[1][i]{T_{\Sigma^*}({#1})}

\renewcommand{\epsilon}{\varepsilon}

\begin{document}

\begin{center}

{\LARGE When does the $K_4$-free process stop?}

\vspace{3mm}

{\large Lutz Warnke} 
\vspace{1mm}

{ Mathematical Institute,  University of Oxford\\
 24--29 St.~Giles', Oxford OX1 3LB, UK\\
 {\small\tt warnke@maths.ox.ac.uk}}

\vspace{6mm}

\small
\begin{minipage}{0.8\linewidth}
\textsc{Abstract.} 
The $K_4$-free process starts with the empty graph on $n$ vertices and at each step adds a new edge chosen uniformly at random from all remaining edges that do not complete a copy of $K_4$. 
Let $G$ be the random maximal $K_4$-free graph obtained at the end of the process. 
We show that for some positive constant $C$, with high probability as $n \to \infty$, the maximum degree in $G$ is at most $C n^{3/5}\sqrt[5]{\log n}$. 
This resolves a conjecture of Bohman and Keevash for the $K_4$-free process and improves on previous bounds obtained by Bollob{\'a}s and Riordan and by Osthus and Taraz. 
Combined with results of Bohman and Keevash this shows that with high probability $G$ has $\Theta(n^{8/5}\sqrt[5]{\log n})$ edges and is `nearly regular', i.e., every vertex has degree $\Theta(n^{3/5}\sqrt[5]{\log n})$.  
This answers a question of Erd{\H o}s, Suen and Winkler for the $K_4$-free process. 
We furthermore deduce an additional structural property: we show that whp the independence number of $G$ is at least $\Omega(n^{2/5}(\log n)^{4/5}/\log \log n)$, which matches an upper bound obtained by Bohman up to a factor of $\Theta(\log \log n)$.  
Our analysis of the $K_4$-free process also yields a new result in Ramsey theory: for a special case of a well-studied function introduced by Erd{\H o}s and Rogers we slightly improve the best known upper bound. 
\end{minipage}
\end{center} 
\normalsize

\section{Introduction} 
We consider the $K_4$-free process. 
This is the random sequence of graphs defined by starting with an empty graph on $n$ vertices and then in each step adding a new edge chosen uniformly at random from all remaining edges that do not complete a copy of $K_4$.  
The process terminates with a maximal $K_4$-free graph on $n$ vertices, and we are interested in the typical structural properties of the resulting graph as $n$ tends to infinity. 
In this paper our main focus is on the final number of edges as well as the degree of each vertex. 
As usual, we say that an event holds \emph{with high probability}, or \emph{whp}, if it holds with probability $1-o(1)$ as $n\to\infty$.

In the $H$-free process one forbids the appearance of a copy of some fixed graph $H$, instead of forbidding a $K_4$. 
This process was suggested  by Bollob{\'a}s and Erd{\H o}s~\cite{Bollobas2010PC} at a conference in 1990, as a way to generate a `natural' probability distribution on the set of maximal $H$-free graphs. 
It was first described in print in 1995 by Erd{\H o}s, Suen and Winkler~\cite{ErdoesSuenWinkler1995}, who asked how many edges the final graph typically has.  
Even earlier results of Ruci{\'n}ski and Wormald~\cite{RucinskiWormald1992} imply that for $H=K_{1,d+1}$, where $d>0$ is fixed, the $K_{1,d+1}$-free process whp ends with $\lfloor nd/2\rfloor$ edges. 
Two other special cases were examined by several researchers: 
$H=K_3$ was first studied by Erd{\H o}s, Suen and Winkler~\cite{ErdoesSuenWinkler1995}, and $H=K_4$ by Bollob{\'a}s and Riordan~\cite{BollobasRiordan2000} and by Osthus and Taraz~\cite{OsthusTaraz2001}. 
Each of them determined the final number of edges up to logarithmic factors. 
Only in a recent breakthrough was Bohman~\cite{Bohman2009K3} able to close the gap for the $K_3$-free process by showing that it whp ends with $\Theta(n^{3/2} \sqrt{\log n})$ edges, thereby proving a conjecture of Spencer~\cite{Spencer1995}. 
He went on to analyse the $K_4$-free process, but, despite his substantial improvements, did not obtain matching lower and upper bounds for the final number of edges.

The general $H$-free process was first considered independently by Bollob{\'a}s and Riordan~\cite{BollobasRiordan2000} and by Osthus and Taraz~\cite{OsthusTaraz2001}.  
For $H$ that satisfy a certain density condition (strictly $2$-balanced), Osthus and Taraz determined the typical number of edges in the final graph of the $H$-free process up to logarithmic factors. 
Under the additional assumption that $H$ is regular, Wolfovitz~\cite{Wolfovitz2009H} later slightly improved the lower bound (for the expected final number of edges). 
Recently Bohman and Keevash~\cite{BohmanKeevash2010H} proved that if $H$ is strictly $2$-balanced, then for some $C>0$  the graph generated by the $H$-free process has whp at least $Cn^{2-(v_H-2)/(e_H-1)} (\log n)^{1/(e_H-1)}$ edges, which they conjectured to be tight up to the constant (in fact, they conjectured that the maximum degree is at most $C'n^{1-(v_H-2)/(e_H-1)} (\log n)^{1/(e_H-1)}$ for some $C' > 0$). 
As one can see, the typical number of edges in the final graph of the $H$-free process has attracted a lot of attention, and for a large class of \mbox{graphs $H$} (including cliques $K_{\ell}$ and cycles $C_{\ell}$ of arbitrary fixed size) interesting bounds are known. 
On the other hand, only for the special cases $H=K_3$ and $H=K_{1,d+1}$ has the question of Erd{\H o}s, Suen and Winkler~\cite{ChungGraham98} been answered so far, i.e., the exact order of magnitude been determined. 
It is an  intriguing problem to develop new upper bounds on the number of steps of the $H$-free process.

The analysis of the $H$-free process has also produced new results for certain Ramsey and Tur{\'a}n type problems, which are two central topics in extremal combinatorics. 
In~\cite{BohmanKeevash2010H,Wolfovitz2009H} new lower bounds for the Tur{\'a}n numbers of certain bipartite graphs were obtained, e.g.\ for $H=K_{r,r}$ with $r \geq 5$. 
It is notable that Bohman's analysis of the $K_3$-free process~\cite{Bohman2009K3} gives a lower bound for the Ramsey number $R(3,t)$, which matches (up to constants) the celebrated result of Kim~\cite{Kim95}. 
The subsequent analysis of the $H$-free process~\cite{Bohman2009K3,BohmanKeevash2010H} has e.g.\ also improved the best known lower bounds for the off-diagonal Ramsey numbers $R(s,t)$ by some logarithmic factor for $s \geq 4$. 
One of the main ingredients for the above Ramsey results is an upper bound on the independence number of the $H$-free process for a certain class of graphs $H$. 
Although it is not mentioned in~\cite{BohmanKeevash2010H}, combining an easy consequence of Tur{\'a}n's theorem~\cite{Turan41} with the results of Osthus and Taraz~\cite{OsthusTaraz2001} gives, up to logarithmic factors, matching lower bounds for the independence number. 
So far, only for the special case $H=K_3$ is the bound obtained in~\cite{Bohman2009K3} known to be best possible (up to constants) for the $H$-free process, and it would be interesting to reduce this gap for other graphs.

Nowadays, the $H$-free process is also studied as a model of independent interest.  
For $H$ satisfying a certain density condition (strictly $2$-balanced), the early evolution of various graph parameters, including the degree and the number of small subgraphs, has been investigated in~\cite{BohmanKeevash2010H,Wolfovitz2009K3Cycles,Wolfovitz2009K3Subgraph}. 
These results suggest that, perhaps surprisingly, during this initial phase the graph produced by the $H$-free process is very similar to the uniform random graph with the same number of edges. 
The behaviour of the $H$-free process in later steps is not well understood, and so far only some preliminary results~\cite{GerkeMakai2010K3,Warnke2010H} are known. For example, in~\cite{Warnke2010H} it was shown that whp very dense subgraphs never appear in the $H$-free process. 
This motivates the continued investigation of certain structural properties, e.g.\ the degree of each vertex, in the later evolution of the $H$-free process.

\subsection{Main result}
In this paper we resolve the conjecture of Bohman and Keevash~\cite{BohmanKeevash2010H} for the $K_4$-free process: we prove that whp the maximum degree is indeed $O(n^{3/5}\sqrt[5]{\log n})$. 
\begin{theorem}%
\label{thm:main_result}
There exists $C>0$ such that with high probability the maximum degree in the graph generated by the $K_4$-free process is at most $C n^{3/5}\sqrt[5]{\log n}$. 
\end{theorem}
This improves the upper bounds by Bollob{\'a}s and Riordan~\cite{BollobasRiordan2000} and by Osthus and Taraz~\cite{OsthusTaraz2001}, who proved that the maximum degree is whp at most $O(n^{3/5} \log n)$ and $O(n^{3/5} \sqrt{\log n})$, respectively. 
In fact, up to the constant our upper bound is best possible, since the results of Bohman and Keevash~\cite{BohmanKeevash2010H} imply that for some $c > 0$, whp the  minimum degree is at least $cn^{3/5} \sqrt[5]{\log n}$. 
Putting things together, this shows that the $K_4$-free process produces whp a `nearly regular' graph, i.e., one in which every vertex has degree $\Theta(n^{3/5}\sqrt[5]{\log n})$. 
In particular, this answers a question of Erd{\H o}s, Suen and Winkler for the $K_4$-free process (see e.g.~\cite{ChungGraham98}): whp the final graph has $\Theta(n^{8/5}\sqrt[5]{\log n})$ edges.

Furthermore, we obtain a new lower bound on the independence number of the $H$-free process for the special case $H=K_4$. 
To this end we use a result of Shearer~\cite{Shearer1995}, which states that for $s\geq 4$, every $K_s$-free graph on $n$ vertices with maximum degree $d$ contains an independent set of size at least $cn \frac{\log d}{d\log\log d}$ for $d$ large enough, where $c=c(s)$ is a constant. 
\begin{corollary}%
\label{cor:main_result:independence set}
There exists $c>0$ such that with high probability the independence number of the graph generated by the $K_4$-free process is at least $c n^{2/5}(\log n)^{4/5}/\log \log n$. 
\end{corollary}
Up to the $\Theta(\log \log n)$ factor our lower bound is best possible, since Bohman~\cite{Bohman2009K3} showed that for some $C>0$, whp the independence number is at most $C n^{2/5}(\log n)^{4/5}$.

Our analysis of the $K_4$-free process also produces a new result in Ramsey theory. 
Given integers $2 \leq r < s < n$, let $f_r(G)$ denote the maximum cardinality of a subset of vertices of $G$ that contains no copy of $K_r$, and define $f_{r,s}(n) := \min f_r(G)$, where the minimum is taken over all $K_s$-free graphs on $n$ vertices. 
This function was introduced in 1962 by Erd{\H o}s and Rogers~\cite{ErdoesRogers62}, and further examined by Bollob{\'a}s and Hind~\cite{BollobasHind1991}, Krivelevich~\cite{Krivelevich1995,Krivelevich1994}, Sudakov~\cite{Sudakov2005RSA,Sudakov2005Comb} and Dudek and R{\"o}dl~\cite{DudekRoedl2010}. 
For more details we refer to the recent survey~\cite{DudekRoedl2010Survey}; here we just remark that  the problem of determining $f_{r,s}(n)$ extends that of determining Ramsey numbers. 
As we shall see, our proof of Theorem~\ref{thm:main_result} gives the following new estimate for the special case  $f_{3,4}(n)$. 
\begin{theorem}%
\label{thm:main_result:ramsey}
There exists $C>0$ such that $f_{3,4}(n) \leq C n^{3/5}\sqrt[5]{\log n}$ for every $n \geq 2$. 
\end{theorem} 
This is a slight improvement on the previously best upper bound, $f_{3,4}(n) = O(n^{3/5}\sqrt{\log n})$, which was established by Krivelevich~\cite{Krivelevich1995} in 1995 by carefully deleting edges from the binomial random graph $G_{n,p}$, where the edge density $p$ is suitably chosen.

\subsection{Techniques} 
To prove Theorem~\ref{thm:main_result} there are several difficulties we need to overcome. 
First of all, the results of Bohman~\cite{Bohman2009K3} as well as Bohman and Keevash~\cite{BohmanKeevash2010H} only allow us to `control' the $K_4$-free process during the initial $m$ steps, where $m$ is $n^{8/5}\sqrt[5]{\log n}$ times some small constant; the behaviour in later steps is so far not well understood. 
To overcome this issue, we prove that already after the first $m$ steps whp every large set of vertices contains a triangle. 
Because the neighbourhood of every vertex has to be triangle free, this indeed gives an upper bound on the maximum degree in the final graph of the $K_4$-free process. 
For the binomial random graph $G_{n,p}$ such results can be derived routinely, e.g.\ using Janson's inequality~\cite{Janson1990} together with a union bound argument. 
But in the $K_4$-free process there is a complicated dependency among the edges, which makes the use of standard tools difficult (Bollob{\'a}s and Riordan~\cite{BollobasRiordan2000} and Osthus and Taraz~\cite{OsthusTaraz2001} apply correlation inequalities and concentration results in a sophisticated way, but at the cost of obtaining asymptotically suboptimal bounds on the maximum degree). 
To overcome this main technical challenge, we introduce a variant of the differential equation method~\cite{BohmanKeevash2010H,Wormald1995DEM,Wormald1999DEM} which might be of independent interest. 
Roughly speaking, in the $K_4$-free process it allows us to `control' certain subgraph counts in \emph{every} large subset of the vertices.  
As usual for applications of the differential equation method, we need to `control' the step-by-step changes of our random variables, which introduces additional technical difficulties. 
To solve one of these issues we use a tool that may also be of independent interest: we essentially show that a slightly weaker version of the so-called Deletion Lemma by R{\"o}dl and Ruci\'nski~\cite{JansonRucinski2004Deletion,RoedlRucinski1995} 
applies to the $K_4$-free process (and to the more general $H$-free process considered in~\cite{BohmanKeevash2010H}).

It is important to note that our method is not a simple refinement of Bohman's argument~\cite{Bohman2009K3} for bounding the maximum degree in the $K_3$-free process or the independence number of the $K_4$-free process. 
Whereas Bohman shows that every very large subset of the vertices contains at least one edge, we study the combinatorial structure much more precisely in order to `control' various subgraph counts in every such large set. 
Our ideas also yield a substantially improved version of the variant of the Differential Equation Method proposed by Bohman and Keevash in~\cite{BohmanKeevash2010H}, which we believe is easier to apply in new contexts. 
For the sake of simplicity and clarity of presentation, we have made no attempt to optimize the constants obtained in our proof, and we also omit floor and ceiling signs whenever these are not crucial.

\subsection{Organization of the paper}
In the next section we introduce some notation and briefly review properties of the $K_4$-free process. 
Section~\ref{sec:bounding_max_degree} is devoted to the proof of Theorems~\ref{thm:main_result} and~\ref{thm:main_result:ramsey}. 
The argument is simple, but relies on a rather involved probabilistic statement, whose proof is deferred to Section~\ref{sec:proof}. 
Next, in Section~\ref{sec:density} we collect various properties of the $K_4$-free process based on density considerations. 
Afterwards, in Section~\ref{sec:dem} we present a variant of the differential equation method which allows us to track several variables in every subset of a certain size. 
We postpone some details of the proof to the appendix, where we also state an improved version of the differential equation method presented in~\cite{BohmanKeevash2010H}. 
In Section~\ref{sec:proof} we give the proof of our main technical result. 
Our argument relies on two combinatorial statements which are proved in Section~\ref{sec:trajectory_verification} and Section~\ref{sec:good_configurations_exist} using the differential equation method and density considerations.

\section{The $K_4$-free process: preliminaries and notation}
In this section we introduce some notation and briefly review properties of the $K_4$-free process needed to prove our main result. 
We closely follow~\cite{BohmanKeevash2010H} and the reader familiar with the results of Bohman and Keevash may wish to skip this section.

\subsection{Terminology and notation}\label{sec:notation}
Let $G(i)$ denote the graph with vertex set $[n]=\{1,\ldots, n\}$ after $i$ steps of the $K_4$-free process. 
Its edge set $E(i)$ contains $i$ edges; we partition the remaining non-edges $\binom{[n]}{2} \setminus E(i)$ into two sets, $O(i)$ and $C(i)$, which we call \emph{open} and \emph{closed} pairs, respectively. 
We say that a pair $uv$ of vertices is \emph{closed} in $G(i)$ if $G(i) \cup \{uv\}$ contains a copy of $K_4$. 
Observe that by definition the $K_4$-free process always chooses the next edge $e_{i+1}$ uniformly at random from $O(i)$. 
In addition, for $uv \in O(i) \cup C(i)$ we write $C_{uv}(i)$ for the set of pairs $xy \in O(i)$ such that adding $uv$ and $xy$ to $G(i)$ creates a copy of $K_4$ containing both $uv$ and $xy$. 
In particular, the pair $uv \in O(i)$ would become closed, i.e., belong to $C(i+1)$, if at step $i+1$ the $K_4$-free process chooses $e_{i+1}$ from $C_{uv}(i)$.

The neighbourhood of a vertex $v$ in $G(i)$ is denoted by $\Gamma_i(v)$, where we usually omit the subscript and just write $\Gamma(v)$ if the corresponding $i$ is clear from the context. 
With a given graph in mind, for two vertex sets $A,B$ we write $e(A,B)$ for the number of edges that have one endpoint in $A$ and the other in $B$, where an edge with both ends in $A \cap B$ is counted once. 
Furthermore, given a set $S$ and an integer $k \geq 0$, we write $\binom{S}{k}$ for the set of all $k$-element subsets of $S$.

For notational convenience we use the symbol $\pm$ in two different ways, following~\cite{Bohman2009K3,BohmanKeevash2010H}. 
First, we denote by $a \pm b$ the interval $\{ a + x b : -1 \leq x \leq 1\}$, where multiple occurrences of $\pm$ are treated independently. 
For brevity we also use the convention that $x=a \pm b$ means $x \in a \pm b$. 
Second, given a label~$i$, expressions containing $\pm_i$ are an abbreviation for two different statements: one with every $\pm_i$ replaced by $+$ and $\mp_i$ by $-$, and the other with every $\pm_i$ replaced by $-$ and $\mp_i$ by $+$. 
For example, $x^{\pm_1\pm_2} = a^{\pm_1} \pm b^{\mp_2}$ is a shorthand for four separate statements, of which one is $x^{+-} = a^{+} \pm b^{+}$. 
As usual, whenever there is no danger of confusion, we omit those labels for brevity.

\subsection{Parameters, functions and constants}\label{sec:K4-parameters}
Following~\cite{BohmanKeevash2010H}, we introduce constants $\epsilon$, $\mu$ and $W$, 
which we fix below, and set 
\begin{equation}
\label{eq:K4-parameters}
p := n^{-2/5} , \qquad t_{\max} := \mu \sqrt[5]{\log n}  \qquad \text{and} \qquad m :=  n^2 p t_{\max}= \mu n^{8/5} \sqrt[5]{\log n}  .
\end{equation}
We analyse the $K_4$-free process for the first $m$ steps. 
For each step $i$ we define $t = t(i) := i/(n^2p)$, where, for the sake of brevity, we simply write $t$ if the corresponding $i$ is clear from the context. 
Note that the edge-density of $G(i)$ is roughly $2tp$. It might be more natural to use a different parametrization (to remove the factor of two), however, as we rely on some previous results of Bohman and Keevash~\cite{BohmanKeevash2010H} we follow their convention. 
Similar as in~\cite{Bohman2009K3,BohmanKeevash2010H} we introduce the functions  
\begin{equation}
\label{eq:K4-functions}
q(t) := e^{-16 t^{5}} \qquad \text{ and } \qquad  f(t) := e^{(t^{5} + t)W}  .
\end{equation}
We now fix the constants for the rest of the paper: 
we choose $W$ sufficiently large and then $\epsilon$ and $\mu$ small enough such that, in addition to the constraints implicit in~\cite{BohmanKeevash2010H} for $H=K_4$, we have 
\begin{equation}
\label{eq:K4-constants:Wepsmu}
W \geq 500  , \qquad \epsilon \leq 1/1000 \qquad \text{ and } \qquad 2W\mu^5 \leq \epsilon  .
\end{equation}
Since the additional constraints in~\cite{BohmanKeevash2010H} only depend on $H=K_4$, we see that $\epsilon$, $\mu$ and $W$ are absolute constants. 
So, for every $0 \leq t \leq t_{\max}$, we readily obtain the following inequalities for $n$ large enough: 
\begin{equation}
\label{eq:K4-functions-estimates}
1 \geq q(t) \geq n^{-\epsilon/2} \qquad \text{ and } \qquad	1 \leq f(t) q(t)^{2} \leq f(t) \leq n^{\epsilon}  .
\end{equation}

\subsection{Results of Bohman and Keevash} 
Using Wormald's differential equation method~\cite{Wormald1995DEM,Wormald1999DEM}, Bohman and Keevash~\cite{BohmanKeevash2010H} track a collection of random variables throughout the \mbox{first $m$ steps} of the $K_4$-free process (in fact, their results hold for the more general $H$-free process, where $H$ satisfies a certain density condition).  
To this end they introduce a `good' event $\cG_i$ for every \mbox{step $i$}, which intuitively ensures that the $K_4$-free process has not terminated up to \mbox{step $i$} and guarantees that certain  random variables are essentially tightly concentrated during the \mbox{first $i$ steps}. 
For our application the key properties of $\cG_i$ are estimates on the number of open pairs as well as bounds for the degree and codegree. 
So, for the reader's convenience we state the results of Bohman and Keevash~\cite{BohmanKeevash2010H} here in a simplified form. 
\begin{theorem}%
\label{thm:BohmanKeevash2010H}%
{\normalfont\cite{BohmanKeevash2010H}} 
Define $m=m(n)$, $p=p(n)$ and $t_{\max}=t_{\max}(n)$ as in~\eqref{eq:K4-parameters}. Set $s_e := n^{1/12-\epsilon}$ and $t=t(i) := i/(n^2p)$.  
Furthermore, define $q(t)$ and $f(t)$ as in~\eqref{eq:K4-functions}.  
Let $\cG_j$ denote the event that for every $0 \leq i \leq j$, in $G(i)$ we have $|O(i)| > 0$, and for all distinct vertices $u,v \in [n]$ we have 
\begin{align}
\label{eq:open-estimate}
|O(i)| &= \left(1 \pm 3f(t)/s_e\right) q(t)n^2/2  ,\\
\label{eq:degree-estimate}
|\Gamma_i(u)| & \leq 3 np t_{\max} \quad \text{ and}\\
\label{eq:codegree-estimate}
|\Gamma_i(u) \cap \Gamma_i(v)| &\leq (\log n) n p^2   .
\end{align}
Let $\cJ_j$ denote the event that for every $0 \leq i \leq j$, for all pairs $uv \in O(i) \cup C(i)$ and all distinct pairs $u'v',u''v'' \in O(i)$ we have 
\begin{align}
\label{eq:closed-estimate} 
|C_{uv}(i)| &= \left(40 t^{4} q(t) \pm 9f(t)/s_e\right) p^{-1}  \quad \text{ and}  \\
\label{eq:closed-intersection-estimate}
|C_{u'v'}(i)  \cap C_{u''v''}(i)| &\leq  n^{-1/6} p^{-1}  .
\end{align}
Then the event $\cG_m \cap \cJ_m$ holds with high probability in the $K_4$-free process. 
\end{theorem}
The definition of the event $\cG_i$ used in~\cite{BohmanKeevash2010H} is more complicated; however, the simpler version given above suffices for our purposes. 
In the following we briefly outline how the previous theorem relates to the results of Bohman and Keevash~\cite{BohmanKeevash2010H}. 
After some simple estimates, the bounds for $|O(i)|$, $|\Gamma_i(u)|$ follow directly from their Theorem~$1.4$ (see also the examples in Section~$2$ of~\cite{BohmanKeevash2010H}). 
Using $p^2n = \omega(1)$, their Theorem~$1.4$ also implies the upper bound on the codegree
(analogous to Corollary~$1.5$ in~\cite{BohmanKeevash2010H}). 
Similarly, the remaining estimates follow from Corollary~$6.2$ and Lemma~$8.4$ in~\cite{BohmanKeevash2010H}. 
(As noted in Section~$1.5$ of~\cite{BohmanKeevash2010H}, their high probability events in fact hold with probability at least $1-n^{-\omega(1)}$. So there is no problem in taking a union bound over all steps $i$ as well as pairs $uv$ and $u'v', u''v''$.) 
Finally, we point out that our definition of $C_{uv}(i)$ is different from that in~\cite{BohmanKeevash2010H}, so there is a factor $2$ difference in the formulas (we use unordered instead of ordered pairs).

\section{Bounding the maximum degree in the $K_4$-free process}
\label{sec:bounding_max_degree}
This section is devoted to the proof of our main result, namely that in the $K_4$-free process the maximum degree is whp at most $O(n^{3/5}\sqrt[5]{\log n})$.  
We first state our main technical result and then show how it implies Theorem~\ref{thm:main_result} and an upper bound on $f_{3,4}(n)$. 
Set 
\begin{equation}
\label{eq:K4-parameters2}
\delta := \frac{1}{7000}  , \qquad  \gamma := \max\left\{\frac{5}{\sqrt{\delta} \mu^{5/2}},150\right\}  \qquad \text{and} \qquad u := \gamma n p t_{\max} = \gamma \mu n^{3/5} \sqrt[5]{\log n} . 
\end{equation}
Recall that an open pair has not yet been added to the graph produced by the $K_4$-free process, but may be added in the next step. 
Intuitively, the following theorem thus states that in the $K_4$-free process every large vertex set $U \subseteq [n]$ is `close' to containing a triangle: it contains many open pairs which would complete a copy of a triangle in $U$ if they were added to the graph generated by the $K_4$-free process. 
\begin{theorem}%
\label{thm:k4_many_edges_close_k3}% 
Define $m=m(n)$ and $p=p(n)$ as in~\eqref{eq:K4-parameters}, and $\delta$ and $u=u(n)$ as in~\eqref{eq:K4-parameters2}. 
Set $t=t(i) := i/(n^2p)$ and define $q(t)$ as in~\eqref{eq:K4-functions}. 
Let $\cT_j$ denote the event that for all $n^2p \leq i \leq j$, in $G(i)$ every set $U \subseteq [n]$ of size $u$ contains at least $\delta u^{3} (tp)^{2} q(t)$ open pairs which would complete a copy of a triangle in $U$ if they were added to $G(i)$. 
Then $\cT_m$ holds with high probability in the $K_4$-free process. 
\end{theorem}
As the proof of this result is rather involved, we defer it to Section~\ref{sec:proof}. 
Let us briefly sketch the main ideas for deducing Theorem~\ref{thm:main_result} from Theorem~\ref{thm:k4_many_edges_close_k3}. 
Observe that in the graph produced by the $K_4$-free process the neighbourhood of every vertex has to be triangle-free. 
In order to bound the maximum degree by $u = C n^{3/5}\sqrt[5]{\log n}$, where $C:=\gamma \mu$, it thus suffices to show that whp every set of $u$ vertices contains a triangle. 
Consider a fixed vertex set $U \subseteq [n]$ of size $u$.  
Intuitively, Theorem~\ref{thm:k4_many_edges_close_k3} implies that (after some initial steps) each step creates with reasonable probability a triangle in $U$. 
This suggests that with very high probability $U$ indeed contains a triangle after the first $m$ steps, which essentially suffices to complete the proof (using a union bound argument).  
\begin{proof}[Proof of Theorem~\ref{thm:main_result}]
As mentioned above, we prove the theorem with $C:=\gamma \mu$. Observe that indeed $u = C n^{3/5}\sqrt[5]{\log n}$. 
Given $U \subseteq [n]$ and  $i \leq m$, 
let $\cE_{U,i}$ denote the event that up to step $i$, the set $U$ is triangle-free in the $K_4$-free process. 
In addition, let $\cE_m$ denote the event that there exists a vertex set $U \subseteq [n]$ of size $u$ for which $\cE_{U,m}$ holds.  
Furthermore, for every $i \leq m$ we define the event $\cH_i := \cG_i \cap \cT_i$. 
Note that $\cH_i$ depends only on the first $i$ steps of the $K_4$-free process and furthermore that $\cH_{i+1}$ implies $\cH_i$. 
Now, to complete the proof of the theorem it suffices to show 
\begin{equation}
\label{eq:prob_triangle_free_small}
\begin{split}
\PP[\cE_m \cap \cH_m] = o(1)  .
\end{split}
\end{equation}
Indeed, by Theorems~\ref{thm:BohmanKeevash2010H} and~\ref{thm:k4_many_edges_close_k3} the event $\cH_m$ holds with high probability and thus \eqref{eq:prob_triangle_free_small} implies $\PP[\cE_m] = o(1)$. 
So, with high probability, every set of $u$ vertices contains a triangle and thus the maximum degree in the $K_4$-free process is bounded by $u = C n^{3/5}\sqrt[5]{\log n}$.

In the following we prove \eqref{eq:prob_triangle_free_small} using a union bound argument.
Fix $U \subseteq [n]$ with $|U|=u$, and let $T_U(i) \subseteq O(i)$ denote the open pairs after $i$ steps which would complete at least one copy of a triangle in $U$ if they were added to $G(i)$. 
Then 
\begin{equation}
\label{eq:prob_no_triangle_closed_0}
\begin{split}
\PP[\cE_{U,m} \cap \cH_m] \; &= \; \PP[\cE_{U,n^2p} \cap \cH_{n^2p}] \prod_{n^2p \leq i \leq m-1} \PP[\cE_{U,i+1} \cap \cH_{i+1} \mid \cE_{U,i} \cap \cH_i]\\
&\leq \; \prod_{n^2p \leq i \leq m-1} \PP[e_{i+1} \notin T_U(i) \mid \cE_{U,i} \cap \cH_i]  .
\end{split}
\end{equation}
Note that $\cE_{U,i} \cap \cH_i$ depends only on the first $i$ steps of the process, so given this, the next edge $e_{i+1}$ is chosen uniformly at random from $O(i)$. 
Furthermore, $\cG_i$ implies \eqref{eq:open-estimate}, which using \eqref{eq:K4-functions-estimates} implies $q(t) \geq |O(i)|/n^2$ for $n^2p \leq i \leq m$. 
Hence, writing $t = i/(n^{2}p)$ as usual, on $\cH_i=\cG_i \cap \cT_i$ we have 
\begin{equation}
\label{eq:bound_triangle_edges}
|T_U(i)| \geq \delta u^{3} (tp)^{2} q(t) =  \delta \frac{u^{3} i^{2}}{n^4} q(t) \geq \delta  \frac{u^{3} i^{2}}{n^6} |O(i)|  .
\end{equation} 
As the process fails to choose the next edge $e_{i+1}$ from $T_U(i)$ with probability $1-|T_U(i)|/|O(i)|$, from \eqref{eq:prob_no_triangle_closed_0} and \eqref{eq:bound_triangle_edges} as well as the inequality $1-x \leq e^{-x}$ we deduce that   
\begin{equation}
\label{eq:prob_no_triangle_closed}
\PP[\cE_{U,m} \cap \cH_m] \leq \exp\left\{- \delta \frac{u^{3}}{n^6} \sum_{n^2p \leq i \leq m-1} i^{2}\right\} \leq \exp\left\{- \frac{\delta}{4} \frac{u^{3} m^3}{n^6}\right\}  .
\end{equation}
Substituting the definitions of $m$, $u$, $p$ and $t_{\max}$ into \eqref{eq:prob_no_triangle_closed} we see that 
\[
\PP[\cE_{U,m} \cap \cH_m] \leq \exp\left\{- \frac{\gamma^2\delta}{4} n^2 p^{5} t_{\max}^5  u \right\} = \exp\left\{- \frac{\gamma^2 \delta \mu^5}{4}  u \log n\right\} \leq n^{- 2 u}  ,
\]
where the last inequality follows from the definition of $\gamma$ in \eqref{eq:K4-parameters2}. 
Finally, taking the union bound over all choices of $U \subseteq [n]$ with $|U|=u$ implies \eqref{eq:prob_triangle_free_small}, and, as explained, this completes the proof. 
\end{proof}
Clearly, Theorem~\ref{thm:main_result:ramsey} is an immediate consequence of the above proof (to prove $f_{3,4}(n)<x$ it suffices to construct an $K_4$-free graph on $n$ vertices such that every subset of $x$ vertices contains a copy of $K_3$).  
Note that we did \emph{not} use that $\cE_{U,i}$ holds when establishing \eqref{eq:bound_triangle_edges}. 
With this observation we can rewrite the proof (using stochastic domination and standard Chernoff bounds) in order to show that whp every subset $U \subseteq [n]$ with $|U|=u$ contains not only one, but at least $\Omega(u^3(pt_{\max})^3)$ copies of $K_3$; we leave the details to the interested reader.

\section{Basic density arguments}
\label{sec:density}
In this section we collect some useful properties of the $K_4$-free process (in fact, these also hold for the more general $H$-free process considered in~\cite{BohmanKeevash2010H}). 
Throughout we consider $m$ and $p$ as defined in \eqref{eq:K4-parameters} and $\epsilon$ as chosen in \eqref{eq:K4-constants:Wepsmu}.

\subsection{The occurrence of a set of edges}
\label{sec:density:edges}
Essentially all results in this section are based on the following lemma by Bohman and Keevash~\cite{BohmanKeevash2010H}, which also holds for the $H$-free process whenever $H$ is strictly $2$-balanced.  
Intuitively, it states that the probability that some set of edges is present in $G(m)$ is `comparable' to that in the binomial model $G_{n,p}$, where $m$ and $p$ are defined as in~\eqref{eq:K4-parameters} and $\epsilon>0$ is chosen as in~\eqref{eq:K4-constants:Wepsmu}. 
\begin{lemma}%
\label{lem:small_subgraph}%
{\normalfont\cite[Lemma~$4.1$]{BohmanKeevash2010H}} 
For any set of edges $F \subseteq \binom{[n]}{2}$, the probability that $\cG_m$ holds and $F \subseteq E(m)$ is at most $\left(p n^{2\epsilon}\right)^{|F|}$.
\end{lemma} 
We remark that the proof given in~\cite{BohmanKeevash2010H} remains valid with our simpler definition of $\cG_i$, as it only uses that the number of open pairs is large, say $|O(i)| > n^{2-\epsilon}/2$, which readily follows from  \eqref{eq:K4-functions-estimates} and \eqref{eq:open-estimate}.

\subsection{The number of edges between two sets}
The next lemma essentially gives reasonable upper bounds on the number of edges between two (not necessarily disjoint) sets, and it is an easy consequence of Lemma~$4.2$ in~\cite{BohmanKeevash2010H}. 
\begin{lemma}%
\label{lem:edges_bounded}%
{\normalfont\cite[Lemma~$4.2$]{BohmanKeevash2010H}}
Let $\cD_i$ denote the event that for all $a,b \geq 1$ and every $A, B \subseteq [n]$ with $|A|=a$ and $|B|=b$, in $G(i)$ we have $e(A,B) < \max\{4\epsilon^{-1}(a+b),pabn^{2\epsilon}\}$.
Then the probability that $\cG_m$ holds and $\cD_m$ fails is $o(n^{-1})$. 
\end{lemma} 
With a similar reasoning as above, the proof given in~\cite{BohmanKeevash2010H} also works with our simpler version of $\cG_i$.

\subsection{Vertices which have many neighbours in some set}
Loosely speaking, the following lemma bounds the number of vertices which have many neighbours in some set $A$. 
It is a straightforward modification of Lemma~$4.3$ in~\cite{BohmanKeevash2010H}, taking into account that the `high degree' vertices may also lie in $A$. 
\begin{lemma}%
\label{lem:large_degree_bounded}%
For $A \subseteq [n]$ and $d \geq 1$, let $D_{A,d} \subseteq [n]$ denote the set of vertices which have at least $d$ neighbours in $A$. 
Let $\cN_i$ denote the event that for all $a \geq 1$ and $d \geq \max\{16\epsilon^{-1}, 2apn^{2\epsilon}\}$, in $G(i)$ we have $|D_{A,d}| < 16\epsilon^{-1}d^{-1}a$ for every $A \subseteq [n]$ with $|A| = a$.  
Then $\cD_i$ implies $\cN_i$.
\end{lemma}
\begin{proof}
We closely follow the proof of Lemma~$4.3$ in~\cite{BohmanKeevash2010H}. 
Suppose $\cD_i$ holds. Pick $A \subseteq [n]$ with $|A|=a \geq 1$ and set $B=D_{A,d}$.
Suppose $|B|=b \geq \lceil 16\epsilon^{-1}d^{-1}a \rceil$. 
Since $e(A,B) \geq db/2$ and $d \geq 16\epsilon^{-1}$, we have $e(A,B) - 4\epsilon^{-1}b \geq db/4 \geq 4\epsilon^{-1}a$.
Furthermore, $d \geq 2apn^{2\epsilon}$ implies $e(A,B) \geq db/2 \geq pabn^{2\epsilon}$.
To summarize, we have $e(A,B) \geq \max\{4\epsilon^{-1}(a+b),pabn^{2\epsilon}\}$, which contradicts $\cD_i$. 
\end{proof}

\subsection{Disjoint pairs which each have very many common neighbours in some set} 
Intuitively, the following lemma states that the number of disjoint vertex pairs, where each pair has very many common neighbours in some set, is not too large. 
\begin{lemma}%
\label{lem:codegree_into_set_bounded}% 
Let $\cM_i$ denote the event that for all $a \geq 1$ and $d \geq  \max\{300\epsilon^{-1}, ap^{2}n^{5\epsilon}, \epsilon^{-1/2} \sqrt{ap} n^{2\epsilon}\}$, for every $A \subseteq [n]$ with $|A| = a$, in $G(i)$  
the size of any set $C$ of disjoint vertex pairs with $|\Gamma(x) \cap \Gamma(y) \cap A| \geq d$ for all $xy \in C$ is at most $30\epsilon^{-1}d^{-1}a$. 
Then the probability that $\cG_m \cap \cD_m$ holds and $\cM_m$ fails is $o(n^{-1})$.
\end{lemma} 
Note that we allow the vertex pairs to intersect with $A$.
This produces some mild technical difficulties, but we overcome these using the `larger' lower bound $d \geq \epsilon^{-1/2} \sqrt{ap} n^{2\epsilon}$ on the codegree. 
\begin{proof} 
We first fix $1 \leq a \leq n$, $A \subseteq [n]$ with $|A|=a$ and $d \geq \max\{300\epsilon^{-1}, ap^{2}n^{5\epsilon}, \epsilon^{-1/2} \sqrt{ap} n^{2\epsilon}\}$, where $a,d$ are integers.   
Set $r := \lceil 30\epsilon^{-1}d^{-1}a \rceil$. It henceforth suffices to consider the case $d \leq a$, otherwise the claim is trivial. 
Assuming that $\cG_m \cap \cD_m$ holds, we now estimate the probability that there exists a set $C$ of disjoint vertex pairs with $|C|=r$, where in $G(m)$ each pair in $C$ has at least $d$ common neighbours in $A$. 
In the following we distinguish several cases, where $A_C$ denotes all vertices of $A$ which are contained in some pair in $C$.

First, suppose there exists $C_1 \subseteq C$ of size $\lceil r/2 \rceil$ in which each pair has at least $\lceil d/2 \rceil$ common neighbours in $A \setminus A_{C}$. 
To bound the probability of this event, we first use a union bound to account for all possible  $C_1$ of size $\lceil r/2 \rceil$ and choices of the $\lceil d/2 \rceil$ common neighbours $N_{xy}$ in $A \setminus A_{C}$ for each pair $xy \in C_1$, and then use Lemma~\ref{lem:small_subgraph} to bound the probability that $G(m)$ contains all the required edges, i.e., $F = \bigcup_{xy \in C_1} \{x,y\} \times N_{xy}$. 
Since by construction $|F|=2 \lceil d/2 \rceil\lceil r/2 \rceil$, whenever $\cG_m$ holds the probability of this case is bounded by 
\[
\binom{n^2}{\lceil r/2 \rceil} {\binom{a}{\lceil d/2 \rceil}}^{\lceil r/2 \rceil} \left(p n^{2 \epsilon}\right)^{2 \lceil d/2 \rceil\lceil r/2 \rceil} \leq n^{3r} \left(\frac{2 e a p^2 n^{4 \epsilon}}{d}\right)^{dr/4} \leq n^{(3-\epsilon d/5)r} \leq n^{-\epsilon dr/6}  \leq n^{-2(a+1)}  ,
\] 
where we used $\binom{x}{y} \leq (ex/y)^{y}$ as well as $d \geq ap^{2}n^{5\epsilon}$, $\epsilon d \geq 300$ and $\epsilon d r \ge 30 a$.

Second, assume there exists $C_2 \subseteq C$ of size $\lceil r/2 \rceil$ in which each pair has at least $\lceil d/2 \rceil$ common neighbours in $A_{C} \subseteq A$.
If there exists $C_3 \subseteq C_2$ of size $\lceil r/4 \rceil$ in which all pairs are outside of $A$, then for every pair $xy \in C_3$ its at least $d$ common neighbours in $A$ are (trivially) disjoint from $C_3$. 
So, with similar reasoning as above, whenever $\cG_m$ holds this occurs with probability at most  
\[
\binom{n^2}{\lceil r/4 \rceil} {\binom{a}{d}}^{\lceil r/4 \rceil} \left(p n^{2 \epsilon}\right)^{2 d \lceil r/4 \rceil} \leq n^{r} \left(\frac{e a p^2 n^{4 \epsilon}}{d}\right)^{d r/4} \leq n^{(1-\epsilon d/5)r} \leq n^{-2(a+1)}  .
\] 
Otherwise there exists $C_4 \subseteq C_2$ of size $\lceil r/4 \rceil$, in which each pair has at least one vertex in $A_C$ and at least $\lceil d/2 \rceil$ common neighbours in $A_C$.
But then $e(A_{C}) \geq rd/16$, so $d \geq 300\epsilon^{-1}$ implies $e(A_{C}) > 16\epsilon^{-1}r$. 
Using $r \leq  60\epsilon^{-1}d^{-1}a$ and $d \geq \epsilon^{-1/2} \sqrt{ap} n^{2\epsilon}$ we see that $d/r \geq \epsilon (60a)^{-1} d^2 \geq (60)^{-1}pn^{4\epsilon}$, thus $e(A_{C}) \geq rd/16 > 4 r^2pn^{2\epsilon}$. 
To sum up,  $e(A_{C}) > \max\{16\epsilon^{-1}r,4r^{2}pn^{2\epsilon}\}$, which contradicts $\cD_m$ because of $|A_{C}| \leq 2r$, so this case can not occur.

Finally, taking the union bound over all choices of $a$, $d$ and $A$ implies 
\[
\PP[ \cG_m \cap \cD_m \cap \neg\cM_m ] \leq \sum_{a \geq 1} n \binom{n}{a} 2 n^{-2(a+1)} = o(n^{-1})  ,
\]
as required. 
\end{proof}

\subsection{Deletion Lemma} 
In our proof we need good exponential upper-tail bounds on the probability that some subset of the vertices contains `too many' copies of some graph $F$.
Unfortunately, even in $G_{n,p}$ this probability is often not as small as it would need to be in order to apply union bounds, see e.g.~\cite{JansonOleszkiewiczRucinski2004SubgraphCounts,JansonRucinski2004UpperTail}.
However, R{\"o}dl and Ruci\'nski showed in~\cite{RoedlRucinski1995} that for $G_{n,p}$ such bounds can be obtained if we allow for deleting a few edges; this is usually referred to as the Deletion Lemma~\cite{JansonRucinski2004Deletion}. 
In the following we extend the classical proof to our scenario at the cost of obtaining slightly worse bounds.
As this lemma may be of independent interest we state it in a slightly more general form than needed for our purposes. 
As usual, here $\epsilon>0$ is any constant for which Lemma~\ref{lem:small_subgraph} holds.  
\begin{lemma}[`Deletion Lemma']%
\label{lem:deletion_lemma}%
Suppose $\ell \geq 1$ and that $\cS$ is a family of $\ell$-element subsets from $\binom{[n]}{2}$. 
We set  $\mu' := |\cS| p^{\ell} n^{2\ell\epsilon}$ and say that a graph $G$ \emph{contains} $\alpha \in \cS$ if all the edges of $\alpha$ are present in $G$. 
Let $\cD\cL_i(d,k)$ denote the event that there exists a set $\cI_0 \subseteq \cS$ with $|\cI_0| \leq d$ such that, setting $E_0 := \bigcup_{\alpha \in \cI_0} \alpha \subseteq \binom{[n]}{2}$, the graph $G(i) \setminus E_0$ contains at most $\mu'+k$ elements from $\cS$. 
Then for every $d,k>0$ the probability that $\cG_m$ holds and $\cD\cL_m(d,k)$ fails is at most
\begin{equation*}
\label{eq:lem:deletion_lemma}
\left(1+\frac{k}{\mu'}\right)^{-d} \leq \exp\left\{-\frac{d k}{\mu'+k}\right\}  . 
\end{equation*}
\end{lemma}
Before giving the proof, let us briefly discuss what a typical application looks like.
Suppose that for some graph $F$ we want to bound the number of $F$-copies in $G(m)$, or perhaps in some subset $U$ of $G(m)$. 
Then we set $\ell = e_F$ and let $\cS$ contain  the edge sets of all possible placements of $F$ (in $U$). With this in mind, observe that $\mu'$ corresponds up to a factor of $n^{2\ell\epsilon}$ to the expected number of $F$-copies in $G_{n,p}$ (restricted to $U$). 
Intuitively, the lemma thus states that if we are allowed to delete some edges, then substantially exceeding the expected value is very unlikely.
For $G_{n,p}$ the Deletion Lemma of R{\"o}dl and Ruci\'nski replaces $\mu'$ by $\mu := |\cS| p^{\ell}$, but we emphasize that for large deviations from $\mu$, say $k = \omega(\mu')$, both versions are essentially equivalent.  
Finally, we point out that Lemma~\ref{lem:deletion_lemma} also holds for the more general $H$-free process considered by Bohman and Keevash~\cite{BohmanKeevash2010H}, because its proof relies only on Lemma~\ref{lem:small_subgraph}, which also holds in this more general setup. 
\begin{proof}[Proof of Lemma~\ref{lem:deletion_lemma}] 
We follow the lines of the proof given by R{\"o}dl and Ruci\'nski for $G_{n,p}$, see \mbox{e.g.\ Lemma~$2.3$ in~\cite{JansonRucinski2004Deletion}}. 
For $\alpha,\beta \in \cS$ we write $\alpha \sim \beta$ if $\alpha \cap \beta \neq \emptyset$. Moreover, for $\alpha \in \cS$ and $\cI \subseteq \cS$ we write $\alpha \sim \cI$ if $\alpha \sim \beta$ for some $\beta \in \cI$.
For every $\alpha\in \cS$ let $Y_{\alpha}$ denote the indicator variable of the event that $G(m)$ contains $\alpha$, i.e., that $\alpha \subseteq E(m)$.  
Set $Z_{r} := \sum^{*}_{\alpha_1, \ldots, \alpha_{r}} \prod_{i \in [r]} Y_{\alpha_i}$, where $\sum^{*}_{\alpha_1, \ldots, \alpha_{r}}$ denotes the sum over all sequences of $\alpha_1,\ldots, \alpha_{r} \in \cS$ with $\alpha_i \not\sim \alpha_j$ for $1 \leq i < j \leq r$.
If $\cD\cL_m(d,k)$ fails, then for every set $\cI \subseteq \cS$ with $|\cI| \leq d$, if we ignore all $Y_{\alpha}$ with $\alpha \sim \cI$ then the sum of the remaining $Y_{\alpha}$ is at least $\mu'+k$. 
Hence, if $\neg\cD\cL_m(d,k)$ holds and $r \leq d$, then
\[
Z_{r+1} = \sideset{}{^*}\sum_{\alpha_1, \ldots, \alpha_{r}} \prod_{i \in [r]} Y_{\alpha_i} \sum_{\substack{\alpha \in \cS\\\alpha \not\sim \{\alpha_1, \ldots, \alpha_{r}\}}} Y_{\alpha} \geq (\mu'+k) \sideset{}{^*}\sum_{\alpha_1, \ldots, \alpha_{r}} \prod_{i \in [r]} Y_{\alpha_i} = (\mu'+k) Z_{r}  ,
\]
so by induction we have $Z_{r} \geq (\mu'+k)^{r}$ for $1 \leq r \leq d+1$.
Recall that the factors $Y_{\alpha_i}$ in each term of $Z_{r}$ are indicator variables for \emph{disjoint} edge sets $\alpha_i$, each of size $\ell$. 
So, Lemma~\ref{lem:small_subgraph} yields 
\[
\EE[Z_{r} \indic{\cG_m}]  = \sideset{}{^*}\sum_{\alpha_1, \ldots, \alpha_{r}} 
\PP\Big[ \alpha_1 \cup \cdots \cup \alpha_r \subseteq E(m) \text{ and } \cG_m] 
\leq \sideset{}{^*}\sum_{\alpha_1, \ldots, \alpha_{r}} \left(p n^{2 \epsilon}\right)^{\ell r}  \leq |\cS|^{r} \left(p n^{2 \epsilon}\right)^{\ell r} = \left(\mu'\right)^{r}  . 
\]
Putting everything together and using Markov's inequality, with $r=\lceil d \rceil$ we obtain 
\[
\PP[ \cG_m \cap \neg\cD\cL_m(d,k) ] \leq \PP[Z_{r} \geq (\mu'+k)^{r} \text{ and } \cG_m] \leq \frac{\EE[Z_{r} \indic{\cG_m}]}{(\mu'+k)^{r}} \leq \left(\frac{\mu'}{\mu'+k}\right)^{r} \leq \left(1+\frac{k}{\mu'}\right)^{-d}  .
\]
Finally, observe that using $1-x \leq e^{-x}$ we also have
\[
\left(1+\frac{k}{\mu'}\right)^{-d} = \left(1-\frac{k}{\mu'+k}\right)^{d} \leq \exp\left\{-\frac{d k}{\mu'+k}\right\}  ,
\]
which completes the proof. 
\end{proof}
If $\cS$ is a family of subsets from $\binom{[n]}{2}$ of arbitrary (possibly distinct) sizes, then essentially the same proof works with $\mu' := \sum_{\alpha \in \cS} (pn^{2\epsilon})^{|\alpha|}$; we leave the straightforward details to the interested reader.

\subsubsection{Bounding the number of certain triples in every subset}
In our application, for every subset we need to bound the number of certain triples that have at least one common neighbour. 
As we expect that a `typical' triple has no common neighbours, it is more convenient to bound the number of corresponding quadruples, where the fourth vertex is a common neighbour of the others. 
Here we use Lemma~\ref{lem:deletion_lemma} to show that after deleting a few edges, the number of such quadruples is bounded. 
\begin{lemma}%
\label{lem:deletion_lemma_subgraph_count_K3edge}
Let $\cQ_i$ denote the event that for all positive integers $r$ and disjoint sets $A,B,C \subseteq [n]$ of size $r$ there exists a set $E_0 \subseteq [n] \times (A \cup B \cup C)$ of size at most $20\epsilon^{-1} r$ such that $G(i)$ contains at most $r^{3}np^{4} n^{10 \epsilon}$ quadruples $(u,v,w,z) \in A \times  B \times C \times [n]$ with $z \notin \{u,v,w\}$ and $\{uw, zu, zv, zw\} \subseteq E(i)\setminus E_0$. 
Then the probability that $\cG_m$ holds and $\cQ_m$ fails is $o(n^{-1})$. 
\end{lemma}
Roughly speaking, in Section~\ref{sec:proof:verification:partial:trend} we will use this lemma to bound the total number of such quadruples, which are as in Figure~\ref{fig:quadruple:Xi} on page~\pageref{fig:quadruple:Xi}. 
To this end we shall show later that for `nice' disjoint subsets $A,B,C \subseteq [n]$ the number of quadruples `destroyed' by $E_0$ is not too large. 
\begin{proof}[Proof of Lemma~\ref{lem:deletion_lemma_subgraph_count_K3edge}] 
We combine the Deletion Lemma with a standard union bound argument. 
First, fix $r$ with $1 \leq r \leq \lfloor n/3 \rfloor$.
Second, fix disjoint sets $A,B,C \subseteq [n]$ of size $r$ and set $\Gamma := [n] \times (A \cup B \cup C)$. 
Let $\cS \subseteq \binom{\Gamma}{4}$ denote the family of edge sets $\{uw, zu, zv, zw\}$ for all $(u,v,w,z) \in A \times  B \times C \times [n]$ with $z \notin \{u,v,w\}$. 
Observe that $|\cS| = r^{3}(n-3)$ and so the $\mu'$ of Lemma~\ref{lem:deletion_lemma} satisfies $\mu' \leq r^{3}np^{4}n^{8\epsilon}$.  
Next, we set $k := r^{3}np^{4} n^{10 \epsilon}/2$ and $d := 5\epsilon^{-1} r$. 
Clearly, we have $\mu' + k < r^{3}np^{4} n^{10 \epsilon}$. 
By Lemma~\ref{lem:deletion_lemma} the probability that $\cG_m$ holds and $\cD\cL_m(d,k)$ fails is at most
\[
\left(1+\frac{k}{\mu'}\right)^{-d} \leq \left(1 + \frac{n^{2\epsilon}}{2} \right)^{-d} \leq n^{-d \epsilon} = n^{-5r}  .
\]
If $\cD\cL_m(d,k)$ holds then this particular choice of disjoint $A,B,C \subseteq [n]$ has the required properties (noting that $|E_0| \le 4d$ holds). 
So, the union bound over all choices of $r$ and $A,B,C$ implies 
\[
\PP[ \cG_m \cap \neg\cQ_m ] \leq \sum_{1 \leq r \leq \lfloor n/3 \rfloor} \binom{n}{r} \binom{n-r}{r} \binom{n-2r}{r} n^{-5r} \leq \sum_{r \geq 1} n^{-2r} = o(n^{-1})  ,
\] 
as claimed.  
\end{proof}

\section{Differential equation method}
\label{sec:dem}
In this section we present a variant of the differential equation method which may be of independent interest: it allows us to track several variables (which may depend on each other) in \emph{every} subset of a certain size.

\subsection{Basic idea of the differential equation method}
\label{sec:dem:idea_wormald}
Wormald~\cite{Wormald1995DEM} developed a method to show that in certain discrete random processes a collection $\cV$ of random variables is whp `close' to the solution of a system of differential equations. 
The basic idea of Wormald's method can briefly be described as follows.
First, we need to make sure that the expected one-step changes of all random variables in $\cV$ can be expressed using only variables from $\cV$, which might involve enlarging $\cV$. 
Then, by pretending that all variables are continuous, these expected changes suggest a system of ordinary differential equations. 
Finally, the main effort is devoted to showing that the random variables in $\cV$ are whp near the solution of the differential equations, and for this purpose tools from probability theory are used. 
For this approach to work we usually need to make sure that the expected one-step changes are roughly `correct' and that very large one-step changes are rare or do not happen at all, see e.g.~\cite{BohmanKeevash2010H,Wormald1995DEM,Wormald1999DEM}. 
As it turns out, martingale estimates are particularly useful for showing the desired concentration results.

\subsection{Large deviation inequalities for martingales}
Suppose we have a filtration $\cF_0 \subseteq \cF_1 \subseteq \cdots$ and a sequence $X_0,X_1,\ldots$ of random variables where each $X_i$ is $\cF_i$-measurable. 
Then $X_0,X_1,\ldots$ is a \emph{supermartingale} if $\EE[X_{i+1}|\cF_i] \leq X_i$ for all $i$ and a \emph{submartingale} if $\EE[X_{i+1}|\cF_i] \geq X_i$ for all $i$. 
Furthermore, we say that $X_0,X_1,\ldots$ is \emph{$(M,N)$-bounded} if for all $i$ we have
\[ -M \leq X_{i+1} - X_{i} \leq N    . \] 
The following martingale inequalities are due to Bohman~\cite{Bohman2009K3} and follow from the original martingale inequality of Hoeffding~\cite{Hoeffding1963}. 
Observe that for supermartingales we have $\EE[X_{i}] \leq \EE[X_0]$ and for submartingales $\EE[X_{i}] \geq \EE[X_0]$.  
Intuitively, both inequalities thus give (one sided) exponential error probabilities for deviating `too much' from the expected value. 
\begin{lemma}%
\label{lem:martingale_upper_tail}%
{\normalfont\cite[Lemma~$7$]{Bohman2009K3}} 
Suppose $0 \equiv X_0,X_1,\ldots$ is an $(M,N)$-bounded supermartingale with $M \leq N/10$. Then for any $m \geq 1$ and $0 < a < m M$ we have
\[ \PP[X_m \geq a] \leq e^{-\frac{a^2}{3 m M N}}  . \]
\end{lemma}
\begin{lemma}%
\label{lem:martingale_lower_tail}%
{\normalfont\cite[Lemma~$6$]{Bohman2009K3}} 
Suppose $0 \equiv X_0,X_1,\ldots$ is an $(M,N)$-bounded submartingale with $M \leq N/2$.
Then for any $m \geq 1$ and $0 < a < m M$ we have
\[ \PP[X_m \leq -a] \leq e^{-\frac{a^2}{3 m M N}}  . \]
\end{lemma}

\subsection{Tracking several variables in every subset of a certain size}
Suppose we want to track several variables in every subset $U \subseteq [n]$ of a certain size $u$ using the differential equation method. 
As it turns out, in order to show that the expected or maximum step by step changes are `correct' or not too large, we often would like to slightly alter the subgraph induced by $U$ in order to remove (or reduce) `bad' substructures, e.g.\ by passing to a subset of $U$ (deleting vertices) or by `ignoring' some edges in $U$. 
Roughly speaking, we would like to use a union bound over all possible `alterations' and then prove that for some alteration all variables tracked are concentrated. 
Such an approach is rather standard in probabilistic combinatorics, however, it does not fit into the usual framework of the differential equation method as presented in~\cite{BohmanKeevash2010H} or~\cite{Wormald1999DEM}, for example. 
In this section we show how this basic idea can be formulated in the framework of the differential equation method, which allows us to \emph{routinely} use these kind of arguments.

\subsubsection{Main concepts and ideas} 
\label{sec:dem:idea}
In the following we introduce the concepts used in our approach. 
First, we consider a set $\cC$ of `configurations' $\Sigma$. 
For example, $\Sigma$ could correspond to an `alteration' of some $U$; for the purpose of deleting vertices we might have $\Sigma = (U,U')$ with $U' \subseteq U$.
Then, for every $\Sigma \in \cC$ we intend to track several random variables, i.e., for all $i \leq m$ and $j \in \cV$ we want to bound the variables $X_{(\Sigma,j)}(i)$. 
As we shall see, it suffices to track the variables for $\Sigma$ only as long as the configuration $\Sigma$ is `good'; otherwise a `certificate' for the fact that $\Sigma$ is `bad' turns out to be sufficient for our purposes. 
To this end we introduce the `bad' events $\badE$, which will hold if $\Sigma$ is bad after \mbox{step $i$} (we point out that it can also capture other bad events not related to $\Sigma$, e.g.\ $\badE$ can be used to avoid/detect undesirable states of the random process). 
For example, if $\Sigma$ is an alteration of $U$, then $\badE$ should hold if $\Sigma$ has `too many bad substructures' after \mbox{step $i$}; in other words, with $\badE$ we are able to `detect' that the alteration $\Sigma$ did not remove the bad substructures, as intended.
With these concepts we aim at a statement of the following form: whp, for all $\Sigma \in \cC$ one of the following holds for every $i^* \leq m$:
\vspace{-0.75em}
\begin{enumerate}
\item[(a)] the configuration $\Sigma$ is bad before step $i^*$, i.e., $\badE$ holds for some $i < i^*$, or\vspace{-0.5em} 
\item[(b)] for $\Sigma$ we can track all variables up to step $i^*$, i.e., for all $i \leq i^*$ and $j \in \cV$ the variables $X_{(\Sigma,j)}(i)$ are close to their expected values.\vspace{-0.75em}
\end{enumerate}
Let us briefly discuss how to use the previous statement. Suppose we want to track several variables in every $U$ of size $u$, and that we intend to `ignore' some edges inside $U$.
Then we construct $\cC$ such that every $\Sigma \in \cC$ corresponds to an alteration of some $U$. For example, here we might have, say, $\Sigma=(U,F)$ with $F \subseteq \binom{U}{2}$.
Now it remains to show (possibly using different methods) that whp at least one `good' alteration $\Sigma$ exists for every $U$, e.g.\ via the Deletion Lemma (cf.\ Lemma~\ref{lem:deletion_lemma}). 
For such good $\Sigma$ we are in \mbox{case (b)} for every $i^* \leq m$, and thus able to track all the desired variables.

\subsubsection{A variant of the differential equation method} 
\label{sec:dem:variant}
Now we state our variant of the differential equation method, which is in large parts based on Lemma~$7.3$ in~\cite{BohmanKeevash2010H}, but contains several improvements and new ingredients, e.g.\ the error function $f_{\sigma}$ as well as the `configurations' $\Sigma$ and `bad' events $\badE$. 
An important difference to~\cite{BohmanKeevash2010H} is that we track the variables in $\cV$ for every configuration $\Sigma \in \cC$; in other words, each variable tracked `belongs' to a certain configuration. 
Inspired by~\cite{BohmanKeevash2010H} we introduce a `global' good event $\cH_i$, and `give up' as soon as it fails. 
An important new ingredient is the `local' bad event $\badEL$. When it holds we only `give up' for $\Sigma$, i.e., stop tracking the variables that belong to $\Sigma$; the other configurations/variables are unaffected. 
\begin{lemma}[`Differential Equation Method']%
\label{lem:dem}%
Suppose that $m=m(n)$ and $s=s(n)$ are positive parameters. 
Let $\cC=\cC(n)$ and $\cV=\cV(n)$ be sets. 
For every $0 \leq i \leq m$ set $t=t(i):=i/s$. 
Suppose we have a filtration $\cF_0 \subseteq \cF_1 \subseteq \cdots$ and random variables $X_{\sigma}(i)$ and $Y^{\pm}_{\sigma}(i)$ which satisfy the following conditions. 
Assume that for all $\sigma \in \cC \times \cV$ the random variables $X_{\sigma}(i)$ are non-negative and $\cF_i$-measurable for all $0 \leq i \leq m$, and that for all $0 \leq i < m$ the random variables $Y^{\pm}_{\sigma}(i)$ are non-negative, $\cF_{i+1}$-measurable and satisfy
\begin{equation}
\label{eq:lem:dem:rv_relation}
X_{\sigma}(i+1) - X_{\sigma}(i) =  Y^{+}_{\sigma}(i) - Y^{-}_{\sigma}(i)  .
\end{equation}
Furthermore, suppose that for all $0 \leq i \leq m$ and $\Sigma \in \cC$ we have an event $\badE \in \cF_i$.
Then, for all $0 \leq i \leq m$ we define $\badEL := \bigcup_{0 \leq j \leq i} \badE[j]$. 
In addition, suppose that for each $\sigma \in \cC \times \cV$ we have positive parameters $u_{\sigma}=u_{\sigma}(n)$, $\lambda_{\sigma}=\lambda_{\sigma}(n)$, $\beta_{\sigma}=\beta_{\sigma}(n)$, $\tau_{\sigma}=\tau_{\sigma}(n)$, $s_{\sigma} = s_{\sigma}(n)$ and $S_{\sigma}=S_{\sigma}(n)$, as well as functions $x_{\sigma}(t)$ and $f_{\sigma}(t)$ that are smooth and non-negative for $t \geq 0$. 
For all $0 \leq i^* \leq m$ and $\Sigma \in \cC$, let $\cG_{i^*}(\Sigma)$ denote the event that for every $0 \leq i \leq i^*$ and $\sigma=(\Sigma,j)$ with $j \in \cV$ we have 
\begin{equation}
\label{eq:lem:dem:parameter_trajectory}
X_{\sigma}(i) = \left(x_{\sigma}(t) \pm \frac{f_{\sigma}(t)}{s_{\sigma}} \right) S_{\sigma}  .
\end{equation}
Next, for all $0 \leq i^* \leq m$ let $\cE_{i^*}$ denote the event that for every $0 \leq i \leq i^*$ and $\Sigma \in \cC$ the event $\badEL[i-1] \cup \cG_{i}(\Sigma)$ holds. 
Moreover, assume that we have an event $\cH_i \in \cF_i$ for all $0 \leq i \leq m$ with $\cH_{i+1} \subseteq \cH_{i}$ for all $0 \leq i < m$. 
Finally, suppose that the following conditions hold: 
\begin{enumerate}
\item(Trend hypothesis)
For all $0 \leq i < m$ and $\sigma = (\Sigma,j) \in \cC \times \cV$, whenever $\cE_i \cap \neg\badEL \cap \cH_i$ holds we have 
\begin{equation}
\label{eq:lem:dem:parameter_martingale_property}
\EE\big[Y^{\pm_1}_{\sigma}(i) \mid \cF_i \big] = \left( y^{\pm_1}_{\sigma}(t) \pm  \frac{h_{\sigma}(t)}{s_{\sigma}} \right)  \frac{S_{\sigma}}{s}  , 
\end{equation}
where $y_{\sigma}^{\pm}(t)$ and $h_{\sigma}(t)$ are smooth non-negative functions such that
\begin{equation}
\label{eq:lem:dem:derivative}
x'_{\sigma}(t) = y^+_{\sigma}(t) - y^-_{\sigma}(t) \qquad \text{ and } \qquad f_{\sigma}(t) \geq 2\int_{0}^{t} h_{\sigma}(\tau) \ d\tau + \beta_{\sigma}  .
\end{equation}

\item(Boundedness hypothesis) 
For all $0 \leq i < m$ and $\sigma = (\Sigma,j) \in \cC \times \cV$, whenever $\cE_i \cap \neg\badEL \cap \cH_i$ holds we have
\begin{equation}
\label{eq:lem:dem:parameter_max_change}
Y^{\pm}_{\sigma}(i)  \leq \frac{\beta_{\sigma}^2}{s_{\sigma}^2 \lambda_{\sigma} \tau_{\sigma}} \cdot  \frac{S_{\sigma}}{u_{\sigma}}  .
\end{equation}

\item(Initial conditions) 
For all $\sigma \in \cC \times \cV$ we have 
\begin{equation}
\label{eq:lem:dem:initial_condition}
X_{\sigma}(0) = \left(x_{\sigma}(0) \pm \frac{\beta_{\sigma}}{3s_{\sigma}} \right) S_{\sigma}  .
\end{equation}

\item(Bounded number of configurations and variables) We have 
\begin{equation}
\label{eq:lem:dem:bounded_parameters}
\max\left\{|\cC|,|\cV|\right\} \leq \min_{\sigma \in \cC \times \cV} e^{u_{\sigma}}  .
\end{equation}
 
\item(Additional technical assumptions)
For all $\sigma \in \cC \times \cV$ we have 
\begin{gather}
\label{eq:lem:dem:technical_assumptions:sm}
s \geq \max\{15 u_{\sigma} \tau_{\sigma} (s_{\sigma}  \lambda_{\sigma}/\beta_{\sigma})^2, 9s_{\sigma} \lambda_{\sigma}/\beta_{\sigma}\}  , \qquad s/(18 s_{\sigma} \lambda_{\sigma}/\beta_{\sigma}) < m \leq s \cdot \tau_{\sigma}/1944  ,\\
\label{eq:lem:dem:technical_assumptions:xy}
\sup_{0 \leq t \leq m/s} y^{\pm }_{\sigma}(t) \leq \lambda_{\sigma}  , \qquad 
\int_0^{m/s} |x''_{\sigma}(t)| \ dt \leq \lambda_{\sigma}  , \\
\label{eq:lem:dem:technical_assumptions:h}
h_{\sigma}(0) \leq s_{\sigma} \lambda_{\sigma}  \qquad \text{ and }  \qquad  \int_0^{m/s} |h'_{\sigma}(t)| \ dt \leq s_{\sigma} \lambda_{\sigma}  .
\end{gather}
\end{enumerate}
Then we have 
\begin{equation}
\label{eq:lem:dem}
\mathbb{P}[\neg\cE_m \cap \cH_m] \leq 4 \max_{\sigma \in \cC \times \cV}e^{-u_{\sigma}}  .
\end{equation}
\end{lemma}  
Note that Lemma~\ref{lem:dem} allows us to deduce that $\cE_m \cap \cH_m$ holds whp, if the above conditions $1$--$5$ are satisfied for $n$ large enough, $\cH_m$ holds whp and $u_{\sigma} = \omega(1)$ for all $\sigma \in \cC \times \cV$.  
Furthermore, observe that $\cE_i \cap \neg\badEL$ implies $\cG_{i}(\Sigma)$. 
So, to calculate the expected value in \eqref{eq:lem:dem:parameter_martingale_property} for some $\sigma=(\Sigma,j)$, we may assume that in the previous step all variables are essentially `correct' for $\Sigma$, i.e., for all $\sigma=(\Sigma,j')$ with $j' \in \cV$ we have \eqref{eq:lem:dem:parameter_trajectory}.

As our result is optimized for a slightly different setup, we only briefly compare it with Lemma~$7.3$ in~\cite{BohmanKeevash2010H}. 
One important advantage of Lemma~\ref{lem:dem} is that we state the approximation error differently and allow for a larger family of error functions. 
To this end we point out that by setting $h_{\sigma}(t) := (e_{\sigma} \cdot x_{\sigma} + \gamma_{\sigma})'(t)/4$ and $f_{\sigma}(t) := e_{\sigma}(t)x_{\sigma}(t) -\theta_{\sigma}(\sigma)e_{\sigma}(t)/s_{\sigma}+\theta_{\sigma}(t)$ we obtain essentially the same approximation error as in~\cite{BohmanKeevash2010H}, but by setting e.g.\ $\beta_{\sigma}^{-1} := \lambda_{\sigma} := n^{\epsilon/7}$, and $\tau_{\sigma} := n^{\epsilon/2}$ and  and then choosing $s_{\sigma} \geq n^{\epsilon}$ and $u_{\sigma} = 2k_{\sigma} \log n$, we can weaken several of the assumptions significantly (for example, in the additional technical assumptions we relax $y^{\pm }_{\sigma}(t) = O(1)$ to $y^{\pm }_{\sigma}(t) \leq n^{\epsilon/7}$, their $c$ from $\Omega(1)$ to $\Omega(n^{-\epsilon/7})$, and the lower bound on $m$ from $m>s$ to, say, $m > s n^{-\epsilon}$).  
The simpler formulation of the approximation error using the function $f_{\sigma}$ was suggested by Oliver Riordan. 
Another new ingredient is the introduction of the parameters $\lambda_{\sigma}$, $\beta_{\sigma}$ and $\tau_{\sigma}$, which allow for a trade-off between the approximation error, the boundedness hypothesis and the additional technical assumptions. 
For example, in certain applications this might allow for larger one-step changes in \eqref{eq:lem:dem:parameter_max_change}, since in contrast to Lemma~$7.3$ in~\cite{BohmanKeevash2010H}, we do not rule out the possibility that our parameters are small, say, $o(n^{\epsilon})$. 
Finally, we remark that essentially all of our improvements/modifications also apply to the setup of Lemma~$7.3$ in~\cite{BohmanKeevash2010H}, and for the readers convenience we have stated the resulting improved variant of the differential equation method in Appendix~\ref{apx:dem:general}.

Next, let us briefly discuss the typical usage of certain parameters and give some intuition for a few assumptions  (we refer to Sections~7 and~5.1 in~\cite{BohmanKeevash2010H} and~\cite{Wormald1999DEM} for a further discussion of the setup). 
The parameter $u_{\sigma}$ relates the number of variables and configurations in \eqref{eq:lem:dem:bounded_parameters} to the error probability in \eqref{eq:lem:dem}. 
For example, if we want to track a few variables inside every subset of size $u$, then the left hand side of \eqref{eq:lem:dem:bounded_parameters} is usually dominated by the number of subsets $O(n^{u})$, so $u_{\sigma} = \Omega(u \log n)$ is a convenient choice. 
With a union bound argument in mind, this indicates that $u_{\sigma}$ also restricts the kind of random variables we can hope to track: their expected values $\mu$ should be larger than $u_{\sigma}$ since (in the independent case) large deviations often occur with probabilities that are exponential in $\mu$. 
Up to additional error terms this is essentially reflected by the boundedness hypothesis: 
assuming that the maximum step-wise changes of $X_{\sigma}$ are at least one the right hand side  of~\eqref{eq:lem:dem:parameter_max_change} must also be at least one, which basically means that $S_{\sigma}$ needs to be larger than $u_{\sigma}$, 
where by \eqref{eq:lem:dem:parameter_trajectory} the `scaling' $S_{\sigma}$ is roughly comparable to the expected value. 
Turning to the error terms, if possible, it is convenient to guess some appropriate function $f_{\sigma}(t)$ with $f_{\sigma}(0) \geq \beta_{\sigma}$ and then define $h_{\sigma}(t) := f'_{\sigma}(t)/2$; this satisfies the corresponding constraint in~\eqref{eq:lem:dem:derivative}. 
Furthermore, if possible, it might be useful to choose the parameters such that $\tau_{\sigma} := \lambda_{\sigma}$ and $\beta_{\sigma} := \lambda_{\sigma}^{-1}$ or $\beta_{\sigma} = \Theta(1)$, since this reduces the number of parameters and simplifies several conditions. 
Finally, we remark that the conditions in~\eqref{eq:lem:dem:technical_assumptions:sm} essentially ensure that  $s$ and $m$ are roughly the same and not too small. 
\begin{proof}[Proof of Lemma~\ref{lem:dem}] 
The proof is similar to the proof of Lemma~$7.3$ in~\cite{BohmanKeevash2010H}. 
The important differences here are the more involved definition of the desired event $\cE_m$, the modified error functions $f_{\sigma}(t)$ and $h_{\sigma}(t)$, as well as the new parameters $\lambda_{\sigma}$, $\beta_{\sigma}$ and $\tau_{\sigma}$. The main new ingredients are the configurations $\Sigma$ and bad events $\badE$. 
Let us briefly outline the main ideas of the proof. 
First, using \eqref{eq:lem:dem:parameter_martingale_property} we add and subtract appropriate functions from $Y^{\pm}_{\sigma}(i)$ 
in order to construct super-/submartingales with an initial value of $0$. 
Suppose $\cH_m$ holds and that $\cE_m$ fails for the first time at step $i$. 
Roughly speaking, it suffices to consider the case when $\cG_{i}(\Sigma)$ fails. 
But, if \eqref{eq:lem:dem:parameter_trajectory} is violated, then, as we shall see, this implies that at least one of our super-/submartingales deviates substantially from $0$. 
By Lemma~\ref{lem:martingale_upper_tail} and~\ref{lem:martingale_lower_tail} these large deviations 
are very unlikely, and it turns out that even after using a union bound over all such events the resulting error probability is negligible.

First, we derive some additional inequalities that our functions satisfy. 
Using \eqref{eq:lem:dem:technical_assumptions:h} we see that
\begin{equation}
\label{eq:lem:dem:technical_assumptions:h:function}
\sup_{0 \leq t \leq m/s} h_{\sigma}(t) \leq h_{\sigma}(0) + \int_0^{m/s} |h'_{\sigma}(t)| \ dt \leq 2 s_{\sigma} \lambda_{\sigma}  .
\end{equation}
We claim that the following estimates hold for all $0 \leq i^* \leq m$, writing $t(i)=i/s$ and $t^*=i^*/s$: 
\begin{align}
\label{eq:lem:dem:sum:x}
\frac{1}{s} \sum_{i=0}^{i^*-1} x'_{\sigma}\big(t(i)\big) \cdot S_{\sigma}&=  \left( x_{\sigma}(t^*) - x_{\sigma}(0) \pm \frac{\lambda_{\sigma}}{s} \right) S_{\sigma} \qquad \text{ and }\\
\label{eq:lem:dem:sum:h:upper}
\frac{1}{s} \sum_{i=0}^{i^*-1} h_{\sigma}\big(t(i)\big) \cdot  \frac{2S_{\sigma}}{s_\sigma} &=  \left( 2\int_{0}^{t^*}h_{\sigma}(t)\ dt \pm \frac{2s_{\sigma}\lambda_{\sigma}}{s} \right) \frac{S_{\sigma}}{s_\sigma}  .
\end{align} 
Both bounds are obtained with very similar calculations as in the proof of Lemma~$7.3$ in~\cite{BohmanKeevash2010H},
 using the Euler-Maclaurin summation formula (see e.g.~\cite{Apostol1999}) and then estimating the approximation error with the additional technical assumptions \eqref{eq:lem:dem:technical_assumptions:xy} and \eqref{eq:lem:dem:technical_assumptions:h}; therefore we defer the details to Appendix~\ref{apx:proof:lem:dem}.

Second, we define several random variables and start with $Y^{\pm_1 \pm_2}_{\sigma}$ (recall that this is an abbreviation for four different variables, one for each way of choosing $\pm_1$ and $\pm_2$). 
For all $(\Sigma,j) = \sigma \in \cC \times \cV$, if $\cE_i \cap \neg\badEL \cap \cH_i$ holds we set 
\begin{equation}
\label{eq:lem:dem:def:rv:Y}
Y^{\pm_1 \pm_2}_{\sigma}(i) := Y^{\pm_1}_{\sigma}(i) - \left(y^{\pm_1}_{\sigma}(t) \mp_2 \frac{h_{\sigma}(t)}{s_{\sigma}}\right) \frac{S_{\sigma}}{s}  ,
\end{equation}
and, otherwise (i.e., if $\neg\cE_i \cup \badEL \cup \neg\cH_i$ holds) we define $Y^{\pm_1 \pm_2}_{\sigma}(i)$ to be $0$.
Note that in this case $Y^{\pm_1 \pm_2}_{\sigma}(i') = 0$ for all $i' \geq i$. 
Next, we define
\begin{equation}
\label{eq:lem:dem:def:rv:ZMN}
 Z^{\pm_1\pm_2}_{\sigma}(i) := \sum_{i'=0}^{i-1} Y^{\pm_1\pm_2}_{\sigma}(i')  , \qquad M_{\sigma} :=  \frac{3\lambda_{\sigma} S_{\sigma}}{s} \qquad \mbox{ and } \qquad N_{\sigma} := \frac{2\beta_{\sigma}^2}{s_{\sigma}^2 \lambda_{\sigma} \tau_{\sigma}} \cdot \frac{S_{\sigma}}{u_{\sigma}} 
\end{equation}
We claim that $Z^{\pm-}_{\sigma}(i)$ and $Z^{\pm+}_{\sigma}(i)$ are $(M_{\sigma},N_{\sigma})$-bounded super-/submartingales with $Z^{\pm_1\pm_2}_{\sigma}(0) = 0$ and $M_{\sigma} \leq N_{\sigma}/10$. 
Clearly, $Z^{\pm_1\pm_2}_{\sigma}(0)=0$ holds, and, furthermore, \eqref{eq:lem:dem:technical_assumptions:sm} implies $M_{\sigma} \leq N_{\sigma}/10$. 
Using the trend hypothesis it is easy to establish the super-/submartingale properties, and we deduce that the random variables are $(M_{\sigma},N_{\sigma})$-bounded using the boundedness hypothesis, the additional technical assumptions and \eqref{eq:lem:dem:technical_assumptions:h:function}.  
We defer the straightforward details to Appendix~\ref{apx:proof:lem:dem}.

In the following we estimate the probability of the event $\neg\cE_m \cap \cH_m$.  
Loosely speaking, we focus on the first step $i^* \leq m$ where $\cE_i$ fails, and, in particular, on the $\Sigma \in \cC$ for which $\badEL[i^*-1] \cup \cG_{i^*}(\Sigma)$ fails. 
Note that \eqref{eq:lem:dem:derivative} and \eqref{eq:lem:dem:initial_condition} ensure that $\cG_{0}(\Sigma)$ holds for all $\Sigma \in \cC$, and thus $\cE_{0}$ holds. 
So, considering all $i^* \leq m$, $\Sigma \in \cC$ and using $\cH_m \subseteq \cH_{i^*-1}$, we have 
\begin{equation}
\label{eq:lem:dem:union_bound:0}
\begin{split}
\neg \cE_m \cap \cH_m & \subseteq \bigcup_{1 \leq i^* \leq m} \big[\cH_m \cap \cE_{i^*-1} \cap \neg\cE_{i^*} \big] \\
& \subseteq \bigcup_{1 \leq i^* \leq m} \bigcup_{\Sigma \in \cC} \left[\cH_{i^*-1} \cap \cE_{i^*-1} \cap \neg\big(\badEL[i^*-1] \cup \cG_{i^*}(\Sigma)\big) \right]   . 
\end{split}
\end{equation}
Henceforth we fix $1 \leq i^* \leq m$ and $\Sigma \in \cC$. 
Using that $\cE_i \cap \neg\badEL$ implies $\cG_{i}(\Sigma)$, we see that 
\begin{equation*}
\label{eq:lem:dem:fails}
 \cH_{i^*-1} \cap \cE_{i^*-1} \cap \neg\big(\badEL[i^*-1] \cup \cG_{i^*}(\Sigma)\big) \ = \ \cH_{i^*-1} \cap \cE_{i^*-1} \cap \neg\badEL[i^*-1] \cap \cG_{i^*-1}(\Sigma) \cap \neg\cG_{i^*}(\Sigma)  . 
\end{equation*} 
Observe that when $\cG_{i^*-1}(\Sigma)$ holds, the event $\cG_{i^*}(\Sigma)$ can only fail if $X_{\sigma}(i^*)$ violates  \eqref{eq:lem:dem:parameter_trajectory} for some $\sigma=(\Sigma,j)$ with $j \in \cV$, and for the following calculations we fix such a $\sigma=(\Sigma,j)$.

Suppose that $\cH_{i^*-1} \cap \cE_{i^*-1} \cap \neg\badEL[i^*-1]$ holds and $X_{\sigma}(i^*)$ fails to satisfy \eqref{eq:lem:dem:parameter_trajectory} because $X_{\sigma}(i^*) > (x_{\sigma}(t^*) + f_{\sigma}(t^*)/s_{\sigma}) S_{\sigma}$. 
With a virtually identical calculation as in the proof of the Lemma~$7.3$ in~\cite{BohmanKeevash2010H}, using the relation \eqref{eq:lem:dem:rv_relation} as well as the definitions \eqref{eq:lem:dem:def:rv:Y} and \eqref{eq:lem:dem:def:rv:ZMN}, we obtain
\begin{equation}
\label{eq:lem:dem:ineq:Z+-Z-+}
Z^{+-}_{\sigma}(i^*) - Z^{-+}_{\sigma}(i^*) = X_{\sigma}(i^*)-X_{\sigma}(0) - \frac{1}{s} \sum_{i=0}^{i^*-1} x'_{\sigma}\big(t(i)\big) \cdot S_{\sigma} - \frac{1}{s} \sum_{i=0}^{i^*-1} h_{\sigma}\big(t(i)\big) \cdot  \frac{2S_{\sigma}}{s_\sigma}  ,
\end{equation}
and we include its short proof in Appendix~\ref{apx:proof:lem:dem} for the sake of completeness. 
With this in hand, using the lower bound on $X_{\sigma}(i^*)$, the initial condition \eqref{eq:lem:dem:initial_condition} as well as \eqref{eq:lem:dem:sum:x} and  \eqref{eq:lem:dem:sum:h:upper}, we deduce 
\[
Z^{+-}_{\sigma}(i^*) - Z^{-+}_{\sigma}(i^*)  > \left( f_{\sigma}(t^*) - 2\int_{0}^{t^*}h_{\sigma}(t)\ dt - \frac{\beta_{\sigma}}{3} - \frac{3s_{\sigma}\lambda_{\sigma}}{s} \right) \frac{S_{\sigma}}{s_{\sigma}} \geq \frac{\beta_{\sigma}}{3} \frac{S_{\sigma}}{s_{\sigma}}  ,
\]
where we used \eqref{eq:lem:dem:derivative} and \eqref{eq:lem:dem:technical_assumptions:sm}, i.e., $s \geq 9s_{\sigma} \lambda_{\sigma}/\beta_{\sigma}$, for the last inequality. 
This readily implies 
\begin{equation}
\label{eq:lem:dem:error_event}
Z^{+-}_{\sigma}(i^*) \geq \frac{\beta_{\sigma}}{6} \frac{S_{\sigma}}{s_{\sigma}} =: a \qquad \text{ or } \qquad Z^{-+}_{\sigma}(i^*) \leq - \frac{\beta_{\sigma}}{6} \frac{S_{\sigma}}{s_{\sigma}} = -a  . 
\end{equation} 
Since the variables $Z^{\pm_1\pm_2}_{\sigma}(i)$ are `frozen' once $X_{\sigma}(i)$ leaves the allowed range \eqref{eq:lem:dem:parameter_trajectory}, we deduce that $Z^{+-}_{\sigma}(m) \geq a$ or $Z^{-+}_{\sigma}(m) \leq - a$ holds.

Similarly, if $\cH_{i^*-1} \cap \cE_{i^*-1} \cap \neg\badEL[i^*-1]$ holds and $X_{\sigma}(i^*)$ fails to satisfy \eqref{eq:lem:dem:parameter_trajectory} because $X_{\sigma}(i^*) < (x_{\sigma}(t^*) - f_{\sigma}(t^*)/s_{\sigma}) S_{\sigma}$, with calculations completely analogous to those of the previous case, we deduce that  $Z^{--}_{\sigma}(m) \geq a$ or  $Z^{++}_{\sigma}(m) \leq -a$ holds.

Plugging our findings into \eqref{eq:lem:dem:union_bound:0}, we obtain 
\begin{equation}
\label{eq:lem:dem:union_bound:1}
\neg \cE_m \cap \cH_m \subseteq \bigcup_{\sigma \in \cC \times \cV}  
\big[ \{ Z^{+-}_{\sigma}(m) \geq a \} \cup \{ Z^{--}_{\sigma}(m) \geq a \} \cup \{ Z^{++}_{\sigma}(m) \leq -a \} \cup \{ Z^{-+}_{\sigma}(m) \leq -a \} \big] .
\end{equation}
Recall that $Z^{\pm-}_{\sigma}(i)$ and $Z^{\pm+}_{\sigma}(i)$ are $(M_{\sigma},N_{\sigma})$-bounded super-/submartingales with $M_{\sigma} \leq N_{\sigma} / 10$ and initial values of $0$.
Note that \eqref{eq:lem:dem:technical_assumptions:sm}, i.e., $m > s/(18 s_{\sigma} \lambda_{\sigma}/\beta_{\sigma})$, implies $a < m M_{\sigma}$. 
Therefore, using Lemmas~\ref{lem:martingale_upper_tail} and~\ref{lem:martingale_lower_tail} as well as the definition of $a$, $M_{\sigma}$ and $N_{\sigma}$, we deduce that the probabilities of $Z^{\pm-}_{\sigma}(m) \geq a$ and $Z^{\pm+}_{\sigma}(m) \leq - a$ are each bounded by 
\begin{equation}
\label{eq:lem:dem:error_probability}
\exp\left\{- \frac{a^2}{3 m M_{\sigma} N_{\sigma}} \right\} = \exp\left\{- \frac{1}{648} \frac{s}{m} \tau_{\sigma} u_{\sigma} \right\} 
\leq \exp\left\{- 3u_{\sigma} \right\}  ,
\end{equation}
where we used \eqref{eq:lem:dem:technical_assumptions:sm}, i.e., $m \leq s \cdot \tau_{\sigma}/1944$, for the last inequality. 
Finally, we estimate \eqref{eq:lem:dem:union_bound:1} with a union bound argument. Using \eqref{eq:lem:dem:bounded_parameters} and \eqref{eq:lem:dem:error_probability} we deduce
\[
\PP[\neg \cE_m \cap \cH_m] \leq \sum_{\sigma \in \cC \times \cV} 4 e^{-3u_{\sigma}} \leq 4 \max_{\sigma \in \cC \times \cV }e^{-u_{\sigma}}  ,
\]
and the proof is complete. 
\end{proof}

\section{Proof of the main technical result}
\label{sec:proof}
This section is devoted to the proof of Theorem~\ref{thm:k4_many_edges_close_k3} and is organized as follows. 
First, in Section~\ref{sec:motivation} we sketch some of the ideas used. 
Next, in Section~\ref{sec:formal_setup} we introduce the formal setup. 
In Section~\ref{sec:finishing_the_proof} we then give the proof of Theorem~\ref{thm:k4_many_edges_close_k3}. 
The argument is simple, but relies on two involved combinatorial statements, which are proved in Sections~\ref{sec:trajectory_verification} and~\ref{sec:good_configurations_exist}.

\subsection{Proof idea} 
\label{sec:motivation}
In this section we informally outline some of the ideas used in the proof of Theorem~\ref{thm:k4_many_edges_close_k3}; we stress that most of the notions and statements are refined and made precise in later sections. 
To focus on the main ideas we will assume that $i$, the number of steps of the $K_4$-free process, is large, and furthermore ignore constants as well as $n^{\epsilon}$ factors whenever they are not crucial. Intuitively, this allows us to ignore whether an edge is open or not in our rough calculations (because by \eqref{eq:K4-functions-estimates} and \eqref{eq:open-estimate} we have $|O(i)| \geq n^{2-\epsilon}$). 
Moreover, by \eqref{eq:K4-parameters}, \eqref{eq:degree-estimate}, \eqref{eq:closed-estimate} and \eqref{eq:K4-parameters2} we obtain the approximations $|C_{e_{i+1}}(i)| \approx p^{-1} = n^{2/5}$ and $|\Gamma(v)| \approx |U| \approx n^{3/5}$.

Fix $U \subseteq [n]$ of size $|U| \approx n^{3/5}$. 
Let $T_U(i)$ contain all open pairs in $U$ which would create a copy of $K_3$ in $U$ if they were added to $G(i)$.  
In order to prove Theorem~\ref{thm:k4_many_edges_close_k3} it suffices to establish a lower bound on $|T_U(i)|$.  
Note that for every pair $uv \in T_U(i)$ there exists at least one $w \in U$ such that $w \in \Gamma(u) \cap \Gamma(v)$ in $G(i)$. 
Let $Z_U(i)$ denote all triples $(u,v,w) \in U^3$ where $uv \in O(i)$ and $w \in \Gamma(u) \cap \Gamma(v)$ in $G(i)$. 
As mentioned in the introduction, we expect that the graph generated by the $K_4$-free process shares many properties with the uniform random graph, which in turn is similar to the binomial random graph $G_{n,p}$. 
With this in mind, for each $uv \in T_U(i)$ the number of $w \in U$ with $(u,v,w) \in Z_U(i)$ should typically be roughly $|U|p^{2}=o(1)$.    
In other words, we expect that up to constants $|Z_U(i)| \approx |T_U(i)|$, and so for proving Theorem~\ref{thm:k4_many_edges_close_k3} it should suffice to prove a lower bound on $|Z_U(i)|$. 
For this we intend to apply the differential equation method of Section~\ref{sec:dem}, and so we introduce additional variables, $X_U(i)$ and $Y_U(i)$, in order to keep track of the step-wise changes resulting in elements of $Z_U(i)$. 
More precisely, let $X_U(i)$ denote all $(u,v,w) \in U^3$ with $\{uv,vw,uw\} \subseteq O(i)$ and let $Y_U(i)$ denote all $(u,v,w) \in U^3$ with $\{uv,vw\} \subseteq O(i)$ and $uw \in E(i)$.

We remark that $|X_U(i)|$ can easily be tracked with the differential equation method. 
However, $|Y_U(i)|$ causes some difficulties with respect to the one-step changes, but these are easy to resolve using the ideas of Section~\ref{sec:dem:idea}. 
Guided by our random graph intuition (and our convention that we ignore $n^{\epsilon}$ factors), we expect the `scaling' $S_{\sigma}$ of $|Y_U(i)|$, which essentially corresponds to the expected value of $|Y_U(m)|$, to satisfy $S_{\sigma} \approx |U|^3p \approx n^{7/5}$. 
Recall that in order to use the differential equation method, the maximum one-step change must be bounded by \eqref{eq:lem:dem:parameter_max_change}, i.e., by roughly $|U|^2p \approx n^{4/5}$.
A triple $(u,v,w) \in X_U(i)$ is only added to $Y_U(i+1)$ if $e_{i+1} = uw$, and so at most $|U| \approx n^{3/5}$ triples  are added in one step, which causes no problems. 
Note that triples $(u,v,w) \in Y_U(i)$ are removed, i.e., not in $Y_U(i+1)$, if $e_{i+1} \in \{uv,vw\}$ or $\{uv,vw\} \cap C_{e_{i+1}}(i) \neq \emptyset$. 
Therefore the number of triples removed in one step can be up to $(1+|C_{e_{i+1}}(i)|) \cdot \max_{v \in U} |\Gamma(v) \cap U| \approx p^{-1} \cdot \max_{v \in U} |\Gamma(v) \cap U|$. 
So, as long as, say, $\max_{v \in U} |\Gamma(v) \cap U| \leq n^{1/3}$ holds, the one-step changes are bounded by $p^{-1}n^{1/3} = n^{11/15} = n^{4/5 - 1/15}$, which is small enough for using the differential equation method. 
Note that although we expect $|\Gamma(v) \cap U| \approx |U|p \approx n^{1/5}$, we can not even guarantee $|\Gamma(v) \cap U| =o(|U|)$ for every $v \in U$ and $U \subseteq [n]$ with $|U| \approx n^{3/5}$, since $|\Gamma(v)| \approx |U|$. 
But, as it turns out, by removing some vertices from $U$ we can overcome this issue. 
Indeed, as we shall see, from Lemma~\ref{lem:edges_bounded} and~\ref{lem:large_degree_bounded} we can deduce that $U$ typically contains a subset $U'$ of size $|U'| \approx |U|$ such that for every $v \in U'$ we have  $|\Gamma(v) \cap U'| \leq n^{1/3}$. 
Of course, at the beginning of the $K_4$-free process, when we start tracking the variables, we do not know which $U' \subseteq U$ satisfy these properties in later steps. 
So, intuitively we `try' all possible $U'$ for each $U$ and `stop' tracking as soon as the maximum degree inside $U'$ is too large. 
This approach is sound because our previous line of argument suggests that for every $U$ at least one `good' $U'$ with bounded degree exists. 
Using the terminology of Section~\ref{sec:dem:idea}, this idea can be formalized as follows. 
We introduce configurations of the form $\Sigma=(U,U')$, where $U' \subseteq U$ satisfies $|U'| \approx |U|$, and for each such $\Sigma$ we track the variables only inside $U'$. Furthermore, we define the bad event $\badE$ such that it holds whenever the maximum degree inside $U'$ is too large. 
These ideas would already suffice to track the variables $|X_{U'}(i)|$ and $|Y_{U'}(i)|$ through the evolution of the $K_4$-free process using Lemma~\ref{lem:dem}.

Obtaining a lower bound on $|Z_U(i)|$ with the differential equation method is more difficult and actually the main technical challenge of our proof. 
In addition to the bound on the one-step changes, we also need to deal with the issue that although the next edge satisfies $e_{i+1}=vw$, the triple $(u,v,w) \in Y_{U}(i)$ is not always added to $Z_{U}(i+1)$, since adding $vw$ might close $uv$. 
This is an important difference to bounding the independence number, where this issue does not arise, cf.~\cite{Bohman2009K3,BohmanKeevash2010H}. 
For the other variables they track, Bohman and Keevash~\cite{BohmanKeevash2010H} use a union bound argument (based on density considerations) to overcome this issue. 
However, in contrast to~\cite{BohmanKeevash2010H} we have to handle this for \emph{every} large subset, and therefore a slight variation of their approach is unlikely to work here. 
To overcome the difficulties arising, we substantially refine the definition of configurations and bad events. 
Unfortunately, the additional ideas used are rather technical and at this point an informal sketch would probably fail to be of much help.

\subsection{Formal setup}
\label{sec:formal_setup}
In this section we present the formal setup which is used in our application of the differential equation method. 
First, we define the configurations as well as the variables we want to track in every subset of a certain size. 
Afterwards we introduce the high probability events $\cH_i$ and `bad' events $\badE$.

\subsubsection{Configurations and random variables}
\label{sec:main-proof:configs_variables}
Recall that by \eqref{eq:K4-parameters2} we have $u = \gamma n p t_{\max} = \gamma \mu n^{3/5} \sqrt[5]{\log n}$. We set
\begin{equation}
\label{eq:K4-parameters3}
k:=u/15 = \gamma/15 \cdot n p t_{\max} = \gamma \mu /15 \cdot  n^{3/5} \sqrt[5]{\log n}  .
\end{equation} 
Now, we define the configurations $\cC$ to be the set of all $\Sigma=(U,\Pi)$, where $U \in \binom{[n]}{u}$ and $\Pi=(A,B,C)$ with disjoint $A,B,C \in \binom{U}{k}$. 
For the sake of brevity we write $K = K(\Sigma) := A \cup B \cup C$.

\begin{figure}[tp]
	\centering
  \setlength{\unitlength}{1bp}%
  \begin{picture}(116.49, 87.81)(0,0)
  \put(0,0){\includegraphics{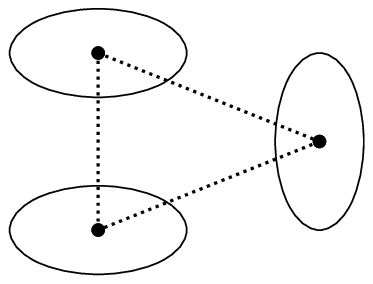}}
  \put(25.71,71.23){\fontsize{11.38}{13.66}\selectfont \makebox[0pt]{$u$}}
  \put(49.99,67.16){\fontsize{11.38}{13.66}\selectfont \makebox[0pt]{$A$}}
  \put(94.83,59.67){\fontsize{11.38}{13.66}\selectfont \makebox[0pt]{$C$}}
  \put(101.06,45.71){\fontsize{11.38}{13.66}\selectfont \makebox[0pt]{$w$}}
  \put(25.71,19.75){\fontsize{11.38}{13.66}\selectfont \makebox[0pt]{$v$}}
  \put(49.99,13.30){\fontsize{11.38}{13.66}\selectfont \makebox[0pt]{$B$}}
  \end{picture}%
	\hspace{2.75em}
  \setlength{\unitlength}{1bp}%
  \begin{picture}(116.49, 87.81)(0,0)
  \put(0,0){\includegraphics{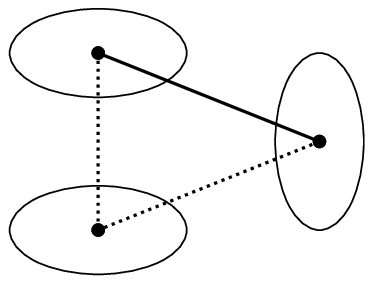}}
  \put(25.71,71.23){\fontsize{11.38}{13.66}\selectfont \makebox[0pt]{$u$}}
  \put(49.99,67.16){\fontsize{11.38}{13.66}\selectfont \makebox[0pt]{$A$}}
  \put(94.83,59.67){\fontsize{11.38}{13.66}\selectfont \makebox[0pt]{$C$}}
  \put(101.06,45.71){\fontsize{11.38}{13.66}\selectfont \makebox[0pt]{$w$}}
  \put(25.71,19.75){\fontsize{11.38}{13.66}\selectfont \makebox[0pt]{$v$}}
  \put(49.99,13.30){\fontsize{11.38}{13.66}\selectfont \makebox[0pt]{$B$}}
  \end{picture}%
	\hspace{2.75em}
  \setlength{\unitlength}{1bp}%
  \begin{picture}(116.49, 87.81)(0,0)
  \put(0,0){\includegraphics{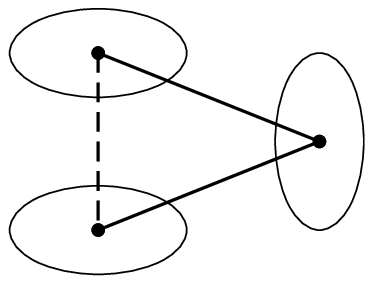}}
  \put(25.71,71.23){\fontsize{11.38}{13.66}\selectfont \makebox[0pt]{$u$}}
  \put(49.99,67.16){\fontsize{11.38}{13.66}\selectfont \makebox[0pt]{$A$}}
  \put(94.83,59.67){\fontsize{11.38}{13.66}\selectfont \makebox[0pt]{$C$}}
  \put(101.06,45.71){\fontsize{11.38}{13.66}\selectfont \makebox[0pt]{$w$}}
  \put(25.71,19.75){\fontsize{11.38}{13.66}\selectfont \makebox[0pt]{$v$}}
  \put(49.99,13.30){\fontsize{11.38}{13.66}\selectfont \makebox[0pt]{$B$}}
  \end{picture}%
	\caption{\label{fig:triples}$(u,v,w)$ is open/intermediate/partial with respect to $\Sigma=(U,\Pi)$ with $\Pi=(A,B,C)$. Solid lines represent edges, dotted lines open pairs and dashed lines pairs that are open or closed. }
\end{figure}

We track several random variables for every $\Sigma \in \cC$, which count certain triples $(u,v,w) \in A \times B \times C$ with $uv \in O(i) \cup C(i)$ and $uw,vw \in E(i)\cup O(i)$. 
The sets of triples which are called \emph{open wrt.\ $\Sigma$} and \emph{intermediate wrt.\ $\Sigma$}, respectively, are defined as
\begin{align}
\label{eq:def:open_triples}
	\opens &:= \big\{ (u,v,w) \in A \times B \times C \;:\; \{uv,vw,uw\} \subseteq O(i) \big\} \quad \text{ and}\\
\label{eq:def:intermediate_triples}
	\interms &:= \big\{ (u,v,w) \in A \times B \times C \;:\; \{uv,vw\} \subseteq O(i) \wedge uw \in E(i) \big\}  .
\end{align}
Finally, we introduce a set $\parts$, whose triples are called \emph{partial wrt.\ $\Sigma$}, which will satisfy 
\begin{equation}
\label{eq:def:partial_inclusion}
\parts \subseteq \big\{ (u,v,w) \in A \times B \times C \;:\; uv \in O(i) \cup C(i) \wedge \{uw,vw\} \subseteq E(i) \big\}  .
\end{equation}
We define $\parts$ inductively as follows. 
At the beginning we set $\parts[0] := \emptyset$. 
Suppose the process chooses $e_{i+1}=xy \in O(i)$ as the next edge in step $i+1$. 
Then a triple $(u,v,w) \in \interms$ is \emph{added} to $\parts[i+1]$, i.e., is in $\parts[i+1]$, if $vw = e_{i+1}$, $uv \notin C_{vw}(i)$,  and there is no $w' \in C$ such that $(u,v,w') \in \parts$. 
Furthermore, a triple $(u,v,w) \in \parts$ is \emph{removed}, i.e., not in $\parts[i+1]$, or \emph{ignored}, i.e., remains in $\parts[i+1]$, according to the following rules (see also Figure~\ref{fig:pairs_R3} on page~\pageref{fig:pairs_R3}):   
\begin{center}
\begin{minipage}[b]{0.95\linewidth}
\begin{description}
\item[Case 1.] If $uv = e_{i+1}$, then the triple $(u,v,w)$ is removed,\vspace{-0.2em}
\item[Case 2.] If $e_{i+1}=xy \in C_{uv}(i)$ and $e_{i+1} \cap uv = \emptyset$, then the triple $(u,v,w)$ is\vspace{-0.5em}
	\begin{enumerate}
		\item[(R2)] removed if $\min\{|\Gamma(x) \cap \Gamma(y) \cap A|,|\Gamma(x) \cap \Gamma(y) \cap B|\} \leq kpn^{-20\epsilon}$ holds in $G(i)$, and\vspace{-0.2em}
		\item[(I2)] ignored otherwise.\vspace{-0.2em}
	\end{enumerate} 
\item[Case 3.] If $e_{i+1}=xy \in C_{uv}(i)$ and $e_{i+1} \cap uv = x$, then the triple $(u,v,w)$ is\vspace{-0.5em}
	\begin{enumerate}
		\item[(R3a)] removed if $|\Gamma(y) \cap K| \leq p^{-1}n^{-15\epsilon}$,\vspace{-0.2em}
		\item[(R3b)] removed if there exists a vertex $z \in \Gamma(x) \cap \Gamma(y)$ such that $\{u,v\} \setminus \{x\} \subseteq \Gamma(y) \cap \Gamma(z) \cap K$ and $|\Gamma(y) \cap \Gamma(z) \cap K| \leq kpn^{-20\epsilon}$ hold in $G(i)$, and\vspace{-0.2em}
		\item[(I3)] ignored otherwise.\vspace{-0.30em}
	\end{enumerate}
\end{description}
\end{minipage}
\end{center}
The way in which triples are added ensures that for every $u \in A$ and $v \in B$ there is at most one triple in $\parts$ which contains $uv$. 
This is an important ingredient and will be exploited repeatedly in our proof. 
The rules for removing triples from $\parts$ ensure that the step-wise changes are not too large (see Section~\ref{sec:proof:verification:partial:bounded} for more details). 
Intuitively, the `ignored' cases occur only infrequently and, as we shall later see, their contribution to $\parts$ turns out to be negligible.  
With the bound on the codegree given by \eqref{eq:codegree-estimate} in mind, these rules are rather natural, with the possible exception of (R3b), which is inspired by~\cite{Bohman2009K3}. 
Finally, the inclusion \eqref{eq:def:partial_inclusion} clearly holds and we remark that every $(u,v,w) \in \parts$ with $uv \in C(i)$ was ignored in some step $i' \leq i$.

We are mainly interested in $\parts$, the other sets $\opens$ and $\interms$ are needed in order to keep track of the step-wise changes resulting in elements of $\parts$, cf.\ Figure~\ref{fig:triples}. 
In order to prove Theorem~\ref{thm:k4_many_edges_close_k3}, it suffices to obtain a lower bound on the number of triples in $\parts$ which can still be completed to a copy of $K_3$. 
For this we define $\partOs$ as follows:
\begin{equation}
\label{eq:def:partial_open}
\partOs := \big\{ (u,v,w) \in \parts \;:\; uv \in O(i) \big\}  .
\end{equation}
The definition of $\partOs$ may seem overly complicated, and one could think that a simpler definition, say similar to $\opens$ and $\interms$, could be sufficient as well.
It turns out (see Lemmas~\ref{lem:dem:trajectories} and~\ref{lem:dem:config}) that this is in fact an important part of our proof: on the one hand we \emph{need} to relax the definition, i.e.,  to ignore some triples and allow for $uv \in C(i)$, in order to ensure that the  step-wise changes are not too large, and on the other hand we must use special rules for removing triples from $\parts$ in order to ensure that the expected changes are still `correct' and furthermore to make sure that we do not ignore too many triples, i.e., to guarantee that $|\parts| \approx |\partOs|$. 
Together with the events defined in the next section this allows us to track $|\parts|$ using our variant of the differential equation method, and, furthermore, to obtain a lower bound on $|\partOs|$.

\subsubsection{`Bad' events and high probability events}
\label{sec:main-proof:bad_high_probability_events}
In this section we introduce the high probability events $\cH_i$ and `bad' events $\badE$. 
Intuitively, $\cH_i$ denotes the event that in addition to the results of Bohman and Keevash~\cite{BohmanKeevash2010H} the density results of Section~\ref{sec:density} hold. 
More precisely, for every $0 \leq i \leq m$ we define the event $\cH_i$ as
\begin{equation}
\label{eq:def:high_probability_Hi}
\cH_i := \cD_i \cap \cG_i \cap \cJ_i  \cap \cM_i \cap \cN_i \cap \cQ_i  , 
\end{equation}
where $\cG_i$ and $\cJ_i$ are defined as in Theorem~\ref{thm:BohmanKeevash2010H}, $\cD_i$, $\cM_i$ and $\cN_i$ as in Lemmas~\ref{lem:edges_bounded}--\ref{lem:codegree_into_set_bounded}, and $\cQ_i$ as in Lemma~\ref{lem:deletion_lemma_subgraph_count_K3edge}. 
Clearly, $\cH_i$ depends only on the first $i$ steps, and $\cH_{i+1} \subseteq \cH_i$ holds.

\begin{figure}[tp]
	\centering
  \setlength{\unitlength}{1bp}%
  \begin{picture}(168.19, 87.81)(0,0)
  \put(0,0){\includegraphics{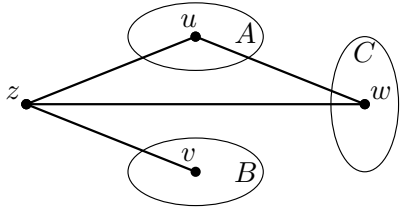}}
  \put(101.68,67.16){\fontsize{11.38}{13.66}\selectfont \makebox[0pt]{$A$}}
  \put(146.52,59.67){\fontsize{11.38}{13.66}\selectfont \makebox[0pt]{$C$}}
  \put(152.75,45.71){\fontsize{11.38}{13.66}\selectfont \makebox[0pt]{$w$}}
  \put(80.24,22.59){\fontsize{11.38}{13.66}\selectfont \makebox[0pt]{$v$}}
  \put(101.68,14.44){\fontsize{11.38}{13.66}\selectfont \makebox[0pt]{$B$}}
  \put(13.86,45.71){\fontsize{11.38}{13.66}\selectfont \makebox[0pt]{$z$}}
  \put(80.24,72.87){\fontsize{11.38}{13.66}\selectfont \makebox[0pt]{$u$}}
  \end{picture}%
 	\caption{\label{fig:quadruple:Xi}A quadruple $(u,v,w,z) \in \Xi_{\Sigma}(i)$: adding the pair $vw$ closes $uv$ (solid lines represent edges). The vertex $z$ may also be in one of the vertex classes, i.e., in $A \cup B \cup C$. }
\end{figure}

Next, we introduce the `bad' event $\badE$. 
To this end recall that we use the abbreviation $K = K(\Sigma) = A \cup B \cup C$. 
Now, for every $0 \leq i \leq m$ and $\Sigma \in \cC$ we define the event $\badE$ as 
\begin{equation}
\label{eq:def:bad_Bi}
\badE := \badE[1,i] \cup \badE[2,i] \cup \badE[3,i]  , 
\end{equation}
where we define as $\badE[1,i]$, $\badE[2,i]$ and $\badE[3,i]$ as follows:
\vspace{-0.5em}
\begin{flushright}
\begin{minipage}[c]{0.95\linewidth}
\begin{center}
\begin{minipage}[b]{0.95\linewidth}
\begin{enumerate}
\item[$\cB_{1,i}(\Sigma)$ :=] in $G(i)$ the maximum degree inside $K$ is larger than $kpn^{5\epsilon}$, \vspace{-0.4em} 
\item[$\cB_{2,i}(\Sigma)$ :=] 
the number of $xy \in \binom{[n]}{2}$ with $\min\{|\Gamma(x) \cap \Gamma(y) \cap A|,|\Gamma(x) \cap \Gamma(y) \cap B|\} \geq kpn^{-20\epsilon}$ is larger than $k n^{-20 \epsilon}$ in $G(i)$, and\vspace{-0.5em}
\item[$\cB_{3,i}(\Sigma)$ :=] in $G(i)$ there exists a pair $xy \in \binom{[n]}{2}$ that satisfies $xy \in \{uw, zu, zv, zw\}$ for more than $k^2pn^{-15\epsilon}$ quadruples $(u,v,w,z) \in \Xi_{\Sigma}(i)$, where $\Xi_{\Sigma}(i)$ contains all quadruples $(u,v,w,z) \in A \times B \times C \times [n]$ with $z \notin \{u,v,w\}$ and $\{uw, zu, zv, zw\} \subseteq E(i)$. 
\vspace{-0.4em}
\end{enumerate} 
\end{minipage}
\end{center}
\end{minipage}
\end{flushright}
Note that $\badE$ depends only on the first $i$ steps and is increasing, so $\cB_{i}(\Sigma) \subseteq \cB_{i+1}(\Sigma)$ holds. 
Loosely speaking, the definition of $\badE$ ensures that whenever it holds certain `bad' substructures do not appear during the first $i$ steps. 
This allows us to track the random variables we defined in Section~\ref{sec:main-proof:configs_variables} with our variant of the differential equation method.

In the following we briefly give some intuition for the definition of $\badE$. 
The event $\neg\badE[1,i]$ ensures that the degree inside $K$ is not too large, which is mainly used to bound the one-step changes. 
Loosely speaking, $\neg\badE[2,i]$ ensures that not too many triples are ignored because of (I2), which will be important for showing  $|\parts| \approx |\partOs|$. 
Finally, together with $\cQ_i$ the event $\neg\badE[3,i]$ essentially implies that the number of triples $(u,v,w) \in \interms$ with $uv \in C_{vw}(i)$ is not too large (observe that for every such triple there exists $z \in [n]$ for which $(u,v,w,z) \in \Xi_{\Sigma}(i)$, cf.\ Figures~\ref{fig:triples} and~\ref{fig:quadruple:Xi}). 
This will be crucial for showing that the expected number of triples added to $\parts[i+1]$ in one step is not too small (see Section~\ref{sec:proof:verification:partial:trend} for more details).

\subsection{Proof of Theorem~\ref{thm:k4_many_edges_close_k3}} 
\label{sec:finishing_the_proof}
We use the following two statements to prove Theorem~\ref{thm:k4_many_edges_close_k3}.  
Intuitively, the first (probabilistic) lemma  implies that for `good' configurations $\Sigma$ the variables $|\opens|$, $|\interms|$ and $|\parts|$ are tightly concentrated. 
Roughly speaking, the second (deterministic) lemma states that for every set $U$ of size $u$ there exists a good configuration $\Sigma^* = (U,\Pi)$ for which $\partOCs \subseteq \partCs$ is large. 
\begin{lemma}
\label{lem:dem:trajectories}
Suppose $\cC$, $\badE$, $\cH_i$ as well as $\opens$, $\interms$ and $\parts$ are defined as in Sections~\ref{sec:main-proof:configs_variables} and~\ref{sec:main-proof:bad_high_probability_events}. 
Furthermore, define $m=m(n)$ and $p=p(n)$ as in~\eqref{eq:K4-parameters}, and $k=k(n)$ as in~\eqref{eq:K4-parameters3}. 
Set $s_{o} := n^{2\epsilon}$ and $t=t(i) := i/(n^2p)$. 
Define $q(t)$ and $f(t)$ as in~\eqref{eq:K4-functions}. 
For all $0 \leq i^* \leq m$ and $\Sigma \in \cC$, let $\cG_{i^*}(\Sigma)$ denote the event that for every $0 \leq i \leq i^*$ we have 
\begin{align}
\label{eq:lem:dem:trajectories:opens}
|\opens| &= \left( {q(t)}^3	 \pm \frac{f(t){q(t)}^2}{s_{o}} \right) k^3  ,\\
\label{eq:lem:dem:trajectories:interms}
|\interms| &= \left( 2 t {q(t)}^2 \pm \frac{f(t){q(t)}}{s_{o}} \right) k^3p  \text{ and}\\
\label{eq:lem:dem:trajectories:parts}
|\parts| &= \left( 2 t^2 q(t) \pm \frac{f(t)}{s_{o}} \right) k^3p^2  .
\end{align}  
Next, let $\cE_{i^*}$ denote the event that for all $0 \leq i \leq i^*$ and $\Sigma \in \cC$ the event $\cB_{i-1}(\Sigma) \cup \cG_{i}(\Sigma)$ holds. 
Then $\cE_m \cap \cH_m$ holds with high probability. 
\end{lemma}
Let us give some intuition for the trajectories our variables follow. 
As usual, we expect $G(i)$ to share many properties with the binomial random graph $G_{n,\rho}$ with edge-density $\rho= 2tp$, since $\binom{n}{2}\rho \approx i = tn^2p$. 
Furthermore, by \eqref{eq:open-estimate} the proportion of pairs which are open roughly equals $q(t)$. 
So, with this in mind, we expect $|\opens| \approx {q(t)}^3 k^3$, $|\interms| \approx 2 t p \: {q(t)}^2 k^3$ and $|\parts| \approx 2 (tp)^2 q(t)k^3$. 
We remark that for the partial triples we `lose' a factor of two since we only count those triples where the edge $uw$ appears before $vw$. 
\begin{lemma}
\label{lem:dem:config}
Suppose $\cC$, $\badE$, $\cH_i$ as well as $\parts$ and $\partOs$ are defined as Sections~\ref{sec:main-proof:configs_variables} and~\ref{sec:main-proof:bad_high_probability_events}. 
Furthermore, define $m=m(n)$ and $p=p(n)$ as in~\eqref{eq:K4-parameters},  $u=u(n)$ as in~\eqref{eq:K4-parameters2} and $k=k(n)$ as in~\eqref{eq:K4-parameters3}. 
Set $t=t(i) := i/(n^2p)$ and define $q(t)$ as in~\eqref{eq:K4-functions}. 
For $n$ large enough, if $G(i)$ was generated by the $K_4$-free process and satisfies $\cH_i$, then for every $U \in \binom{[n]}{u}$ there exists $\Sigma^*=(U,\Pi) \in \cC$ such that $\neg\cB_{i-1}(\Sigma^*)$ and $|\partCs \setminus \partOCs| \leq k^3p^2n^{-10\epsilon}$ hold. 
\end{lemma}
The proofs of these lemmas are rather involved and therefore deferred to Sections~\ref{sec:trajectory_verification} and~\ref{sec:good_configurations_exist}. 
Next we show how they imply our main technical result. 
\begin{proof}[Proof of Theorem~\ref{thm:k4_many_edges_close_k3}] 
By Lemma~\ref{lem:dem:trajectories} the event $\cE_m \cap \cH_m$ holds whp, and so by monotonicity whp $\cE_i \cap \cH_i$ holds for every $0 \leq i \leq m$. 
Therefore, using Lemma~\ref{lem:dem:config} we deduce that, whp, for every $0 \leq i \leq m$ and $U \in \binom{[n]}{u}$ there exists a configuration $\Sigma^*=(U,\Pi)$ such that 
\begin{equation}
\label{eq:eq:main_proof:bound_ignored}
|\partCs \setminus \partOCs| \leq k^3p^2n^{-10\epsilon} 
\end{equation}
and $\cE_i \cap \neg\badEC[i-1]$ hold. 
In the following we show that for every such `good' $\Sigma^*=(U,\Pi)$ the set $\partOCs$ is large. 
Note that $t=i/(n^2p) \geq 1$ for $i \geq n^2p$, and recall that $q(t) \geq n^{-\epsilon/2}$ and $f(t) \leq n^{\epsilon}$ by \eqref{eq:K4-functions-estimates}. 
So, since $\cE_i \cap \neg\badEC[i-1]$ implies $\cG_{i}(\Sigma^*)$, using \eqref{eq:K4-functions-estimates} and \eqref{eq:lem:dem:trajectories:parts}, for $n^2p \leq  i \leq m$ we have 
\begin{equation}
\label{eq:eq:main_proof:lower_bound}
|\partCs| \geq k^{3}(tp)^2q(t)  . 
\end{equation}
Thus, using $\partOCs \subseteq \partCs$, \eqref{eq:eq:main_proof:bound_ignored} and \eqref{eq:eq:main_proof:lower_bound} as well as \eqref{eq:K4-parameters2} and \eqref{eq:K4-parameters3}, i.e., $k=u/15$ and $\delta = 1/7000$, we deduce that, whp, for every $n^2p \leq i \leq m$ and $U \in \binom{[n]}{u}$ there exists $\Sigma^*=(U,\Pi)$ such that 
\begin{equation*}
\begin{split}
|\partOCs| &=  |\partCs| - |\partCs \setminus \partOCs| \\
&\geq k^{3}(tp)^2q(t) - k^{3}p^2n^{-10\epsilon} \geq k^{3}(tp)^2q(t)/2 \geq \delta u^{3}(tp)^2q(t)  . 
\end{split}
\end{equation*}
Note that for every triple $(u,v,w) \in \partOCs$ we have $\{u,v,w\} \subseteq U$ as well as $uv \in O(i)$ and $\{uw,vw\} \subseteq E(i)$. 
Recall that the inductive definition of $\partCs$ ensures that every $uv$ with $u \in A$ and $v \in B$ is contained in at most one of the triples in $\partCs$. 
So, using $\partOCs \subseteq \partCs$, we see that for every pair of distinct triples $(u,v,w),(u',v',w') \in \partOCs$ we have $uv \neq u'v'$. 
We deduce that, whp, for every $n^2p \leq i \leq m$ and $U \in \binom{[n]}{u}$ there exists $\Sigma^*=(U,\Pi)$ such that the set 
\[
\tilde{T}_{U}(i) := \{ uv \in O(i) \;:\; \text{there exists $w \in U$ such that $(u,v,w) \in \partOCs$}\}
\]
has size $|\tilde{T}_{U}(i)| = |\partOCs| \geq \delta u^{3}(tp)^2q(t)$ and contains only open pairs $uv \in \tilde{T}_{U}(i)$ such that adding $uv$ to $G(i)$ completes a triangle in $U$. 
Therefore $\cT_m$ holds whp, which proves the theorem. 
\end{proof}

\section{Trajectory verification}
\label{sec:trajectory_verification}
In this section we prove Lemma~\ref{lem:dem:trajectories}, which is our main probabilistic statement. 
We work with the `natural' filtration given by the $K_4$-free process, where $\cF_i$ corresponds to the first $i$ steps, and  tacitly assume that $n$ is sufficiently large whenever necessary. 
Note that $\cH_i$ depends only on the first $i$ steps, and that $\cH_m$ holds whp (using  Theorem~\ref{thm:BohmanKeevash2010H}, Lemmas~\ref{lem:edges_bounded}--\ref{lem:codegree_into_set_bounded} and~\ref{lem:deletion_lemma_subgraph_count_K3edge}). 
We apply the differential equation method (Lemma~\ref{lem:dem}) with $s := n^2p$ and the purely formal set of variables $\cV:=\{X,Y,Z\}$, where for all $\Sigma \in \cC$ we let $X_{(\Sigma,X)}(i) := |\opens|$, $X_{(\Sigma,Y)}(i) := |\interms|$ and $X_{(\Sigma,Z)}(i) := |\parts|$. For the sake of clarity we will usually work directly with the variables under consideration, e.g.\ with $|\interms|$ instead of $X_{(\Sigma,Y)}(i)$. 
Recalling that $\badE$ is monotone increasing, we see that $\badE = \badEL$. In addition, for all $\sigma \in \cC \times \cV$ we define 
\begin{equation}
\label{def:dem:parameters}
\beta_{\sigma}=1  , \qquad u_{\sigma} := u n^{\epsilon} = \omega(1)  , \qquad s_{\sigma} := s_{o} = n^{2\epsilon} \qquad \text{ and } \qquad \lambda_{\sigma} := \tau_{\sigma} := n^{\epsilon}  .
\end{equation} 
Note that using $k=u/15$, with this parametrization the boundedness hypothesis \eqref{eq:lem:dem:parameter_max_change} simplifies  to 
\begin{equation}
\label{eq:dem:triples:max_bound}
Y^{\pm}_{\sigma}(i)  \leq  \frac{\beta_{\sigma}^2}{s_{\sigma}^2 \lambda_{\sigma} \tau_{\sigma}} \cdot  \frac{S_{\sigma}}{u_{\sigma}} = \frac{S_{\sigma}}{15k n^{7\epsilon}} . 
\end{equation}

The remainder of this section is organized as follows. 
First, in Section~\ref{sec:pm_inequalities} we give some inequalities for dealing with expressions containing $\pm$ symbols. 
Afterwards, in Sections~\ref{sec:proof:verification:open}--\ref{sec:proof:verification:partial} we verify the trend and boundedness hypotheses of Lemma~\ref{lem:dem} for the open, intermediate and partial triples. 
Finally, in Section~\ref{sec:proof:verification:end} we finish the proof of Lemma~\ref{lem:dem:trajectories} by checking the remaining conditions.

\subsection{Estimates for expressions containing $\pm$ operators} 
\label{sec:pm_inequalities}
The following inequalities can easily be verified using elementary calculus. 
Recall that $a \pm b$ is a shorthand for $\{a + xb : -1 \leq x \leq 1 \}$, where multiple occurrences of $\pm$ are treated independently (see Section~\ref{sec:notation}). 
\begin{lemma}%
\label{lem:pm_inequalities}%
Suppose $0 \leq x \leq 1/2$. Then
\begin{equation}
\label{eq:pm_inverse}
(1 \pm x)^{-1} \subseteq 1 \pm 2 x   . 
\end{equation}
\end{lemma} 
The next lemma provides estimates for products of a special form. 
\begin{lemma}%
\label{lem:pm_product_inequality_extended}%
Suppose that $x,y, f_x, f_y, g, h \geq 0$ and $g \leq 1$. 
Then $f_x + x g\leq h/2$ implies 
\begin{equation}
\label{eq:pm_product_ext1}
( x \pm f_{x} ) (1 \pm g) \subseteq x \pm h  . 
\end{equation}
Furthermore, $x f_y + y f_x + f_x f_y + x y g \leq h/2$ implies  
\begin{equation}
\label{eq:pm_product_ext2}
( x \pm f_{x} ) ( y \pm f_{y} ) (1 \pm g) \subseteq xy \pm h  . 
\end{equation}
\end{lemma} 
\begin{proof}
Using $x,y, f_x, f_y \geq 0$ we see that
\begin{equation}
\label{eq:pm_product_proof_1}
( x \pm f_{x} ) ( y \pm f_{y} ) \subseteq xy \pm (x f_y + y f_x + f_x f_y)  .
\end{equation}
Plugging $y = 1$ and $f_{y}=g$ into \eqref{eq:pm_product_proof_1}, and using $g \leq 1$, $x g \geq 0$ as well as $f_x + x g\leq h/2$, we obtain
\[
( x \pm f_{x} ) ( 1 \pm g ) \subseteq x \pm (x g + f_x + f_x g) \subseteq x \pm 2(f_x + x g) \subseteq x \pm h  ,
\]
which establishes \eqref{eq:pm_product_ext1}. 
Finally, using \eqref{eq:pm_product_proof_1} and plugging $x' = xy$ and $f_x' = x f_y + y f_x + f_x f_y$ together with $f'_x+x'g \leq h/2$ into \eqref{eq:pm_product_ext1} gives \eqref{eq:pm_product_ext2}, which completes the proof.  
\end{proof}

\subsection{Open triples}
\label{sec:proof:verification:open}
In every open triple $(u,v,w) \in \opens$ all pairs are open, i.e., $\{uv,vw,uw\} \subseteq O(i)$.  
We define 
\begin{gather}
\label{def:open_triples:x}
x(t) := {q(t)}^3 = e^{-48 t^{5}}  , \qquad x^{+}(t) := 0  , \qquad x^{-}(t) := 240 t^4 {q(t)}^3 \\
\label{def:open_triples:fx_hx}
f_x(t) := f(t) {q(t)}^2 = e^{(W-32) t^{5}+W t} \quad \text{ and } \quad h_x(t) := f'_x(t)/2  .
\end{gather} 
For every $\sigma = (\Sigma,X)$ with $\Sigma \in \cC$ we set $x_{\sigma}(t) := x(t)$, $y_{\sigma}^{\pm}(t) := x^{\pm}(t)$, $f_{\sigma}(t) := f_x(t)$, $h_{\sigma}(t) := h_x(t)$ and $S_{\sigma} := k^3$. 
Moreover, we define $X^+_{\Sigma}(i) := \opens[i+1] \setminus \opens$ and $X^-_{\Sigma}(i) := \opens \setminus \opens[i+1]$. 
Formally we then set $Y^{\pm}_{(\Sigma,X)}(i) := |X^{\pm}_{\Sigma}(i)|$, but henceforth we work directly with $|X^{\pm}_{\Sigma}(i)|$. 
In the following we check the trend and boundedness hypothesis for the open triples.

\subsubsection{Trend hypothesis} 
\label{sec:proof:verification:open:trend}
Note that whenever $\cE_i \cap \neg\badE \cap \cH_i$ holds we have to prove 
\begin{equation}
\label{eq:open_triples:th_goal}
\EE[|X^{\pm_1}_{\Sigma}(i)| \mid \cF_i ] = \left(x^{\pm_1}(t) \pm \frac{h_x(t)}{s_{o}} \right) \frac{k^3}{n^2p}  .
\end{equation}

\textbf{Triples added in one step.} 
We start by verifying \eqref{eq:open_triples:th_goal} for $X^{+}_{\Sigma}(i)$. 
Clearly, adding an edge to $G(i)$ can not create new open triples. So $|X^+_{\Sigma}(i)| = 0 = x^+(t)$ always holds, which settles this case.

\textbf{Triples removed in one step.} 
Next, we prove \eqref{eq:open_triples:th_goal} for $X^{-}_{\Sigma}(i)$. 
Recall that $e_{i+1} \in O(i)$ is added to $G(i)$. 
Observe that a triple $(u,v,w) \in \opens$ is removed, i.e., not in $\opens[i+1]$, if it either contains $e_{i+1}$, or one of its pairs is in $C_{e_{i+1}}(i)$, which is equivalent to $e_{i+1} \in C_{uv}(i) \cup C_{vw}(i) \cup C_{uw}(i)$. 
Thus, the number of choices for $e_{i+1}$ which remove $(u,v,w)$ from $\opens$ is $|C_{uv}(i) \cup C_{vw}(i) \cup C_{uw}(i)| \pm 3$. 
Recall that the $K_4$-free process chooses the edge $e_{i+1}$ uniformly at random from the open pairs in $G(i)$. 
Thus, whenever $\cE_i \cap \neg\badE \cap \cH_i$ holds we have 
\begin{equation}
\label{eq:open_triples:th0}
\EE[|X^-_{\Sigma}(i)| \mid \cF_i ] = \sum_{(u,v,w) \in \opens} \frac{|C_{uv}(i) \cup C_{vw}(i) \cup C_{uw}(i)| \pm 3}{|O(i)|}   .
\end{equation}
Recall that $\cE_i \cap \neg\badE$ implies $\cG_{i}(\Sigma)$, and thus for calculating the expected one-step changes we may assume that $|\opens|$ satisfies \eqref{eq:lem:dem:trajectories:opens}.  
Moreover, $\cH_i$ implies that the inequalities \eqref{eq:open-estimate}, \eqref{eq:closed-estimate} and \eqref{eq:closed-intersection-estimate} hold. 
In addition, note that $s_{e} = n^{1/12-\epsilon}$ and \eqref{eq:K4-functions-estimates} imply $f(t)/s_e = o(1)$. 
Substituting the former estimates into \eqref{eq:open_triples:th0}, and using $n^{1/6} = \omega(s_e)$, $f(t) \geq 1$ as well as Lemma~\ref{lem:pm_inequalities}, we deduce that 
\begin{equation*}
\label{eq:open_triples:th1}
\begin{split}
\EE[|X^-_{\Sigma}(i)| \mid \cF_i ] & = \frac{( {q(t)}^3	 \pm f_x(t)/s_{o}) k^3 \cdot  [3(40 t^{4}q(t) \pm 9f(t)/s_e)p^{-1} \pm 3 n^{-1/6} p^{-1} \pm 3]}{(1 \pm 3f(t)/s_e)q(t)n^2/2}\\
& \subseteq  (1 \pm 6f(t)/s_e ) \cdot ( {q(t)}^2	 \pm f(t){q(t)}/s_{o} ) \cdot ( 240 t^{4}q(t) \pm 70f(t)/s_e ) \cdot k^3/(n^2p)  . 
\end{split}
\end{equation*}
Therefore the desired bound, i.e., \eqref{eq:open_triples:th_goal} for $X^{-}_{\Sigma}(i)$, follows if 
\begin{equation}
\label{eq:open_triples:th_sufficient_goal}
(1 \pm 6f(t)/s_e) \cdot ( {q(t)}^2	 \pm f(t){q(t)}/s_{o} ) \cdot \left(240 t^{4}q(t) \pm 70f(t)/s_e\right) \subseteq x^-(t) \pm h_x(t)/s_{o}  .
\end{equation}
Now, using $f(t)=o(s_e)$ and Lemma~\ref{lem:pm_product_inequality_extended}, observe that to prove \eqref{eq:open_triples:th_sufficient_goal} it is enough to show 
\begin{equation*}
\label{eq:open_triples:th_sufficient}
140 f(t) {q(t)}^2 s_{o}/s_e + 480 t^4 f(t) {q(t)}^2 + 140 {f(t)}^2 {q(t)}/s_e + 2880 t^4  f(t) {q(t)}^3 s_{o}/s_e \leq h_x(t)  .
\end{equation*}
Note that by \eqref{eq:K4-functions-estimates} all terms involving $s_{e}$ are in fact $o(1)$. So, using $f_x(t)=f(t) {q(t)}^2$ it suffices if 
\begin{equation*}
\label{eq:open_triples:th_sufficient_shorther}
1 + 480 t^4 f_x(t) \leq h_x(t)  ,
\end{equation*}
which is easily seen to be true, since $h_x(t) \geq W t^4 f_x(t) + W/2$ by \eqref{eq:K4-constants:Wepsmu} and \eqref{def:open_triples:fx_hx}. 
To summarize, we have verified the trend hypothesis~\eqref{eq:open_triples:th_goal} for the open triples.

\subsubsection{Boundedness hypothesis} 
\label{sec:proof:verification:open:bounded}
Second, we verify the boundedness hypothesis~\eqref{eq:dem:triples:max_bound} whenever $\cE_i \cap \neg\badE \cap \cH_i$ holds. 
Clearly, no triples are added to  $\opens[i+1]$, and so $|X^+_{\Sigma}(i)| = 0$. 
Suppose that $e_{i+1} \in O(i)$ is added to $G(i)$. 
Recall that a triple $(u,v,w) \in \opens$ is removed, i.e., not in $\opens[i+1]$, if it either contains $e_{i+1}$, or one of its pairs is in $C_{e_{i+1}}(i)$. 
Note that every pair is in at most $k$ triples and, furthermore, that $\cH_i$ implies \eqref{eq:closed-estimate}, which gives $|C_{e_{i+1}}(i)| \leq p^{-1}n^{\epsilon}$. 
Therefore, we deduce that 
\[ |X^-_{\Sigma}(i)| \leq k (1+|C_{e_{i+1}}(i)|) \leq 2 k p^{-1}n^{\epsilon} = o(k^2n^{-7\epsilon})  , \] 
which establishes the boundedness hypothesis for the open triples.

\subsection{Intermediate triples}
\label{sec:proof:verification:intermediate}
Every intermediate triple $(u,v,w) \in \interms$ satisfies $\{uv,vw\} \subseteq O(i)$ and $uw \in E(i)$. 
We define 
\begin{gather}
\label{def:interm_triples:y}
y(t) := 2 t {q(t)}^2 = 2t e^{-32 t^{5}} , \qquad y^{+}(t) := 2 {q(t)}^2 , \qquad y^{-}(t) := 320 t^5 {q(t)}^2 \\
\label{def:interm_triples:fy_hy}
f_y(t) := f(t) {q(t)} = e^{(W-16) t^{5} + W t} \qquad \text{ and } \qquad h_y(t) := f'_y(t)/2  .
\end{gather} 
For every $\sigma = (\Sigma,Y)$ with $\Sigma \in \cC$ we set $x_{\sigma}(t) := y(t)$, $y_{\sigma}^{\pm}(t) := y^{\pm}(t)$, $f_{\sigma}(t) := f_y(t)$, $h_{\sigma}(t) := h_y(t)$ and $S_{\sigma} := k^3p$. 
Similar as for open triples, we define $Y^+_{\Sigma}(i) := \interms[i+1] \setminus \interms$ and $Y^-_{\Sigma}(i) := \interms \setminus \interms[i+1]$. Then we set $Y^{\pm}_{(\Sigma,Y)}(i) := |Y^{\pm}_{\Sigma}(i)|$, but henceforth work directly with $|Y^{\pm}_{\Sigma}(i)|$.

\subsubsection{Trend hypothesis} 
\label{sec:proof:verification:intermediate:trend}
Whenever $\cE_i \cap \neg\badE \cap \cH_i$ holds we have to prove 
\begin{equation}
\label{eq:interm_triples:th_goal}
\EE[|Y^{\pm_1}_{\Sigma}(i)| \mid \cF_i] = \left(y^{\pm_1}(t) \pm \frac{h_y(t)}{s_{o}} \right) \frac{k^3p}{n^2p}  .
\end{equation}

\textbf{Triples added in one step.} 
Note that a triple $(u,v,w) \in \opens$ is added to $\interms[i+1]$, i.e., is in $\interms[i+1]$, if and only if $e_{i+1}=uw$ (because $e_{i+1}=uw$ can not close any of the open pairs $uv$ or $vw$). 
Recall that $\cE_i \cap \neg\badE$ implies $\cG_{i}(\Sigma)$, and thus $|\opens|$ satisfies \eqref{eq:lem:dem:trajectories:opens}. 
Furthermore, whenever $\cH_i$ holds so does $\cG_i$, which implies \eqref{eq:open-estimate}. 
Using that $e_{i+1}$ is chosen uniformly at random from $O(i)$ as well as $f(t) = o(s_e)$ and Lemma~\ref{lem:pm_inequalities}, whenever $\cE_i \cap \neg\badE \cap \cH_i$ holds we deduce that
\begin{equation}
\label{eq:interm_triples:th+}
\begin{split}
\EE[|Y^+_{\Sigma}(i)| \mid \cF_i ] &= \sum_{(u,v,w) \in \opens} \frac{1}{|O(i)|}  \subseteq \frac{({q(t)}^3	 \pm f_x(t)/s_{o}) k^3}{(1 \pm 3f(t)/s_e)q(t)n^2/2}\\
& \subseteq (1 \pm 6f(t)/s_e) \cdot (2{q(t)}^2	 \pm 2f_y(t)/s_{o}) \cdot k^3p/(n^2p)  .
\end{split}
\end{equation} 
Now, using Lemma~\ref{lem:pm_product_inequality_extended} we see that \eqref{eq:interm_triples:th_goal} for $Y^{+}_{\Sigma}(i)$ follows if 
\begin{equation}
\label{eq:interm_triples:th+_sufficient_goal}
4 f_y(t) + 24 f(t) {q(t)}^2 s_{o} / s_e \leq h_y(t)  .
\end{equation}
Using \eqref{eq:K4-constants:Wepsmu} and \eqref{def:interm_triples:fy_hy} the last inequality follows readily by observing that the term involving $s_e$ is $o(1)$.

\textbf{Triples removed in one step.} 
Observe that a triple $(u,v,w) \in \interms$ is removed, i.e., not in $\interms[i+1]$, if either $e_{i+1} \in \{uv,vw\}$, or one of the pairs $uv$ or $vw$ is in $C_{e_{i+1}}(i)$, which is equivalent to $e_{i+1} \in C_{uv}(i) \cup C_{vw}(i)$. 
Using that $e_{i+1}$ is chosen uniformly at random from $O(i)$, we see that 
\begin{equation}
\label{eq:interm_triples:th-}
\EE[|Y^-_{\Sigma}(i)| \mid \cF_i ] = \sum_{(u,v,w) \in \interms} \frac{|C_{uv}(i) \cup C_{vw}(i)| \pm 2}{|O(i)|}   .
\end{equation}
Recall that $\cH_i$ implies \eqref{eq:open-estimate}, \eqref{eq:closed-estimate} and \eqref{eq:closed-intersection-estimate}. 
Furthermore, as argued before, $f(t)/s_e = o(1)$ holds and $\cE_i \cap \neg\badE \cap \cH_i$ implies that $|\interms|$ satisfies \eqref{eq:lem:dem:trajectories:interms}. 
Substituting the former estimates into \eqref{eq:interm_triples:th-}, and using $n^{1/6} = \omega(s_e)$, $f(t) \geq 1$ as well as Lemma~\ref{lem:pm_inequalities}, 
whenever $\cE_i \cap \neg\badE \cap \cH_i$ holds we have
\begin{equation*}
\label{eq:interm_triples:th-1}
\begin{split}
\EE[|Y^-_{\Sigma}(i)| \mid \cF_i ] & = \frac{(2t {q(t)}^2 \pm f_y(t)/s_{o}) k^3p \cdot [2(40 t^{4}q(t) \pm 9f(t)/s_e)p^{-1} \pm n^{-1/6} p^{-1} \pm 2]}{(1 \pm 3f(t)/s_e)q(t)n^2/2}\\
& \subseteq  (1 \pm 6f(t)/s_e) \cdot ( 4 t q(t)	 \pm 2f(t)/s_{o} ) \cdot (8 0 t^{4}q(t) \pm 30f(t)/s_e ) \cdot k^3p/(n^2p)  . 
\end{split} 
\end{equation*}
We intend to show \eqref{eq:interm_triples:th_goal} for $Y^{-}_{\Sigma}(i)$ using Lemma~\ref{lem:pm_product_inequality_extended}. 
Similar as for the removed open triples, by writing down the assumptions of \eqref{eq:pm_product_ext2}, multiplying with $2 s_{o}$ and then noticing that all terms containing $s_{e}$ are $o(1)$, we see that it suffices if
\begin{equation*}
\label{eq:interm_triples:th_-sufficient}
1 + 320 t^4 f_y(t) \leq h_y(t)  ,
\end{equation*}
which clearly holds  by \eqref{eq:K4-constants:Wepsmu} and \eqref{def:interm_triples:fy_hy}. 
This establishes the trend hypothesis~\eqref{eq:interm_triples:th_goal}.

\subsubsection{Boundedness hypothesis} 
\label{sec:proof:verification:intermediate:bounded}
Second, we verify the boundedness hypothesis~\eqref{eq:dem:triples:max_bound} whenever $\cE_i \cap \neg\badE \cap \cH_i$ holds. 
Note that a triple $(u,v,w) \in \opens$ is added to $\interms[i+1]$, i.e., is in $\interms[i+1]$, if $e_{i+1}=uw$. As every pair is in at most $k$ triples, we thus obtain
\[ |Y^+_{\Sigma}(i)| \leq k = o(k^2pn^{-7\epsilon})  . \] 
Observe that a triple $(u,v,w) \in \interms$ is only removed, i.e., not in $\interms[i+1]$, if $e_{i+1} \in \{uv,vw\}$ or $\{uv,vw\} \cap C_{e_{i+1}}(i) \neq \emptyset$ holds. 
Since $\cH_i$ implies \eqref{eq:closed-estimate}, we have $|C_{e_{i+1}}(i)| \leq p^{-1}n^{\epsilon}$. 
Furthermore, whenever $\neg\badE$ holds so does $\neg\badE[1,i]$, and thus every vertex in $K$ has at most $kpn^{5 \epsilon}$ neighbours in $K$. 
Therefore, using that $u$ and $w$ are neighbours for every $(u,v,w) \in \interms$, we deduce that 
\[ |Y^-_{\Sigma}(i)| \leq (1+|C_{e_{i+1}}(i)|) kpn^{5 \epsilon} \leq k n^{10 \epsilon} = o(k^2pn^{-7\epsilon})  , \] 
which establishes the boundedness hypothesis for the intermediate triples.

\subsection{Partial triples} 
\label{sec:proof:verification:partial}
Every partial triple $(u,v,w) \in \parts$ satisfies $uv \in O(i) \cup C(i)$ and $\{uw,vw\} \subseteq E(i)$. 
Recall that $\parts$ is defined inductively in a manner that ensures that every $uv$ with $u \in A$ and $v \in B$ is contained in at most one of the triples in $\parts$. 
We define 
\begin{gather}
\label{def:part_triples:z}
z(t) := 2 t^2 q(t) = 2 t^2 e^{-16 t^{5}}  , \qquad z^{+}(t) := 4 t q(t)  , \qquad z^{-}(t) := 160 t^6 q(t) \\
\label{def:part_triples:fz_hz}
f_z(t) := f(t) = e^{(t^{5}+t)W} \qquad \text{ and } \qquad h_z(t) := f'_z(t)/2 = f'(t)/2  .
\end{gather} 
For every $\sigma = (\Sigma,Z)$ with $\Sigma \in \cC$ we set $x_{\sigma}(t) := z(t)$, $y_{\sigma}^{\pm}(t) := z^{\pm}(t)$, $f_{\sigma}(t) := f_z(t)$, $h_{\sigma}(t) := h_z(t)$ and $S_{\sigma} := k^3p^2$. 
Similar as for open and intermediate triples, we define $Z^+_{\Sigma}(i) := \parts[i+1] \setminus \parts$ and $Z^-_{\Sigma}(i) := \parts \setminus \parts[i+1]$. 
Then we set $Y^{\pm}_{(\Sigma,Z)}(i) := |Z^{\pm}_{\Sigma}(i)|$, but work directly with $|Z^{\pm}_{\Sigma}(i)|$.

\subsubsection{Trend hypothesis} 
\label{sec:proof:verification:partial:trend}
As usual, whenever $\cE_i \cap \neg\badE \cap \cH_i$ holds we have to prove 
\begin{equation}
\label{eq:part_triples:th_goal}
\EE[|Z^{\pm_1}_{\Sigma}(i)| \mid \cF_i ] = \left(z^{\pm_1}(t) \pm \frac{h_z(t)}{s_{o}} \right) \frac{k^3p^2}{n^2p}  .
\end{equation}

\textbf{Triples added in one step.} 
Recall that a triple $(u,v,w) \in \interms$ is added to $\parts[i+1]$, i.e., is in $\parts[i+1]$, if $e_{i+1}=vw$, $uv \notin C_{vw}(i)$ and there exists no $w' \in C$ such that $(u,v,w') \in \parts$. 
First we determine all $(u,v,w) \in \interms$ that satisfy $uv \in C_{vw}(i)$. 
Let $C_{\Sigma}(i)$ denote all such triples. 
Recall that $\Xi_{\Sigma}(i)$ contains all quadruples $(u,v,w,z) \in A \times B \times C \times [n]$ which are as in Figure~\ref{fig:quadruple:Xi}, i.e., with $z \notin \{u,v,w\}$ and $\{uw,zu,zv,zw\} \subseteq E(i)$. 
Observe that for every $(u,v,w) \in C_{\Sigma}(i)$ there exists $z \in [n]$ such that $(u,v,w,z) \in \Xi_{\Sigma}(i)$, see also Figure~\ref{fig:triples} on page~\pageref{fig:triples}. 
Therefore 
\begin{equation}
\label{eq:part_triples:Ci_Xi}
|C_{\Sigma}(i)| \leq |\Xi_{\Sigma}(i)|  . 
\end{equation}
Note that whenever $\cH_i$ holds so does $\cQ_i$ (defined in Lemma~\ref{lem:deletion_lemma_subgraph_count_K3edge}); hence for $r=k$ there exists a set $E_0 \subseteq [n] \times K$ of pairs with $|E_0| \leq 20 \epsilon^{-1}k$ such that there are at most $k^{3}np^{4} n^{10 \epsilon}$ quadruples $(u,v,w,z) \in \Xi_{\Sigma}(i)$ with $\{uw, zu, zv, zw\} \subseteq E(i)\setminus E_0$. 
Furthermore, as $\neg\badE$ holds, by $\neg\badE[3,i]$ we know that every $xy \in E_0$ satisfies $xy \in \{uw, zu, zv, zw\}$ for at most $k^2pn^{-15\epsilon}$ quadruples $(u,v,w,z) \in \Xi_{\Sigma}(i)$. So, together with \eqref{eq:part_triples:Ci_Xi} we deduce that 
\begin{equation}
\label{eq:part_triples:Ci} 
|C_{\Sigma}(i)| \leq  |\Xi_{\Sigma}(i)| \leq k^{3}np^{4} n^{10 \epsilon} + 20 \epsilon^{-1}k^3pn^{-15\epsilon} \leq k^3pn^{-10\epsilon}  .
\end{equation} 
Next, let $D_{\Sigma}(i)$ contain all triples $(u,v,w) \in \interms$ for which there exists $w' \in C$ such that $(u,v,w') \in \parts$. 
Whenever $\neg\badE$ holds so does $\neg\badE[1,i]$, and thus every vertex in $K$ has at most $kpn^{5 \epsilon}$ neighbours in $K$.
Since $w$ must be a neighbour of $u$ for every $(u,v,w) \in \interms$, we deduce that $|D_{\Sigma}(i)| \leq |\parts| \cdot kpn^{5 \epsilon}$. 
As $\cE_i \cap \neg\badE$ implies $\cG_{i}(\Sigma)$, we know that $|\parts|$ satisfies \eqref{eq:lem:dem:trajectories:parts}, which gives $|\parts| \leq k^3p^2 n^{\epsilon}$. 
So, we see that 
\begin{equation}
\label{eq:part_triples:Di} 
|D_{\Sigma}(i)| \leq k^4p^3n^{6 \epsilon} = o(k^3pn^{-10\epsilon})  .   
\end{equation} 
To summarize, $(u,v,w) \in \interms$ is added to $\parts[i+1]$ if and only if $e_{i+1}=vw$ and $(u,v,w) \notin  C_{\Sigma}(i) \cup D_{\Sigma}(i)$. 
As noted before, $\cE_i \cap \neg\badE \cap \cH_i$ implies that $O(i)$ and $\interms$ satisfy \eqref{eq:open-estimate} and \eqref{eq:lem:dem:trajectories:interms}, respectively. 
So, using that $e_{i+1}$ is chosen uniformly at random from $O(i)$ as well as \eqref{eq:part_triples:Ci}, \eqref{eq:part_triples:Di}, $n^{10\epsilon} = \omega(s_{o})$, $f_y(t) \geq 1$, $f(t) = o(s_e)$ and Lemma~\ref{lem:pm_inequalities}, whenever $\cE_i \cap \neg\badE \cap \cH_i$ holds we have 
\begin{equation*}
\label{eq:part_triples:th+_0}
\begin{split}
\EE[|Z^+_{\Sigma}(i)| \mid \cF_i ] &= \sum_{(u,v,w) \in \interms \setminus [C_{\Sigma}(i) \cup D_{\Sigma}(i)]} \frac{1}{|O(i)|} \subseteq \frac{(2t{q(t)}^2 \pm f_y(t)/s_{o}) k^3p \pm 2k^3pn^{-10\epsilon}}{(1 \pm 3f(t)/s_e)q(t)n^2/2}\\
& \subseteq (1 \pm 6f(t)/s_e) \cdot (4t{q(t)} \pm 4f(t)/s_{o}) \cdot k^3p^2/(n^2p)  .
\end{split}
\end{equation*} 
Now, using Lemma~\ref{lem:pm_product_inequality_extended} we see that \eqref{eq:part_triples:th_goal} for $Z^{+}_{\Sigma}(i)$ follows if 
\begin{equation*}
\label{eq:part_triples:th+_sufficient_goal}
8 f(t) + 48 t  f(t) q(t) s_{o} / s_e \leq h_z(t)  .
\end{equation*}
Using \eqref{eq:K4-constants:Wepsmu} and \eqref{def:part_triples:fz_hz} the last inequality follows readily by observing that the term involving $s_e$ is $o(1)$.

\textbf{Triples removed in one step.} 
Recall that a triple $(u,v,w) \in \parts$ is not always removed if $e_{i+1}=uv$ or $e_{i+1} \in C_{uv}(i)$ holds (since it can be ignored). 
For estimating the expected number of removed triples we now derive sufficient and necessary conditions for a triple to be removed and we start with a sufficient condition. 
To this end we first determine the triples which might be ignored because of (I2).  
Let $P_{\Sigma}(i)$ contain all open pairs $xy \in O(i)$ with $\min\{|\Gamma(x) \cap \Gamma(y) \cap A|,|\Gamma(x) \cap \Gamma(y) \cap B|\} \geq kpn^{-20\epsilon}$ in $G(i)$. 
As $\neg\badE$ holds, from $\neg\badE[2,i]$ we deduce that $|P_{\Sigma}(i)| \leq k n^{-20 \epsilon}$. 
Now let $I_{\Sigma}(i)$ denote all triples $(u,v,w) \in \parts$ such that in $G(i)$ we have $\{u,v\} \subseteq \Gamma(x) \cap \Gamma(y)$ for some $xy \in P_{\Sigma}(i)$. 
Recall that every $uv$ with $u \in A$ and $v \in B$ is contained in at most one triple in $\parts$.
Furthermore, $\cH_i$ implies \eqref{eq:codegree-estimate}, and so  $|\Gamma(x) \cap \Gamma(y)| \leq (\log n) n p^2$ for every $xy \in P_{\Sigma}(i)$. 
Thus, 
using $|P_{\Sigma}(i)| \leq k n^{-20 \epsilon}$, we obtain 
\begin{equation}
\label{eq:part_triples:Ii}
|I_{\Sigma}(i)| \leq |P_{\Sigma}(i)| \cdot [(\log n) n p^2]^2 \leq k^3 p^{2} n^{-15 \epsilon}  .
\end{equation}
Note that (R2) can only fail if $(u,v,w) \in I_{\Sigma}(i)$. So, if $(u,v,w) \in \parts$ satisfies $e_{i+1} \in C_{uv}(i)$ and $e_{i+1} \cap uv = \emptyset$, then $(u,v,w) \notin I_{\Sigma}(i)$ is a sufficient condition for (R2) to hold.

Next we derive a sufficient condition for (R3a). 
To this end let $L_{\Sigma}(i)$ contain all vertices $y \in [n]$ which satisfy $|\Gamma(y) \cap K| \geq p^{-1}n^{-15\epsilon}$ in $G(i)$. 
Recall that whenever $\cH_i$ holds so does $\cN_i$ (defined in Lemma~\ref{lem:large_degree_bounded}). 
One can check that $\cN_i$ implies, say, $|L_{\Sigma}(i)| \leq kp n^{20 \epsilon}$. 
For every triple $(u,v,w) \in \parts$ we then set $L_{uv,\Sigma}(i) := \{u,v\} \times L_{\Sigma}(i)$, and thus, using $s_{e} = n^{1/12-\epsilon}$, we see that
\begin{equation}
\label{eq:part_triples:Luvi}
|L_{uv,\Sigma}(i)| = 2 |L_{\Sigma}(i)| \leq kp n^{25 \epsilon} = o\left((s_{e} p)^{-1}\right) .
\end{equation}
Note that (R3a) can only fail if $e_{i+1} \in L_{uv,\Sigma}(i)$. So,  if $(u,v,w) \in \parts$ satisfies $e_{i+1} \in C_{uv}(i)$ and $e_{i+1} \cap uv \neq \emptyset$, then $e_{i+1} \notin L_{uv,\Sigma}(i)$ is a sufficient condition for (R3a) to hold.

To summarize, if $(u,v,w) \in \parts \setminus I_{\Sigma}(i)$ satisfies $e_{i+1} \in C_{uv}(i) \setminus L_{uv,\Sigma}(i)$, then either (R2) or (R3a) holds and hence $(u,v,w)$ is removed. 
Clearly, a necessary condition for $(u,v,w) \in \parts$ being removed is $e_{i+1} \in C_{uv}(i) \cup \{uv\}$. 
So, since $e_{i+1}$ is chosen uniformly at random from $O(i)$, we have 
\begin{equation}
\label{eq:part_triples:th-_0}
\EE[|Z^-_{\Sigma}(i)| \mid \cF_i ] = \sum_{(u,v,w) \in \parts} \frac{|C_{uv}(i)| \pm | L_{uv,\Sigma}(i)| \pm 1}{|O(i)|} \pm \sum_{(u,v,w) \in I_{\Sigma}(i)} \frac{|C_{uv}(i)|}{|O(i)|}  .
\end{equation} 
As argued before, $|\parts|$ satisfies \eqref{eq:lem:dem:trajectories:parts} whenever $\cE_i \cap \neg\badE$ holds. Furthermore, $\cH_i$ implies that \eqref{eq:open-estimate} and \eqref{eq:closed-estimate} hold. 
Substituting the former estimates and \eqref{eq:part_triples:Ii}, \eqref{eq:part_triples:Luvi} into \eqref{eq:part_triples:th-_0}, and using $n^{15 \epsilon} = \omega(s_{o})$, $f(t) \geq 1$ and Lemma~\ref{lem:pm_inequalities}, 
whenever $\cE_i \cap \neg\badE \cap \cH_i$ holds we deduce that 
\begin{equation*}
\label{eq:part_triples:th-_1}
\begin{split}
\EE[|Z^-_{\Sigma}(i)| \mid \cF_i ] & = \frac{[(2 t^{2}q(t) \pm f(t)/s_{o})k^3p^2 \pm  k^3 p^{2} n^{-15 \epsilon}] \cdot [(40 t^{4}q(t) \pm 9 f(t)/s_{e})p^{-1} \pm (s_{e} p)^{-1} \pm 1]}{(1 \pm 3f(t)/s_e)q(t)n^2/2}\\
& \subseteq (1 \pm 6f(t)/s_e) \cdot (2 t^{2}q(t) \pm 2f(t)/s_{o}) \cdot [80 t^{4} \pm 30 f(t)/(s_{e}q(t))] \cdot k^3 p^{2}/(n^2p)  .
\end{split}
\end{equation*}
We intend to show \eqref{eq:part_triples:th_goal} for $Z^{-}_{\Sigma}(i)$ using Lemma~\ref{lem:pm_product_inequality_extended}. 
Similar as for the removed open and intermediate triples, by writing down the assumptions of \eqref{eq:pm_product_ext2} and then noticing that all terms containing $s_{e}$ are negligible, we see that it suffices if
\begin{equation*}
\label{eq:part_triples:th--sufficient}
1 + 320 t^4 f(t) \leq h_z(t)  ,
\end{equation*}
which clearly holds  by \eqref{eq:K4-constants:Wepsmu} and \eqref{def:part_triples:fz_hz}. 
This establishes the trend hypothesis~\eqref{eq:part_triples:th_goal}.

\subsubsection{Boundedness hypothesis} 
\label{sec:proof:verification:partial:bounded}
In this section we verify the boundedness hypothesis~\eqref{eq:dem:triples:max_bound} whenever $\cE_i \cap \neg\badE \cap \cH_i$ holds. 
Note that for a triple $(u,v,w) \in \interms$ to be added to $\parts[i+1]$ the conditions $e_{i+1}=vw$ and $u \in \Gamma(w)$ are necessary. 
Furthermore, whenever $\neg\badE$ holds so does $\neg\badE[1,i]$, and thus every vertex in $K$ has at most $kpn^{5 \epsilon}$ neighbours in $K$. 
Therefore 
\[ |Z^+_{\Sigma}(i)| \leq kpn^{5 \epsilon} = o(k^2p^2n^{-7\epsilon})  . \]

Recall that a triple $(u,v,w) \in \parts$ is removed, i.e., not in $\parts[i+1]$, according to different rules. 
In the following we bound the total number of triples removed in one step by each rule. 
As the inductive definition of $\parts$ ensures that every $uv$ with $u \in A$ and $v \in B$ is contained in at most one of the triples in $\parts$, it clearly suffices to bound the number of corresponding pairs $uv$ with $u \in A$ and $v \in B$ that are removed by $e_{i+1}=xy$.  
With $e_{i+1}=xy$ given, we need to consider pairs $uv$ in three different relations to $xy$; these were called cases~$1$--$3$ in Section~\ref{sec:main-proof:configs_variables}. 
In case~$1$ we have $uv=e_{i+1}$, and so, given $e_{i+1}$, at most one triple is removed under case~$1$.

For $e_{i+1} =xy$ the rule (R2) only removes triples $(u,v,w) \in \parts$ with $u \in \Gamma(x) \cap \Gamma(y) \cap A$ and $v \in \Gamma(x) \cap \Gamma(y) \cap B$. 
Because $\cH_i$ holds, all codegrees are at most $(\log n)np^2$ by \eqref{eq:codegree-estimate}.  
Applying the bound in (R2) to bound the number of possibilities for $u$ or $v$ as appropriate, and then using the codegree to bound the number of choices for the other, this rule removes at most  $kpn^{-20 \epsilon} \cdot (\log n)np^2 \leq knp^3n^{-15 \epsilon}$ triples.

Every triple removed under case~$3$ satisfies $uv \cap e_{i+1} = x$ and $\{u,v\} \setminus \{x\} \subseteq \Gamma(y)$, where $e_{i+1} =xy$. 
Hence (R3a) removes at most $2p^{-1}n^{-15\epsilon}$ triples (the factor of two accounts for the different choices of $x$). 
The last rule (R3b) only removes triples $(u,v,w) \in \parts$ for which there exists $z \in \Gamma(x) \cap \Gamma(y)$ such that $\{u,v\} \setminus \{x\} \subseteq \Gamma(y) \cap \Gamma(z) \cap K$ and $|\Gamma(y) \cap \Gamma(z) \cap K| \leq kpn^{-20 \epsilon}$. 
So, using the codegree to bound the number of choices for $z$ and, given $z$, the above bound for the number of vertices in $\Gamma(y) \cap \Gamma(z) \cap K$, this rule removes at most $2 (\log n)np^2 \cdot kpn^{-20 \epsilon} \leq knp^3n^{-15 \epsilon}$ triples.

Putting it all together, we deduce that 
\[ |Z^-_{\Sigma}(i)| \leq 1 + 2p^{-1}n^{-15\epsilon} + 2knp^3n^{-15 \epsilon} \leq k^2 p^2n^{-10 \epsilon}  , \]
which clearly establishes the boundedness hypothesis for the partial triples.

\subsection{Finishing the trajectory verification} 
\label{sec:proof:verification:end}
In this section we verify the remaining conditions of the differential equation method (Lemma~\ref{lem:dem}).

\textbf{Initial conditions.} 
Using \eqref{def:open_triples:x}, \eqref{def:interm_triples:y} and \eqref{def:part_triples:z}, for all $\Sigma \in \cC$ it is easy to see that $\opens[0] = k^3 = x(0) k^3$, $\interms[0] = 0 = y(0) k^3p$ and $\parts[0] = 0 = z(0) k^3p^2$ hold, which establishes \eqref{eq:lem:dem:initial_condition}.

\textbf{Bounded number of configurations and variables.} 
By construction we have 
\begin{equation*}
\label{eq:bound:nr_configs}
|\cC| \leq \binom{n}{u} 4^u \leq n^{2 u} = e^{2u \log n}  ,
\end{equation*}
which together with $|\cV|=3$ and $u_{\sigma}=un^{\epsilon}$ clearly establishes \eqref{eq:lem:dem:bounded_parameters}.

\textbf{Additional technical assumptions and the function $f_{\sigma}(t)$.}  
Recall that $m=\mu n^2p\sqrt[5]{\log n}$, $s=n^2p$ and $u=\gamma \mu np\sqrt[5]{\log n}$. %, which readily gives 
Now, using \eqref{def:dem:parameters} it is easy to see that \eqref{eq:lem:dem:technical_assumptions:sm} holds, with room to spare. 
Next, using \eqref{eq:K4-constants:Wepsmu}, \eqref{def:open_triples:x}, \eqref{def:interm_triples:y} and \eqref{def:part_triples:z}, elementary calculus shows that for all $\sigma \in \cC \times \cV$ we have 
\[
x'_{\sigma}(t) = y^+_{\sigma}(t) - y^-_{\sigma}(t) , \qquad \sup_{0 \leq t \leq m/s} y_{\sigma}^{\pm }(t) \leq n^\epsilon = \lambda_{\sigma}  \qquad \text{and} \qquad \int_0^{m/s} |x_{\sigma}''(t)| \ dt \leq n^\epsilon = \lambda_{\sigma}   ,
\]
with plenty of room to spare for large $n$. 
Recall that for all $\sigma \in \cC \times \cV$ we have $h_{\sigma}(t) = f'_{\sigma}(t)/2$ and $f_{\sigma}(t) = f(t) {q(t)}^{\iota}$, where $\iota \in \{0,1,2\}$. Hence, using $f_{\sigma}(0)=1=\beta_{\sigma}$, we see that 
\[
f_{\sigma}(t) = 2 \int_{0}^{t} h_{\sigma}(\tau) \ d\tau + f_{\sigma}(0) = 2 \int_{0}^{t} h_{\sigma}(\tau) \ d\tau + \beta_{\sigma}  .
\] 
Note that $h_{\sigma}(0) = W/2 \leq s_{\sigma} \lambda_{\sigma}$ and $h'_{\sigma}(t) \geq 0$. In addition, observe that $h'_{\sigma}(t)$ is bounded by some constant for, say, $t \leq 30$. 
For larger $t$, we have $t^8 \leq e^t$, which implies, say, $h'_{\sigma}(t) \leq W^3{f(t)}^2$. Putting things together, using elementary calculus as well as \eqref{eq:K4-parameters} and \eqref{eq:K4-functions-estimates}, for $n$ large enough we obtain 
\[
\int_{0}^{m/s} |h'_{\sigma}(t)| \ dt \leq \int_{0}^{30} h'_{\sigma}(t) \ dt + \int_{30}^{m/s} W^3{f(t)}^2 \ dt \leq O(1) + m/s \cdot W^3 {f(m/s)}^{2} \leq n^{3\epsilon} = s_{\sigma} \lambda_{\sigma} . 
\]
To summarize, we showed that \eqref{eq:lem:dem:derivative} as well as the additional technical assumptions \eqref{eq:lem:dem:technical_assumptions:sm}--\eqref{eq:lem:dem:technical_assumptions:h} hold, and this completes the proof of Lemma~\ref{lem:dem:trajectories}. \qed

\section{`Very good' configurations exist for every subset} 
\label{sec:good_configurations_exist}
In this section we prove Lemma~\ref{lem:dem:config}, which is our main combinatorial statement. 
Since this lemma is purely deterministic, it suffices to prove its claim for fixed $U \in \binom{[n]}{u}$ and $G(i)$ satisfying $\cH_i$. 
We proceed in two steps, always tacitly assuming that $n$ is sufficiently large whenever necessary. 
First, in Section~\ref{sec:good_configurations_exist:find_config} we pick a `nice' configuration $\Sigma^*=(U,\Pi)$.
Afterwards, in Sections~\ref{sec:good_configuration}--\ref{sec:good_configuration:few_ignored} we verify the claimed properties using the density arguments of Section~\ref{sec:density}. 
Perhaps surprisingly, for showing that $\partCs \setminus \partOCs$ is small we do \emph{not} need to know all $G(i')$ with $i' \leq i$; our proof only uses the trivial inclusion $G(i') \subseteq G(i)$.

In the remainder of this section $\Gamma(v)$ denotes the neighbourhood of $v$ in $G(i)$, unless otherwise stated.  
Furthermore, will use without further reference that $G(i)$ satisfies 
\[
\cH_i  \subseteq \cD_i \cap \cG_i \cap \cM_i \cap \cN_i ,
\]
where $\cG_i$ is defined in Theorem~\ref{thm:BohmanKeevash2010H}, and $\cD_i$, $\cM_i$ and $\cN_i$ are defined in Lemmas~\ref{lem:edges_bounded}--\ref{lem:codegree_into_set_bounded}. 
In particular, as $\cG_i$ implies \eqref{eq:codegree-estimate}, in $G(i)$ all codegrees are bounded by $(\log n) n p^2$. 
For the subsequent calculations it may be useful to keep in mind that $p=n^{-2/5}$, $k=n^{3/5 \pm \epsilon}$ and $u=\Theta(k)$.

\subsection{Finding $\Sigma^*=(U,\Pi)$ in $G(i)$} 
\label{sec:good_configurations_exist:find_config}
In this section we pick $\Sigma^*=(U,\Pi)$ by only considering the edges of $G(i)$; 
in anticipation of our later arguments we also construct an associated set of vertices $I$ with additional properties. 
Set 
\begin{equation}
\label{def:LU}
L := \{ v \in [n] \;:\; |\Gamma(v) \cap U| \geq kp n^{5 \epsilon}\}  ,
\end{equation}
so $L$ contains those vertices which have `too many' neighbours in $U$. 
Observe that $\cN_i$ implies, say, $|L| \leq p^{-1} = o(k)$. 
Henceforth we assume that the vertices $v_1, \ldots, v_n \in [n]$ are arranged in decreasing order 
wrt.\  their number of neighbours in $U$, i.e., we have 
\begin{equation}
\label{eq:vertex_ordering}
|\Gamma(v_1) \cap U| \geq |\Gamma(v_2) \cap U| \geq \cdots \geq  |\Gamma(v_j) \cap U| \geq \cdots \geq  |\Gamma(v_n) \cap U|  .
\end{equation}
We greedily choose first $\ell_A$, then $\ell_B$ and finally $\ell_C$ such that they are the smallest indices for which 
\begin{gather*}
N_A := \bigcup_{1 \leq j \leq \ell_A}  \big(\Gamma(v_j) \cap U\big)  , \qquad N_B  := \bigcup_{\ell_A < j \leq \ell_B}  \big(\Gamma(v_j) \cap U\big) \setminus N_A \qquad \text{ and } \\
N_C  := \bigcup_{\ell_B < j \leq \ell_C}  \big(\Gamma(v_j) \cap U\big) \setminus \big(N_A \cup N_B\big)
\end{gather*} 
each have cardinality at least $2k$, where we set the corresponding index to $\infty$ if this is not possible. 
Recall that $k=u/15 = \gamma/15 \cdot np t_{\max}$ by \eqref{eq:K4-parameters3} and $\gamma \geq 150$ by \eqref{eq:K4-parameters2}.  
Furthermore, since $\cG_i$ holds, by \eqref{eq:degree-estimate} the maximum degree in $G(i)$ is at most $3 np t_{\max} \leq k/3$. 
With this in mind, we deduce that the size of $N_A$, $N_B$ and $N_C$ is each at most $2k + k/3 = 7k/3$. 
Using $|L| =o(k)$ this implies 
\begin{equation}
\label{eq:size_neighbourhoods}
|N_A \cup N_B \cup N_C \cup L| \leq 7k + o(k) \leq u/2  .
\end{equation} 
Now we pick a partition $\Pi=(A,B,C)$ as follows. 
If $\ell_C=\infty$ or $\ell_C > kp n^{-5\epsilon}$, define $I := \emptyset$ and choose arbitrary disjoint sets of size $k$ satisfying 
\[ A,B,C \subseteq U \setminus \big(N_A \cup N_B \cup N_C \cup L\big)  ,  \]
which is possible by \eqref{eq:size_neighbourhoods}. 
Otherwise $\ell_C \leq kp n^{-5 \epsilon}$ holds. 
In this case, we define $\Gamma(S) := \bigcup_{v \in S} \Gamma(v)$ for every vertex set $S$, and set $I_A := \{v_1, \ldots, v_{\ell_A}\}$, $I_{BC} := \{v_{\ell_A+1}, \ldots, v_{\ell_C}\}$ and $I := I_A \cup I_{BC}$. 
Since by \eqref{eq:codegree-estimate} all codegrees are bounded by $(\log n) n p^2$, using $\ell_A \leq \ell_C$ we see that  
\begin{equation}
\label{eq:size_neighbourhoods:IB}
|\Gamma(I_{BC}) \cap N_A| \leq |\Gamma(I_{BC}) \cap \Gamma(I_A)| \leq \ell_C \cdot \ell_A \cdot (\log n) n p^2  \leq k n^{-5\epsilon}  . 
\end{equation}
Now we choose arbitrary sets of size $k$ satisfying 
\[ 
A \subseteq N_A \setminus \big(\Gamma(I_{BC}) \cup L\big)  , \quad   B \subseteq N_B \setminus L \quad \text{ and } \quad C \subseteq N_C \setminus L  , 
\] 
which is possible by \eqref{eq:size_neighbourhoods:IB} and $|L|=o(k)$. Note that $A$, $B$ and $C$ are disjoint. 
Finally, we set $\Pi:=(A,B,C)$, $\Sigma^*:=(U,\Pi)$ and write $K=K(\Sigma^*) := A \cup B \cup C$.

We remark that the above construction borrows some ideas from Bohman~\cite{Bohman2009K3}, but differs in many details. 
An important difference to~\cite{Bohman2009K3} is that we may not assume that $K$ is an independent set. 
One of the new ingredients here is the removal of the high-degree vertices contained in $L$, which implies an upper bound on $|\Gamma(v) \cap K|$ for every vertex $v \in K$. 
Furthermore, in contrast to~\cite{Bohman2009K3} our construction also allows us to reason about $|\Gamma(v) \cap K|$ for every $v \in U \setminus K$, cf.\ Section~\ref{sec:good_configuration:neighbourhoods}. 
In the following sections we argue that $\Sigma^*$ has the properties claimed by Lemma~\ref{lem:dem:config}.

\subsection{Bounding the size of certain neighbourhoods}
\label{sec:good_configuration:neighbourhoods}
In this section we collect some bounds on the number of neighbours in $A$, $B$ or $K=A \cup B \cup C$, which will be used extensively in the sequel. 
We claim that in $G(i)$ every vertex $v \in [n] \setminus I$ satisfies 
\begin{equation}
\label{eq:vnotinI:degree_bounded}
|\Gamma(v) \cap K| \leq p^{-1} n^{10 \epsilon}  ,
\end{equation}
and, furthermore, that every vertex $v \in I$ satisfies 
\begin{equation}
\label{eq:vinI:neighbourhood}
\min\left\{|\Gamma(v) \cap A|, |\Gamma(v) \cap B|\right\} = 0  . 
\end{equation}

First, we consider the case $\ell_C \leq kp n^{-5 \epsilon}$. Every $v \in I$ clearly satisfies \eqref{eq:vinI:neighbourhood}, since by construction $\Gamma(v) \cap K \subseteq A$ or $\Gamma(v) \cap K \subseteq B \cup C$. 
Similar as in~\cite{Bohman2009K3}, using the codegree bound $(\log n) n p^2$, for every $v \notin I = \{v_1, \ldots, v_{\ell_C}\}$ we establish \eqref{eq:vnotinI:degree_bounded} as follows: 
\[
|\Gamma(v) \cap K | \leq \sum_{1 \leq j \leq \ell_C} |\Gamma(v) \cap \Gamma(v_j)| \leq \ell_C (\log n) n p^2 \leq p^{-1}  .  
\]
Otherwise $\ell_C=\infty$ or $\ell_C > kp n^{-5 \epsilon}$ holds. 
Since in this case $I = \emptyset$, we have to show that \eqref{eq:vnotinI:degree_bounded} holds for all $v \in [n]$. 
If $\ell_C=\infty$, then all vertices satisfy $|\Gamma(v) \cap K| = 0$. 
Thus we may assume that $kp n^{-5 \epsilon} < \ell_C < \infty$ holds. 
The following argument is based on an idea of Bohman~\cite{Bohman2009K3}. The important difference here is that our conclusion also holds for the vertices in $U$. 
Set $R := \{v_{\ell_C+1}, \ldots, v_{n}\}$, and note that all vertices $v \notin R$ satisfy $|\Gamma(v) \cap K| = 0$ since $\Gamma(v) \cap U \subseteq N_A \cup N_B \cup N_C$. 
Now, due to \eqref{eq:vertex_ordering} and $K \subseteq U$, to prove that \eqref{eq:vnotinI:degree_bounded} holds for all $v \in R$, it is enough to show $|\Gamma(v_{\ell}) \cap U| \leq p^{-1} n^{10 \epsilon}$ for $\ell := kp n^{-5 \epsilon}$.  
Set $H := \{v_1, \ldots, v_{\ell}\}$.  
On the one hand, using \eqref{eq:vertex_ordering} we have 
\[ 2 e(H,U) \geq \sum_{1 \leq j \leq \ell} |\Gamma(v_{j}) \cap U| \geq kp n^{-5 \epsilon} |\Gamma(v_{\ell}) \cap U|  . \] 
On the other hand, since $G(i)$ satisfies $\cD_i$, using $|H| = \ell \leq p^{-1} = o(k)$ and $|U|=15k$ we have, say, $e(H,U) \leq kn^{3\epsilon}$. 
Putting things together, we deduce that $|\Gamma(v_{\ell}) \cap U| \leq p^{-1} n^{10 \epsilon}$, with room to spare. 
As explained, this completes the proof of \eqref{eq:vnotinI:degree_bounded} and \eqref{eq:vinI:neighbourhood}.

\subsection{The configuration $\Sigma^*$ is good}  
\label{sec:good_configuration} 
In this section we show that $\neg\badEC \;=\; \neg\badEC[1,i] \cap \neg\badEC[2,i] \cap \neg\badEC[3,i]$ holds, which by monotonicity (see Section~\ref{sec:main-proof:bad_high_probability_events}) implies $\neg\badEC[i-1]$.  
Observe that the maximum degree inside $K$ is at most $kpn^{5\epsilon}$ since $K \cap L = \emptyset$, which establishes $\neg\badEC[1,i]$.

Turning to $\neg\badEC[2,i]$, recall that we need to show that the number of pairs $xy \in \binom{[n]}{2}$ with 
\begin{equation}
\label{eq:badEventB2i:pair_constraint}
\min\{|\Gamma(x) \cap \Gamma(y) \cap A|,|\Gamma(x) \cap \Gamma(y) \cap B|\} \geq kpn^{-20\epsilon} 
\end{equation}
 is less than $kn^{-20\epsilon}$. 
By \eqref{eq:vinI:neighbourhood} we may restrict our attention to pairs which contain no vertices from $I$. 
Now, let $H$ contain all vertices $x \in [n] \setminus I$ with $|\Gamma(x) \cap K| \geq kpn^{65 \epsilon}$. 
Using $\cN_i$ we obtain, say, $|H| \leq p^{-1}n^{-60\epsilon}$. 
In the following we use a case distinction to count all pairs that satisfy \eqref{eq:badEventB2i:pair_constraint}. 
To this end we first define $P_H$ as the set of all pairs $xy \in \binom{[n]}{2}$ which satisfy \eqref{eq:badEventB2i:pair_constraint} and contain at least one vertex from $H$. 
Fix $x \in H$. 
By \eqref{eq:badEventB2i:pair_constraint} we know that every $y$ with $xy \in P_H$ satisfies 
\begin{equation}
\label{eq:badEventB2i:pair_constraint:K}
|\Gamma(y) \cap (\Gamma(x) \cap K)| \geq kpn^{-20\epsilon} =: d  .
\end{equation}
Since $x \notin I$ we have $|\Gamma(x) \cap K| \leq p^{-1} n^{10 \epsilon}$ by \eqref{eq:vnotinI:degree_bounded}. Thus, using $\cN_i$ we see that the number of such $y$ is bounded by $16|\Gamma(x) \cap K| / (\epsilon d) \leq kpn^{35\epsilon}$. 
So, considering the number of choices for $x \in H$, and then the number of $y$ with $xy \in P_H$ for each such vertex $x$, we deduce that 
\begin{equation}
\label{eq:badEventB2i:pair_large_degree}
|P_H| \leq p^{-1}n^{-60\epsilon} \cdot kpn^{35\epsilon} = kn^{-25\epsilon}  .
\end{equation} 
Second, we define $P_L$ as the set of all pairs $xy \in \binom{[n]}{2}$ which satisfy \eqref{eq:badEventB2i:pair_constraint} and do not contain any vertex from $H \cup I$. 
To bound the size of $P_L$ we define an auxiliary graph $\cG(P_L)$ with vertex set $P_L$.
Two distinct vertices $xy, x'y' \in V(\cG(P_L))$ are joined by an edge if $xy \cap x'y' \neq \emptyset$, i.e., if the corresponding pairs share a vertex. 
As $\cG(P_L)$ has $|P_L|$ vertices, we estimate the size of $P_L$ with the next lemma, whose very simple bound is e.g.\ attained by the complete graph and its complement. 
\begin{lemma}%
\label{lem:graph_nr_vertices}%
Suppose $G$ is a simple graph. Let $\alpha(G)$ denote the size of the largest independent set in $G$ and let $\Delta(G)$ denote the maximum degree of $G$. Then $G$ has at most $\alpha(G) [ 1+\Delta(G) ]$ vertices. 
\end{lemma} 
Recall that every $xy \in P_L$ satisfies \eqref{eq:badEventB2i:pair_constraint}, which in turn implies that \eqref{eq:badEventB2i:pair_constraint:K} holds. 
So, using $\cM_i$ we deduce that $\alpha(\cG(P_L)) \leq 90 k / (\epsilon d) \leq p^{-1} n^{25 \epsilon}$. 
To bound the maximum degree in $\cG(P_L)$ we fix $x \notin H \cup I$ and estimate the number of $y$ with $xy \in P_L$. 
As argued above, such pairs satisfy \eqref{eq:badEventB2i:pair_constraint:K}. 
So, using $\cN_i$ and $|\Gamma(x) \cap K| \leq k p n^{65 \epsilon}$, the number of such $y$ is bounded by $16|\Gamma(x) \cap K| / (\epsilon d) \leq n^{90 \epsilon}$, which in turn implies $\Delta(\cG(P_L)) \leq 2n^{90 \epsilon}$. 
Now, Lemma~\ref{lem:graph_nr_vertices} together with the above bounds for $\alpha(\cG(P_L))$ and $\Delta(\cG(P_L))$ yields 
\begin{equation}
\label{eq:badEventB2i:pair_low_degree}
|P_L| \leq p^{-1} n^{25 \epsilon} \big[1 + 2n^{90 \epsilon}\big] \leq p^{-1} n^{120 \epsilon}  . 
\end{equation} 
Putting things together, using \eqref{eq:badEventB2i:pair_large_degree} and \eqref{eq:badEventB2i:pair_low_degree} the number of $xy \in \binom{[n]}{2}$ satisfying \eqref{eq:badEventB2i:pair_constraint} is bounded by
\[
|P_H| + |P_L| \leq kn^{-25\epsilon} + p^{-1} n^{120 \epsilon} \leq kn^{-20\epsilon}  ,
\]
where we used \eqref{eq:K4-constants:Wepsmu} for the last inequality. 
Therefore $\neg\badEC[2,i]$ holds.

Finally, we show that $\neg\badEC[3,i]$ holds. 
Recall that $\Xi_{\Sigma^*}(i)$ contains all quadruples $(u,v,w,z) \in A \times B \times C \times [n]$ which are as in Figure~\ref{fig:quadruple:Xi}, i.e., with $z \notin \{u,v,w\}$ and $\{uw,zu,zv,zw\} \subseteq E(i)$. 
Since $z$ has neighbours in $A$ and $B$, using \eqref{eq:vinI:neighbourhood} we see that $z \notin I$. 
Fix a pair $xy \in \binom{[n]}{2}$. 
Roughly speaking, in the following we bound the number of quadruples in $\Xi_{\Sigma^*}(i)$ which contain $xy$. 
First we count the number of quadruples $(u,v,w,z) \in \Xi_{\Sigma^*}(i)$ with $xy=uw$. 
Given $uw$, by the codegree bound there are at most $(\log n) n p^2$ choices for $z \in (\Gamma(u) \cap \Gamma(w)) \setminus I$ and by \eqref{eq:vnotinI:degree_bounded} we have at most $p^{-1} n^{10 \epsilon}$ possibilities for $v \in \Gamma(z) \cap B$. 
To sum up, there are at most $(\log n) n p^2 \cdot p^{-1} n^{10 \epsilon} \leq k n^{15 \epsilon}$ quadruples in $\Xi_{\Sigma^*}(i)$ with $xy=uw$. 
The remaining cases, where $xy$ equals to one of $zu, zv, zw$ are similar: for every quadruple $(u,v,w,z) \in \Xi_{\Sigma^*}(i)$ that contains $xy$, we need to pick two vertices $a$ and $b$, where in each case $a,b$ are two of $u,v,w$. 
Applying the estimate \eqref{eq:vnotinI:degree_bounded} to bound the number of possibilities for $a$ or $b$ as appropriate, and using the codegree to bound the number of choices for the other, in each case there are again at most $(\log n) n p^2 \cdot p^{-1} n^{10 \epsilon} \leq k n^{15 \epsilon}$ quadruples in $\Xi_{\Sigma^*}(i)$ that contain $xy$. 
Putting things together, in $G(i)$ every pair $xy \in \binom{[n]}{2}$ satisfies $xy \in \{uw,zy,zv,zw\}$ for at most 
\[ 4k n^{15 \epsilon} = o(k^2pn^{-15 \epsilon}) \] 
quadruples $(u,v,w,z) \in \Xi_{\Sigma^*}(i)$, which establishes $\neg\badEC[3,i]$.

\subsection{`Few' partial triples are ignored for $\Sigma^*$}  
\label{sec:good_configuration:few_ignored} 
In this section we show that $\partCs \setminus \partOCs$ is small. 
Let $I_{2,\Sigma^*}(i)$ and $I_{3,\Sigma^*}(i)$  contain all $(u,v,w) \in \partCs$ with $uv \in C(i)$ that were ignored in \emph{any} of the previous steps $0 \leq i' \leq i$ because of (I2) or (I3). 
Thus, since every $(u,v,w) \in \partCs \setminus \partOCs$ was ignored in one of the first $i$ steps, we obtain 
\begin{equation}
\label{eq:good_configuration:ignored_pairs_basic_bound}
|\partCs \setminus \partOCs| \leq |I_{2,\Sigma^*}(i)| + |I_{3,\Sigma^*}(i)|  . 
\end{equation} 
Recall that for every $u \in A$ and $v \in B$ there is at most one triple in $\partCs$ which contains $uv$. 
We claim that for every such pair $uv$ there is in fact at most one $w \in C$ such that $(u,v,w) \in \bigcup_{i' \leq i} \partCs[i']$. 
Indeed, when the pair $uv$ first appears in a partial triple $(u,v,w)$ the pair $uv$ must be open, then no other triple containing $uv$ is added until $(u,v,w)$ is removed, at which point $uv$ is definitely closed, so no other triples $(u,v,w')$ can be added in later steps.

Observe that for every $(u,v,w) \in I_{2,\Sigma^*}(i)$ there exists $i' <i$ with $e_{i'+1}=xy$ and $e_{i'+1} \cap uv = \emptyset$ such that in $G(i')\subseteq G(i)$ we have $u,v \in \Gamma(x) \cap \Gamma(y)$ and \eqref{eq:badEventB2i:pair_constraint}. 
By the findings of Section~\ref{sec:good_configuration} we furthermore know that $\neg\badEC[2,i]$ holds. 
So the number of pairs $xy \in \binom{[n]}{2}$ which satisfy \eqref{eq:badEventB2i:pair_constraint} in $G(i)$ is bounded by $kn^{-20\epsilon}$. 
Furthermore, as argued above, for every triple $(u,v,w) \in I_{2,\Sigma^*}(i)$ the pair $uv$ uniquely determines the third vertex $w$.  
So, considering the number of choices for $xy$ satisfying \eqref{eq:badEventB2i:pair_constraint}, and then the number of $u,v \in \Gamma(x) \cap \Gamma(y)$ for each such pair $xy$, using the codegree bound $(\log n)np^2$ we deduce that 
\begin{equation}
\label{eq:good_configuration:ignored_R1}
|I_{2,\Sigma^*}(i)| \leq kn^{-20\epsilon}  \cdot \big[ (\log n)np^2 \big]^2 \leq k^{3}p^2n^{-15\epsilon}  .
\end{equation}

\begin{figure}[tp]
	\centering
  \setlength{\unitlength}{1bp}%
  \begin{picture}(114.54, 87.81)(0,0)
  \put(0,0){\includegraphics{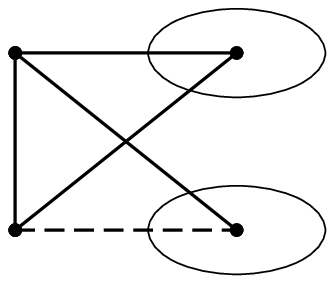}}
  \put(99.11,67.16){\fontsize{11.38}{13.66}\selectfont \makebox[0pt]{$A$}}
  \put(77.75,10.11){\fontsize{11.38}{13.66}\selectfont \makebox[0pt]{$x=v$}}
  \put(99.11,14.44){\fontsize{11.38}{13.66}\selectfont \makebox[0pt]{$B$}}
  \put(80.22,72.87){\fontsize{11.38}{13.66}\selectfont \makebox[0pt]{$u$}}
  \put(14.17,72.87){\fontsize{11.38}{13.66}\selectfont \makebox[0pt]{$z$}}
  \put(14.17,10.11){\fontsize{11.38}{13.66}\selectfont \makebox[0pt]{$y$}}
  \end{picture}%
	\hspace{5.0em}
  \setlength{\unitlength}{1bp}%
  \begin{picture}(114.54, 87.81)(0,0)
  \put(0,0){\includegraphics{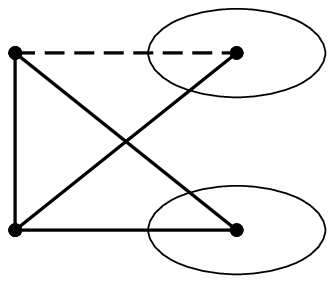}}
  \put(99.11,67.16){\fontsize{11.38}{13.66}\selectfont \makebox[0pt]{$A$}}
  \put(80.50,10.11){\fontsize{11.38}{13.66}\selectfont \makebox[0pt]{$v$}}
  \put(99.11,14.44){\fontsize{11.38}{13.66}\selectfont \makebox[0pt]{$B$}}
  \put(78.0,72.87){\fontsize{11.38}{13.66}\selectfont \makebox[0pt]{$x=u$}}
  \put(14.17,72.87){\fontsize{11.38}{13.66}\selectfont \makebox[0pt]{$y$}}
  \put(14.17,10.11){\fontsize{11.38}{13.66}\selectfont \makebox[0pt]{$z$}}
  \end{picture}%
	\caption{\label{fig:pairs_R3} Pairs $uv$ with $u \in A$ and $v \in B$ such that $e_{i+1}=xy \in C_{uv}(i)$ and $xy \cap uv = x$. Solid lines represent edges, and the dashed line corresponds to the next edge $e_{i+1}=xy$. The vertices $y$ and $z$ may also be in $A \cup B$.}
\end{figure}

Turning to $|I_{3,\Sigma^*}(i)|$, let $H$ contain all vertices $y \in [n] \setminus I$ with $|\Gamma(y) \cap K| \geq p^{-1}n^{-15 \epsilon}$. 
Using $\cN_i$ we infer that, say, $|H| \leq kp n^{20 \epsilon}$. 
Observe that for every $(u,v,w) \in I_{3,\Sigma^*}(i)$ there exists $i' <i$ with $e_{i'+1}=xy \in C_{uv}(i')$ and $e_{i'+1} \cap uv = x$. Hence in $G(i')$ there exists $z \in \Gamma(x) \cap \Gamma(y)$ such that $\{u,v\} \setminus \{x\} \subseteq \Gamma(y) \cap \Gamma(z) \cap K$. 
Similarly as for $I_{2,\Sigma^*}(i)$, the pair $uv$ uniquely determines the third vertex $w$ for every triple  $(u,v,w) \in I_{3,\Sigma^*}(i)$, and therefore it suffices to bound the number of pairs $uv$ with the above properties, cf.\ Figure~\ref{fig:pairs_R3}. 
Using that the triple was not removed due to (R3a) and (R3b), we deduce two additional properties. On the one hand $|\Gamma(y) \cap K| \geq p^{-1}n^{-15 \epsilon}$ holds, which implies $y \in H \cup I$, and, on the other hand, in $G(i')$ all $z \in \Gamma(x) \cap \Gamma(y)$ with $\{u,v\}\setminus\{x\} \subseteq \Gamma(y) \cap \Gamma(z) \cap K$ satisfy 
\begin{equation}
\label{eq:good_configuration:pair_constraint:K}
|\Gamma(z) \cap (\Gamma(y) \cap K)| \geq kpn^{-20\epsilon} =: d . 
\end{equation} 
Now, since $y$ has neighbours in $A$ and $B$, namely $u \in A$ and $v \in B$, using \eqref{eq:vinI:neighbourhood} we deduce $y \not\in I$, which in turn implies $y \in H$. 
With this in mind, we define $\Psi_{\Sigma^*}(i)$ as the set of all quadruples $(u',v',y,z) \in K^2 \times H \times [n]$ for which $z \in \Gamma(y)$, $\{u',v'\} \subseteq \Gamma(y) \cap \Gamma(z) \cap K$ and \eqref{eq:good_configuration:pair_constraint:K} holds. 
The above discussion yields 
\begin{equation}
\label{eq:good_configuration:ignored_R2_0}
|I_{3,\Sigma^*}(i)| \leq |\Psi_{\Sigma^*}(i)|  . 
\end{equation}
Since $y \in H$ satisfies $y \notin I$, by \eqref{eq:vnotinI:degree_bounded} we have $|\Gamma(y) \cap K| \leq p^{-1} n^{10 \epsilon}$. 
So, similar as in Section~\ref{sec:good_configuration}, using $\cN_i$ we see that for every $y \in H$ there are at most $16|\Gamma(y) \cap K|/(\epsilon d) \leq k p n^{35 \epsilon}$ vertices $z$ satisfying \eqref{eq:good_configuration:pair_constraint:K}. 
Furthermore, given $y$ and $z$, the number of choices for $u',v' \in \Gamma(y) \cap \Gamma(z) \cap K \subseteq \Gamma(y) \cap \Gamma(z)$ are each bounded by $(\log n)np^2$.  
Putting things together, using \eqref{eq:good_configuration:ignored_R2_0} we deduce that
\begin{equation}
\label{eq:good_configuration:ignored_R2}
|I_{3,\Sigma^*}(i)| \leq |\Psi_{\Sigma^*}(i)| \leq |H| \cdot kp n^{35 \epsilon} \cdot \big[ (\log n)np^2 \big]^2 \leq k^{2}pn^{60\epsilon}  ,
\end{equation}
where we used $|H| \leq kp n^{20 \epsilon}$ for the last inequality.

Finally, plugging using \eqref{eq:good_configuration:ignored_R1} and \eqref{eq:good_configuration:ignored_R2} into \eqref{eq:good_configuration:ignored_pairs_basic_bound}, using \eqref{eq:K4-constants:Wepsmu} we see that 
\begin{equation*}
\label{eq:good_configuration:ignored_quadruples_bound}
|\partCs \setminus \partOCs| \leq k^{3}p^2n^{-15\epsilon} + k^{2}pn^{60\epsilon} \leq k^{3}p^2n^{-10\epsilon}  ,
\end{equation*}
as claimed. 
This completes the proof of Lemma~\ref{lem:dem:config}. \qed

\bigskip{\bf Acknowledgements.} 
I am grateful to my supervisor Oliver Riordan for many helpful discussions and want to thank him for a very careful reading of an earlier version of this paper, leading to several simplifications in the proof. 
Furthermore, I would also like to thank the referees for many detailed comments.

\small\begin{spacing}{0.4}
\bibliographystyle{plain}

\end{spacing}
\normalsize

\begin{appendix}
\section{Appendix}\label{Apx}

\subsection{Differential equation method}
\label{apx:dem:general}
In this section we formulate the improved version of Lemma~$7.3$ in~\cite{BohmanKeevash2010H}, which can be obtained by adapting the ideas/modifications we used in the proof of Lemma~\ref{lem:dem} back to the original setup. 
Intuitively, there are different `types' $j \in \cV$ of random variables, where $\sigma \in \cI_{j}$ denotes particular `instances', which can e.g.\ take into account different `positions' in a graph. 
\begin{lemma}[`Differential Equation Method']%
\label{lem:dem:general}%
Suppose that $m=m(n)$ and $s=s(n)$ are positive parameters. 
Let $\cV=\cV(n)$ be a set, and $\{\cI_{j}\}_{j \in \cV}$ be a family of sets, where $\cI_{j}=\cI_{j}(n)$. 
For every $0 \leq i \leq m$ set $t=t(i):=i/s$. 
Suppose we have a filtration $\cF_0 \subseteq \cF_1 \subseteq \cdots$ and random variables $X_{\sigma}(i)$ and $Y^{\pm}_{\sigma}(i)$ which satisfy the following conditions. 
Assume that for all $j \in \cV$ and $\sigma \in \cI_{j}$ the random variables $X_{\sigma}(i)$ are non-negative and $\cF_i$-measurable for all $0 \leq i \leq m$, and that for all $0 \leq i < m$ the random variables $Y^{\pm}_{\sigma}(i)$ are non-negative, $\cF_{i+1}$-measurable and satisfy
\begin{equation}
\label{eq:lem:dem:general:rv_relation}
X_{\sigma}(i+1) - X_{\sigma}(i) =  Y^{+}_{\sigma}(i) - Y^{-}_{\sigma}(i)  .
\end{equation}
In addition, suppose that for each $j \in \cV$ and $\sigma \in \cI_{j}$ we have positive parameters $u_{\sigma}=u_{\sigma}(n)$, $\lambda_{\sigma}=\lambda_{\sigma}(n)$, $\beta_{\sigma}=\beta_{\sigma}(n)$, $\tau_{\sigma}=\tau_{\sigma}(n)$, $s_{\sigma} = s_{\sigma}(n)$ and $S_{\sigma}=S_{\sigma}(n)$, as well as functions $x_{\sigma}(t)$ and $f_{\sigma}(t)$ that are smooth and non-negative for $t \geq 0$.  
For all $0 \leq i^* \leq m$, let $\cG_{i^*}$ denote the event that for all $0 \leq i \leq i^*$, $j \in \cV$ and $\sigma \in \cI_{j}$, we have 
\begin{equation}
\label{eq:lem:dem:general:parameter_trajectory}
X_{\sigma}(i) = \left(x_{\sigma}(t) \pm \frac{f_{\sigma}(t)}{s_{\sigma}} \right) S_{\sigma}  .
\end{equation}
Moreover, assume that we have an event $\cH_i \in \cF_i$ for all $0 \leq i \leq m$ with $\cH_{i+1} \subseteq \cH_{i}$ for all $0 \leq i < m$. 
Finally, suppose that for $n$ large enough the following conditions hold: 
\begin{enumerate}
\item(Trend hypothesis)
For all $0 \leq i < m$, $j \in \cV$ and $\sigma \in \cI_{j}$, whenever $\cE_i \cap \cH_i$ holds we have 
\begin{equation}
\label{eq:lem:dem:general:parameter_martingale_property}
\EE\big[Y^{\pm_1}_{\sigma}(i) \mid \cF_i \big] = \left( y^{\pm_1}_{\sigma}(t) \pm  \frac{h_{\sigma}(t)}{s_{\sigma}} \right)  \frac{S_{\sigma}}{s}  ,  
\end{equation}
where $y_{\sigma}^{\pm}(t)$ and $h_{\sigma}(t)$ are smooth non-negative functions such that
\begin{equation}
\label{eq:lem:dem:general:derivative}
x'_{\sigma}(t) = y^+_{\sigma}(t) - y^-_{\sigma}(t) \qquad \text{ and } \qquad f_{\sigma}(t) \geq 2\int_{0}^{t} h_{\sigma}(\tau) \ d\tau + \beta_{\sigma}  .
\end{equation}

\item(Boundedness hypothesis) 
For all $0 \leq i < m$, $j \in \cV$ and $\sigma \in \cI_{j}$, whenever $\cE_i \cap \cH_i$ holds we have
\begin{equation}
\label{eq:lem:dem:general:parameter_max_change}
Y^{\pm}_{\sigma}(i)  \leq \frac{\beta_{\sigma}^2}{s_{\sigma}^2 \lambda_{\sigma} \tau_{\sigma}} \cdot  \frac{S_{\sigma}}{u_{\sigma}}  . 
\end{equation}

\item(Initial conditions) 
For all $j \in \cV$ and $\sigma \in \cI_{j}$ we have 
\begin{equation}
\label{eq:lem:dem:general:initial_condition}
X_{\sigma}(0) = \left(x_{\sigma}(0) \pm \frac{\beta_{\sigma}}{3s_{\sigma}} \right) S_{\sigma}  .
\end{equation}

\item(Bounded number of variables) For all $j \in \cV$  and $\sigma \in \cI_{j}$ we have
\begin{equation}
\label{eq:lem:dem:general:bounded_parameters}
\max\{|\cV|, |\cI_j|\} \leq e^{u_{\sigma}}  . 
\end{equation}

\item(High probability event) The event $\cH_i$ satisfies
\begin{equation}
\label{eq:lem:dem:general:H_i}
\PP[\exists i \leq m: \ \cG_i \cap \neg\cH_i] = o(1)  .
\end{equation}

\item(Additional technical assumptions) 
For all $j \in \cV$ and $\sigma \in \cI_{j}$ we have $u_{\sigma} = \omega(1)$ as well as 
\begin{gather}
\label{eq:lem:dem:general:technical_assumptions:sm}
s \geq \max\{15 u_{\sigma} \tau_{\sigma} (s_{\sigma}  \lambda_{\sigma}/\beta_{\sigma})^2, 9s_{\sigma} \lambda_{\sigma}/\beta_{\sigma}\}  , \qquad s/(18 s_{\sigma} \lambda_{\sigma}/\beta_{\sigma}) < m \leq s \cdot \tau_{\sigma}/1944  ,\\ 
\label{eq:lem:dem:general:technical_assumptions:xy}
\sup_{0 \leq t \leq m/s} y^{\pm }_{\sigma}(t) \leq \lambda_{\sigma}  , \qquad 
\int_0^{m/s} |x''_{\sigma}(t)| \ dt \leq \lambda_{\sigma}  , \\ 
\label{eq:lem:dem:general:technical_assumptions:h}
h_{\sigma}(0) \leq s_{\sigma} \lambda_{\sigma}  \qquad \text{ and }  \qquad  \int_0^{m/s} |h'_{\sigma}(t)| \ dt \leq s_{\sigma} \lambda_{\sigma}  . 
\end{gather}
\end{enumerate}
Then  $\cG_m \cap \cH_m$ holds with high probability. 
\end{lemma} 
Compared to Lemma~$7.3$ in~\cite{BohmanKeevash2010H}, one important advantage of Lemma~\ref{lem:dem:general} is that we state the approximation error in a simpler form and allow for more freedom in choosing the corresponding error functions; this should make it easier to apply our variant in other contexts. 
If possible, it is often convenient to choose the same parametrization for all $\sigma \in \cI_{j}$, e.g.\ $x_{\sigma}(t) = x_{j}(t)$, since they typically correspond to different `instances' of the same type of random variables. 
We point out that by using this simplification, then choosing $s_{j} \geq n^{\epsilon}$ as well as $u_{j} = 2 k_{j} \log n$, and afterwards setting $\tau_j = 1944n^{\epsilon/2}$, $\beta_{j}^{-1} := \lambda_{j} := n^{\epsilon/7}$, $h_{j}(t) := (e_{j} \cdot x_{j} + \gamma_{j})'(t)/4$ and $f_{j}(t) := e_{j}(t)x_{j}(t)-\theta_{j}(t)e_{j}(t)/s_{j}+\theta_{j}(t)$, this not only implies Lemma~$7.3$ in~\cite{BohmanKeevash2010H}, but also weakens certain assumptions significantly. For example, in the additional technical assumptions we relax $y^{\pm }_{j}(t)=O(1)$ to $y^{\pm }_{j}(t) \leq n^{\epsilon/7}$, their $c$ from $\Omega(1)$ to $\Omega(n^{-\epsilon/7})$, and also weaken the lower bound on $m$ from $m > s$ to, say, $m \geq s n^{-\epsilon}$. Furthermore, we e.g.\ improve and simplify the initial conditions by allowing for a small error in the initial value $X_{\sigma}(0)$ and removing the requirement that $e_{j}(0)=\gamma_{j}(0)=0$. 
At first sight our assumption that $\cH_i$ satisfies \eqref{eq:lem:dem:general:H_i} seems to be more restrictive, however, due to an oversight in the proof given by Bohman and Keevash in~\cite{BohmanKeevash2010H}, their Lemma~$7.3$ also needs this additional assumption, which of course holds in their application.  
Another new ingredient is the introduction of the parameters $\lambda_{\sigma}$, $\beta_{\sigma}$ and $\tau_{\sigma}$, which allows for a trade-off between the approximation error, the boundedness hypothesis and the additional technical assumptions. 
For example, as already mentioned in Section~\ref{sec:dem:variant}, in certain applications this might allow for larger one-step changes than Lemma~$7.3$ in~\cite{BohmanKeevash2010H}. 
Finally, as noted in~\cite{BohmanKeevash2010H}, compared to Wormald's formulation of the differential equation method~\cite{Wormald1995DEM,Wormald1999DEM}, if applicable, Lemma~\ref{lem:dem:general} has the advantage that in certain applications 
much weaker estimates on the one-step changes suffice.

\subsection{Proof of Lemma~\ref{lem:dem}}
\label{apx:proof:lem:dem}
We omitted some details in the proof of Lemma~\ref{lem:dem}, since they were very similar to the corresponding calculations in proof of Lemma~$7.3$ in~\cite{BohmanKeevash2010H}. 
In this section we give the missing calculations, and keep the notation and assumptions of Lemma~\ref{lem:dem}.

\textbf{Using the Euler-Maclaurin summation formula.} 
In the following we prove the estimates \eqref{eq:lem:dem:sum:x} and \eqref{eq:lem:dem:sum:h:upper}. 
To this end we use the Euler-Maclaurin summation formula, which relates the integral $\int_{a}^{b} f(x) dx$ with the sum $\sum_{k=a}^{b-1}f(k)$. 
The following variant is implicit in~\cite{Apostol1999}. 
\begin{lemma}%
\label{lem:EulerMaclaurin}%
{\normalfont\cite{Apostol1999}} 
Let $a < b$ be integers. Then for any function $f$ with a continuous derivative on the interval $[a,b]$ we have
\begin{equation}
\label{eq:euler-macluarin:bound}
\left| \int_{a}^{b} f(x) dx - \sum_{k=a}^{b-1}f(k) \right|  \leq  \int_{a}^{b} |f'(x)| dx  .
\end{equation}
\end{lemma} 
We start by proving \eqref{eq:lem:dem:sum:x}. Elementary calculus shows
\begin{equation}  
\label{eq:lem:dem:integral:x}
\int_{0}^{i^*} x'_{\sigma}\big(t(i)\big) \ di  =  s \int_{0}^{t^*} x'_{\sigma}(t) \ dt = s \bigl[ x_{\sigma}(t^*) - x_{\sigma}(0) \bigr]  .
\end{equation}  
Furthermore, using the Euler-Maclaurin summation formula \eqref{eq:euler-macluarin:bound} and $t(i)=i/s$, we see that
\begin{equation}  
\label{eq:lem:dem:em:x}
\left|\int_{0}^{i^*} x'_{\sigma}\big(t(i)\big) \ di - \sum_{i=0}^{i^*-1} x'_{\sigma}\big(t(i)\big)\right| \leq \frac{1}{s}\int_{0}^{i^*} |x''_{\sigma}\big(t(i)\big)| \ di = \int_{0}^{t^*} |x''_{\sigma}(t)| \ dt  . 
\end{equation}  
Now, combining \eqref{eq:lem:dem:integral:x} and \eqref{eq:lem:dem:em:x} with the additional technical assumptions \eqref{eq:lem:dem:technical_assumptions:xy}, we deduce 
\[
\left|x_{\sigma}(t^*) - x_{\sigma}(0) - \frac{1}{s} \sum_{i=0}^{i^*-1} x'_{\sigma}\big(t(i)\big)\right| \leq \frac{1}{s} \int_{0}^{t^*} |x''_{\sigma}(t)| \ dt \leq \frac{\lambda_{\sigma}}{s}  ,
\]
which in turn implies \eqref{eq:lem:dem:sum:x}. 
Using \eqref{eq:lem:dem:technical_assumptions:h}, with a similar calculation we obtain 
\[
\left|\int_{0}^{t^*}h_{\sigma}(t)\ dt - \frac{1}{s}\sum_{i=0}^{i^*-1} h_{\sigma}\big(t(i)\big)\right| \leq \frac{1}{s} \int_{0}^{t^*} |h'_{\sigma}(t)| \ dt \leq \frac{s_{\sigma}\lambda_{\sigma}}{s}  ,
\]
from which \eqref{eq:lem:dem:sum:h:upper} readily follows.

\textbf{Bounded super-/submartingales.}  
We show that $Z^{\pm -}_{\sigma}(i)$ and $Z^{\pm +}_{\sigma}(i)$ are $(M_{\sigma},N_{\sigma})$-bounded super-/submartingales. 
For bounding the maximum change $Z^{\pm_1\pm_2}_{\sigma}(i+1)-Z^{\pm_1\pm_2}_{\sigma}(i)$ we may assume that $\cE_i \cap \neg\badEL \cap \cH_i$ holds (otherwise the difference is by definition equal to $0$). 
In that case 
\begin{equation}
\label{eq:lem:dem:z-difference}
Z^{\pm_1\pm_2}_{\sigma}(i+1) - Z^{\pm_1\pm_2}_{\sigma}(i) \; = \; Y^{\pm_1\pm_2}_{\sigma}(i) \;=\; Y^{\pm_1}_{\sigma}(i)  - \left(y^{\pm_1}_{\sigma}(t) \mp_2 \frac{h_{\sigma}(t)}{s_{\sigma}}\right) \frac{S_{\sigma}}{s}  .
\end{equation}
Now, using the boundedness hypothesis \eqref{eq:lem:dem:parameter_max_change} as well as $y^{\pm}_{\sigma}(t) \geq 0$ and \eqref{eq:lem:dem:technical_assumptions:h:function}, i.e., $h_{\sigma}(t) \leq 2s_{\sigma}\lambda_{\sigma}$, we see that \eqref{eq:lem:dem:z-difference} is bounded from above by 
\[ \frac{\beta_{\sigma}^2}{s_{\sigma}^2 \lambda_{\sigma} \tau_{\sigma}} \frac{S_{\sigma}}{u_{\sigma}} + \frac{h_{\sigma}(t)}{s_{\sigma}} \frac{S_{\sigma}}{s} \leq \frac{N_{\sigma}}{2} + \frac{2 \lambda_{\sigma} S_{\sigma}}{s} \leq N_{\sigma}  , \]
where we used \eqref{eq:lem:dem:technical_assumptions:sm}, i.e., $s \geq 15 u_{\sigma} \tau_{\sigma} (s_{\sigma}  \lambda_{\sigma}/\beta_{\sigma})^2$,  and \eqref{eq:lem:dem:def:rv:ZMN} for the last inequality. 
Similarly, using $Y^{\pm}_{\sigma}(i),h_{\sigma}(t) \geq 0$ and \eqref{eq:lem:dem:technical_assumptions:xy}, i.e., $y^{\pm}_{\sigma}(t) \leq \lambda_{\sigma}$, we see that \eqref{eq:lem:dem:z-difference} is bounded from below by
\[ -\left(\lambda_{\sigma} + \frac{h_{\sigma}(t)}{s_{\sigma}}\right)\frac{S_{\sigma}}{s} \geq - \frac{3\lambda_{\sigma} S_{\sigma}}{s} = - M_{\sigma}  . \]
For checking the super-/submartingale property we may again assume that $\cE_i \cap \neg\badEL \cap \cH_i$ holds (otherwise the value of $Z^{\pm_1\pm_2}_{\sigma}(i)$ remains unchanged). 
Now by combining \eqref{eq:lem:dem:z-difference} with the trend hypothesis \eqref{eq:lem:dem:parameter_martingale_property}, it is easy to see that $Z^{\pm-}_{\sigma}(i)$ is a supermartingale and $Z^{\pm+}_{\sigma}(i)$ a submartingale.

\textbf{Relating certain variables.}  
We prove that whenever $\cH_{i^*-1} \cap \cE_{i^*-1} \cap \neg\badEL[i^*-1]$ holds we have   
\begin{equation}
\label{eq:lem:dem:ineq:Z}
Z^{+ \pm_2}_{\sigma}(i^*) - Z^{- \mp_2}_{\sigma}(i^*) = X_{\sigma}(i^*)-X_{\sigma}(0) - \frac{1}{s} \sum_{i=0}^{i^*-1} x'_{\sigma}\big(t(i)\big) \cdot S_{\sigma} \pm_2 \frac{1}{s} \sum_{i=0}^{i^*-1} h_{\sigma}\big(t(i)\big) \cdot  \frac{2S_{\sigma}}{s_\sigma}  ,
\end{equation}
which readily implies \eqref{eq:lem:dem:ineq:Z+-Z-+}. 
Using \eqref{eq:lem:dem:rv_relation} and \eqref{eq:lem:dem:derivative} as well as the definition of $Y^{\pm}_{\sigma}(i)$ and $Z^{+ \pm_2}_{\sigma}(i)$, i.e.~\eqref{eq:lem:dem:def:rv:Y} and~\eqref{eq:lem:dem:def:rv:ZMN}, because $\cH_{i^*-1} \cap \cE_{i^*-1} \cap \neg\badEL[i^*-1]$ holds we have 
\begin{equation*}
\begin{split}
X_{\sigma}(i^*)-X_{\sigma}(0) - \sum_{i=0}^{i^*-1} x'_{\sigma}\big(t(i)\big) \cdot \frac{S_{\sigma}}{s} &= \sum_{i=0}^{i^*-1} \left( X_{\sigma}(i+1)-X_{\sigma}(i) - x'_{\sigma}\big(t(i)\big) \frac{S_{\sigma}}{s} \right)\\
& = \sum_{i=0}^{i^*-1} \left(Y^{+}_{\sigma}(i)- y^{+}_{\sigma}\big(t(i)\big) \frac{S_{\sigma}}{s} - Y^{-}_{\sigma}(i) + y^{-}_{\sigma}\big(t(i)\big) \frac{S_{\sigma}}{s} \right)\\
&= \sum_{i=0}^{i^*-1} Y^{+ \pm_2}_{\sigma}(i) - \sum_{i=0}^{i^*-1} Y^{- \mp_2}_{\sigma}(i) \mp_2 2 \sum_{i=0}^{i^*-1} \frac{h_{\sigma}\big(t(i)\big)}{s_\sigma} \cdot \frac{S_{\sigma}}{s}\\
&= Z^{+ \pm_2}_{\sigma}(i^*) - Z^{- \mp_2}_{\sigma}(i^*) \mp_2 \frac{1}{s} \sum_{i=0}^{i^*-1} h_{\sigma}\big(t(i)\big) \cdot \frac{2S_{\sigma}}{s_\sigma}  .
\end{split}
\end{equation*}
Rearranging gives \eqref{eq:lem:dem:ineq:Z}, which, as explained, implies \eqref{eq:lem:dem:ineq:Z+-Z-+}.

\end{appendix}

\end{document}